%% file: Main.tex
\documentclass[journal]{IEEEtran}
\ifCLASSINFOpdf
\else
\fi
\usepackage{graphicx}
\input{defs.tex}

\usepackage{color}
\usepackage{cleveref}
\usepackage{silence}
\usepackage{soul}
\theoremstyle{plain} 

\newtheorem{thm}{Theorem}[section] 
\newcommand{\thistheoremname}{}
\newtheorem{genericthm}[thm]{\thistheoremname}

\newtheorem*{genericthm*}{\thistheoremname}
\newenvironment{namedthm*}[1]
  {\renewcommand{\thistheoremname}{#1}%
   \begin{genericthm*}}
  {\end{genericthm*}}

\usepackage{mathtools, cuted}
\usepackage{relsize}
\usepackage[export]{adjustbox}
\usepackage[dvipsnames]{xcolor}
\usepackage[skins]{tcolorbox}

\definecolor{darkred}{rgb}{0.6,0,0}
\definecolor{darkgreen}{rgb}{0,0.5,0}
\definecolor{darkblue}{rgb}{0,0,0.5}
\definecolor{goldenrod}{rgb}{0.85, 0.65, 0.13}
\definecolor{goldenbrown}{rgb}{0.6, 0.4, 0.08}
\usepackage{tikz,pgfplots}
\pgfplotsset{compat=1.5.1}
\usepgfplotslibrary{groupplots}

\usepackage{multicol}

\newcommand{\cb}{\color{black}}
\newcommand{\Mcb}{\color{black}}
\newcommand{\MRcb}{\color{black}}

\long\def\red#1{\bgroup\color{black}#1\egroup}

\begin{document}
%
\title{A Complex Quasi-Newton Proximal Method for Image Reconstruction in Compressed Sensing MRI}
%
%
%

\author{Tao Hong, Luis Hernandez-Garcia, and Jeffrey A. Fessler, \IEEEmembership{Fellow, IEEE}
\thanks{T. Hong and L. Hernandez-Garcia are with the Department of Radiology, University of Michigan, Ann Arbor, MI 48109, USA (Email: {\{tahong, hernan\}@umich.edu}).}
\thanks{J. Fessler is with the Department of Electrical and Computer Engineering, University of Michigan, Ann Arbor, MI 48109, USA (Email: {fessler@umich.edu}).}
}

%
%

\markboth{}
{Shell \MakeLowercase{\mrmit{et al.}}: Bare Demo of IEEEtran.cls for IEEE Journals}
%


\maketitle
\begin{abstract}
Model-based methods are widely used for reconstruction
in compressed sensing (CS) magnetic resonance imaging (MRI),
using regularizers to describe the images of interest.
The reconstruction process is equivalent to solving a composite optimization problem.
Accelerated proximal methods (APMs) are very popular approaches for such problems.
This paper proposes a complex quasi-Newton proximal method (CQNPM)
for the wavelet and total variation based CS MRI reconstruction.
Compared with APMs,
CQNPM requires fewer iterations to converge
but needs to compute a more challenging proximal mapping
called weighted proximal mapping (WPM).
To make CQNPM more practical,
we propose efficient methods to solve the related WPM.
Numerical experiments {\MRcb on reconstructing non-Cartesian MRI data}
demonstrate the effectiveness and efficiency of CQNPM.
\end{abstract}

%
\IEEEpeerreviewmaketitle

\section{Introduction}

\IEEEPARstart{M}{AGNETIC} resonance imaging (MRI)  scanners acquire samples of the Fourier transform
(known as k-space data) of the image of interest.
However, MRI is slow
since the speed of acquiring k-space data is limited by many constraints,
e.g., hardware, physics, and physiology etc.
Improving the acquisition speed is crucial for many MRI applications.
Lustig et al. \cite{lustig2007sparse} proposed a technique called compressed sensing (CS) MRI
that improves the imaging speed significantly.
CS MRI allows one to get an image of interest from undersampling data
by solving the following composite optimization problem:
\begin{equation}
  \vx^*=\arg\min_{\vx\in\mathbb C^N}
  \underbrace{\frac{1}{2}\|\mA\vx-\vy\|_2^2}_{f(\vx)}
  + \lambda \, h(\vx),
  \label{eq:mainMin}
\end{equation}
where $\mA\in\mathbb{C}^{ML\times N}$ denotes the forward model
describing a mapping from the latent image $\vx$
to the acquired k-space data $\vy\in \mathbb C^{ML}$,
$h(\vx)$ is the regularizer that provides some prior assumptions about $\vx$,
$L\geq 1$ denotes the number of coils,
and $\lambda>0$ is a tradeoff parameter to balance $f(\vx)$ and $h(\vx)$.
 We note that $\mA$ consists of $L$ different
submatrices
$\mA_l = \mPhi\mF\mS_l \in \mathbb C^{M\times N}$ for $l=1,2,\ldots,L$,
where $\mPhi$ denotes the downsampling mask,
$\mF$ represents the nonuniform fast Fourier transform that depends on the sampling trajectory,
and $\mS_l$ is a diagonal matrix involving the sensitivity map for the $l$th coil which differs for each scan. 

Sparsity plays a key role in the success of CS MRI.
In general, MR images are not sparse but they can be sparsely represented under some transforms,
e.g., total variation (TV) \cite{lustig2007sparse}, wavelets \cite{guerquin2011fast},
and transform-learning \cite{ravishankar2011mr} etc.
Recently, more advanced priors or frameworks {\cb were} introduced for CS MRI reconstruction,
such as low-rank \cite{dong2014compressive},
plug and play \cite{venkatakrishnan2013plug,ahmad2020plug},
model-based deep learning \cite{aggarwal2018modl},
score-based generative models \cite{song2021solving},
to name a few.
{\cb Although deep learning based reconstruction methods
have shown better performance than classical priors like TV and wavelet
when trained with sufficient data,
Gu et al. \cite{gu2022revisiting} recently found that
suitably trained wavelet regularizers can also achieve comparable performance,
demonstrating the power of the classical regularizers.
Following Lustig et al. work in \cite{lustig2007sparse},
we consider both wavelet and TV regularizers for CS MRI reconstruction, i.e.,}
we address the following composite minimization problem
for image reconstruction in CS MRI:
\begin{equation}
  \vx^*=\arg\min_{\vx\in\mathbb C^N}
  \frac{1}{2}\|\mA\vx-\vy\|_2^2
  +\lambda \big[\alpha\|\mT\vx\|_1
  +(1-\alpha)\mrm{TV}(\vx)\big]
  ,\label{eq:mainMinWavTV}
\end{equation}
where $\mT$ and $\|\cdot\|_1$ denote a general wavelet transform and $\ell_1$ norm,
$\mrm{TV}(\cdot)$ represents the TV function
(see definition in \Cref{sec:pre:definationTV}),
and $\alpha\in[0,1]$ is used to balance the wavelet and TV terms.
For $\alpha=1$ (respectively, $\alpha=0$),
\eqref{eq:mainMinWavTV} becomes the wavelet (respectively, TV) based CS MRI reconstruction.
Since $\ell_1$ and TV functions are nonsmooth,
accelerated proximal methods (APMs) \cite{parikh2014proximal},
which have the optimal convergence rate $\mathcal O(1/k^2)$
where $k$ is the number of iterations,
are very popular algorithms for \eqref{eq:mainMinWavTV}.
In \cite{beck2009fast},
Beck et al. proposed a fast iterative shrinkage-thresholding algorithm (FISTA)
(a specific type of APMs)
for wavelet-based image reconstruction
and showed a closed-form solution for the related proximal mapping \cite{parikh2014proximal}.
Beck et al. \cite{beck2009fastTV} extended FISTA
to solve TV-based image reconstruction
and suggested a fast dual gradient descent method to compute the proximal mapping.
{\cb Primal-dual methods \cite{chambolle2011first} are also appealing methods
for composite problems.
The work in \cite{condat2013primal} showed that primal-dual methods
can also achieve the optimal convergence rate
and showed their connection to proximal methods.
However, primal-dual methods have to tune parameters
that affect the practical convergence rates and such tuning is nontrivial.
For a review of different variants of primal-dual methods,
see \cite{komodakis2015playing}.} 
For more optimization methods and the use of different regularizers
for reconstruction in CS MRI,
see \cite{fessler2020optimization}.

Modern MR images are typically acquired using multiple receiver coils and non-Cartesian trajectories,
resulting in an expensive forward process from the image to the k-space domains and ill-conditioned or under-determined $\mA$ \cite{fessler2020optimization}.
An ill-conditioned $\mA$ can lead to slow reconstruction \cite{iyer2022accelerating}.
To accelerate the recovery process, some preconditioning techniques have been introduced.
In \cite{ong2019accelerating}, Ong et al. proposed a diagonal matrix $\tilde{\mD}$ as a preconditioner
such that they solved the following problem instead of \eqref{eq:mainMin}:
\begin{equation}
  \vx^*=\arg\min_{\vx\in\mathbb C^N}
  \frac{1}{2} \|\tilde{\mD}^{\frac{1}{2}}\left(\mA\vx-\vy\right)\|_2^2
  +\lambda h(\vx)
  .\label{eq:mainMin:DiagPre}
\end{equation}
Recently, Iyer et al. \cite{iyer2022accelerating} developed more effective polynomial preconditioners
than $\tilde{\mD}$,
based on Chebyshev polynomials.
Although \cite{iyer2022accelerating} showed promising results for practical reconstruction,
adding such a preconditioner changes the incoherence of $\mA$, which breaks the original theoretical guarantee.
For $\alpha\in(0,1)$, both wavelet and TV are used as regularizers.
When we have two nonsmooth terms,
the alternating direction method of multipliers (ADMM) \cite{boyd2011distributed}
is one of the appealing approaches.
However, ADMM only provides linear convergence rate $\mathcal O(1/k)$ \cite{he20121}
and the computation in each iteration is high because we need to solve a least square problem.
In \cite{ramani2010parallel,weller2013augmented,koolstra2019accelerating},
the authors proposed several preconditioning methods to solve the least square problem quickly,
which reduces the computation time of the whole reconstruction significantly. 

Similar to the quasi-Newton methods for smooth minimization problems \cite{jorge2006numerical},
the authors in \cite{lee2014proximal,becker2019quasi} developed quasi-Newton proximal methods (QNPMs)
{\cb that use second-order information}
for solving composite problems
when $\vx\in\mathbb R^N$.
Compared with APMs, QNPMs need fewer iterations to converge
which is appealing for problems when computing the gradient $\nabla f(\vx)$ is expensive.
Indeed, the authors in \cite{hong2020solving,kadu2020high,ge2020proximal}
applied QNPMs to solving the RED model and the TV based inverse-scattering and X-ray reconstruction
and observed faster convergence than APMs.
However, QNPMs require computing a weighted proximal mapping (WPM),
defined in \eqref{eq:def:WeightedProximal},
that needs more computation than computing proximal mapping in APMs.
So often QNPMs are impractical for many real applications.
To compute the WPM, Kadu et al. \cite{kadu2020high} applied primal-dual methods.
Alternatively, Ge et al. \cite{ge2020proximal} treated the WPM as a TV based image deblurring problem and computed the WPM with APMs \cite{beck2009fastTV}.
Those methods require inner and outer (i.e., two layers) iterations to compute the WPM,
making them inefficient.
Similar to QNPMs, the variable metric operator splitting methods (VMOSMs) \cite{chouzenoux2014variable}
introduce new metrics to accelerate the proximal methods.
For a discussion of the differences between QNPMs and VMOSMs,
see the prior work section in \cite{becker2019quasi}.

{ The primary \emph{contribution} of this paper lies in two significant advancements. Firstly, we expand QNPMs to address \eqref{eq:mainMin} for \emph{complex} $\vx$ {\Mcb  (Recall that reconstructed MRI images are inherently complex
\cite{fessler:10:mbi}
and in some applications
the image phase itself is useful,
e.g., high-field MRI \cite{duyn2007high}
and quantitative susceptibility mapping \cite{wang2015quantitative}).
} This is achieved by introducing a symmetric rank-$1$ method in the complex plane to approximate the Hessian matrix of  $f(\vx)$, which we called complex quasi-Newton proximal methods (CQNPMs). Secondly, we propose efficient approaches to compute the WPM. Notably, the computational needs of CQNPMs align closely with the proximal mapping in APMs for wavelet and/or TV-based reconstructions.}
Our numerical experiments on wavelet and TV based CS MRI reconstruction
show that CQNPMs converge faster than APMs in terms of iterations and CPU time,
demonstrating the potential advantage of CQNPMs for practical applications.

The rest of this paper is organized as follows.
\Cref{sec:prelminaries} first defines some notation and then reviews
the formulation of the discretized TV function and the definition of WPM.
\Cref{Sec:CQNPM}
derives our algorithm.
\Cref{sec:numericalExp}
reports numerical experiments on the wavelet and TV based CS MRI reconstruction.
\Cref{Sec:conclusion}
presents some conclusions and future work.

\section{Preliminaries}
\label{sec:prelminaries}
This section first defines some notation that simplifies the following discussion
and then describes the discretized TV functions.
Finally, we define the WPM that generalizes the well-known proximal mapping.

\subsection{Notation}

\begin{enumerate}[\texorpdfstring{$\bullet$}{*}]
\item
Denote by $\mX\in\mathbb C^{I\times J}$ the matrix form of $\vx\in\mathbb C^{N}$ with relation $\vx=\mrm{vec}(\mX)$  and $\mX=\mrm{mat}(\vx)$ where $\mrm{vec}(\cdot)$ denotes a column-stacking operator and $\mrm{mat}(\cdot)$ is an operator to reshape a vector to its matrix form.

\item
The $(i,j)$th (respectively, $n$th) element of a matrix $\mX\in\mathbb C^{I\times J}$
(respectively, vector $\vx\in \mathbb C^N$) is represented as $\mX_{i,j}$ (respectively, $\vx_n$).

\item
$\SetP_1$ denotes the set of matrix-pairs $(\mP,\mQ)$
where $\mP\in \mbb C^{(I-1)\times J}$ and $\mQ\in \mbb C^{I\times (J-1)}$ satisfy
\begin{equation*}
	\begin{array}{rl}
		|\mP_{i,j}|^2+|\mQ_{i,j}|^2\leq 1,&i=1,\cdots,I-1,~\,j=1,\cdots,J-1,\\[6pt]
		|\mP_{i,J}|\leq 1,&i=1,\cdots,I-1,\\[6pt]
		|\mQ_{I,j}|\leq 1,&J=1,\cdots,J-1.
	\end{array}
\end{equation*} 
 
\item
$\SetP_2$ is the set of matrix-pairs $(\mP,\mQ)$
where $\mP\in \mbb C^{(I-1)\times J}$ and $\mQ\in \mbb C^{I\times (J-1)}$
satisfy
$|\mP_{i,j}|\leq 1,\ |\mQ_{i,j}|\leq1, \forall i,j$.

\item
$\SetZ$ is the set of vectors $\vz\in\mbb{C}^N$
such that $|\vz_n|\leq 1,~\forall n$. 

\item
$\SetL:\, \mbb C^{(I-1)\times J}\times \mbb C^{I\times (J-1)}\rightarrow\mbb C^{I\times J}$
denotes a linear operator that satisfies
\begin{equation*}
	{\SetL}(\mP,\mQ)_{i,j} = \mP_{i,j}+\mQ_{i,j}-\mP_{i-1,j}-\mQ_{i,j-1},\forall i,j,
\end{equation*} 
where we assume that $\mP_{0,j}=\mP_{I,j}=\mQ_{i,0}=\mQ_{i,J}=0,\forall i,j$.

\item
The adjoint operator of
$\SetL:\,\mbb C^{I\times J}\rightarrow \mbb C^{(I-1)\times J}\times \mbb C^{I\times (J-1)}$
is 
$$\mathcal L^\SetT(\mX)=(\mP,\mQ),$$
where $\mP\in\mbb C^{(I-1)\times J}$ and $\mQ\in \mbb C^{I\times (J-1)}$
are the matrix pairs that satisfy
\[
	\begin{array}{ll}
		\mP_{i,j}=\mX_{i,j}-\mX_{i+1,j},&i=1,\cdots,I-1,\,j=1,\cdots,J,\\ [6pt]
		\mQ_{i,j}=\mX_{i,j}-\mX_{i,j+1},&i=1,\cdots,I,\,j=1,\cdots,J-1.
	\end{array}	
\]
\end{enumerate}

\subsection{Discretized Total Variation}
\label{sec:pre:definationTV}
 
Assuming zero Neumann boundary conditions for an image $\mX\in\mbb C^{I\times J}$, i.e.,
$$
\mX_{I+1,j}-\mX_{I,j}=0,~~\forall j~\mrm{and}~\mX_{i,J+1}-\mX_{i,J}=0,~~\forall i,
$$
the isotropic and anisotropic TV functions are defined as follows 
\begin{equation}
\begin{array}{rl}
 	 \mrm{TV}_{\mrm{iso}}(\mX)=&\mathlarger{\sum}\limits_{i=1}^{I-1}\mathlarger{\sum}\limits_{j=1}^{J-1}\sqrt{\left(\mX_{i,j}-\mX_{i+1,j}\right)^2+\left(\mX_{i,j}-\mX_{i,j+1}\right)^2}\\
 	 &+\mathlarger{\sum}\limits_{i=1}^{I-1} \left|\mX_{i,J}-\mX_{i+1,J}\right|+\mathlarger{\sum}\limits_{j=1}^{J-1}\left|\mX_{I,j}-\mX_{I,j+1}\right|,
\end{array}
\label{eq:isotropic TV}
\end{equation}
 and
\begin{equation}
  \begin{array}{rl}
 	 \mrm{TV}_{\ell_1}(\mX)=&\mathlarger{\sum}\limits_{i=1}^{I-1}\mathlarger{\sum}\limits_{j=1}^{J-1}\Big\{\left|\mX_{i,j}-\mX_{i+1,j}\right|+\left|\mX_{i,j}-\mX_{i,j+1}\right|\Big\}\\
 	 &+\mathlarger{\sum}\limits_{i=1}^{I-1} \left|\mX_{i,J}-\mX_{i+1,J}\right|+\mathlarger{\sum}\limits_{j=1}^{J-1}\left|\mX_{I,j}-\mX_{I,j+1}\right|,
 \end{array}
\label{eq:anisotropic TV}
\end{equation}
respectively.
Hereafter, we use $\mrm{TV}(\vx)$ to represent either
$\mrm{TV}_{\mrm{iso}}(\mX)$ or $\mrm{TV}_{\ell_1}(\mX)$.

\subsection{Weighted Proximal Mapping}
\label{sec:subsectionWPMProperties}


Given a proper closed convex function $h(\vx)$
and a Hermitian positive definite matrix
$\mW \succ0:\in\mbb{C}^{N\times N}$, the WPM associated to $h$ is defined as  
\begin{equation}
\mrm{prox}_{h}^{\mW}(\vx)=\arg\min\limits_{\vu}\left(h(\vu)+\frac{1}{2}\|\vu-\vx\|^2_{\mW}\right),
	\label{eq:def:WeightedProximal}
\end{equation} 
where $\|\cdot\|_{\mW}$ denotes the $\mW$-norm
defined by $\|\vq\|_{\mW}=\sqrt{\vq^\mathcal H\mW\vq}$.
Here $\mathcal H$ denotes Hermitian transpose. 
Clearly, \eqref{eq:def:WeightedProximal} simplifies to the proximal mapping for $\mW=\mI_{N}$
where $\mI_{N}$ represents the identity matrix.
Since  $h(\vu)+\frac{1}{2}\|\vu-\vx\|^2_{\mW}$ is strongly convex,
$\mrm{prox}_{h}^{\mW}(\vx)$ exists and is unique for $\vx\in\mrm{dom}\,h$
so that the WPM is well defined. 

\section{Complex Quasi-Newton Proximal Methods}
\label{Sec:CQNPM}

{\cb This section first describes a complex quasi-Newton proximal method (CQNPM)
for solving \eqref{eq:mainMin}
with regularizer
$h(\vx)=\alpha\|\mT\vx\|_1+(1-\alpha)\mrm{TV}(\vx)$ and $\alpha\in[0,1]$.
Here, we consider $\mT\in\mathbb C^{\tilde N\times N}$
to be a wavelet transform.
Then, we propose efficient methods to compute the related WPM.
Moreover, to avoid applying wavelet transforms when computing the WPM for $\alpha\in[0,1)$,
we propose a partial smooth approach.
Our numerical experiments show that such a partial smooth strategy
recovers the desired images with less computation. 

At $k$th iteration, CQNPM solves \eqref{eq:CQNPMSub} for $\vx_{k+1}$,
\begin{align}
\vx_{k+1}&=\arg\min_{\vx\in\mathbb C^N} f(\vx_k)+\left<\nabla f(\vx_k),\vx-\vx_k\right>
+\frac{1}{2a_k}\|\vx-\vx_k\|_{\mB_k}^2
\nonumber\\
&~~~~~~~~~~~~~
\mbox{} + \lambda \, h(\vx)
\nonumber\\
&=\mathrm{prox}^{\mB_k}_{a_k\lambda h}(\vx_k-a_k\mB_k^{-1}\nabla_{\vx}f(\vx_k))
, \label{eq:CQNPMSub}
\end{align}
where $a_k$ is the step-size
and $\mB_k\in\mathbb C^{N\times N}$ is a Hermitian symmetric positive definite matrix.
For clarity, we present the detailed steps of CQNPM in \Cref{alg:WPMs}.
Note that \Cref{alg:WPMs} would be identical to the proximal methods \cite{parikh2014proximal}
if one chose $\mB_k=\mI_N$.
In \cite{chouzenoux2014variable,bonettini2016variable},
the authors suggested using a diagonal matrix $\mB_k$ for their application.
However, building such a diagonal matrix is nontrivial and its effectiveness is problem dependent.
In this paper, we choose $\mB_k$ to be a more accurate approximation of the Hessian of $f(\vx)$.
Specifically,
we select $\mB_k$
based on the Symmetric Rank-$1$ (SR$1$) method \cite{jorge2006numerical},
a popular method used in quasi-Newton methods for approximatnig a Hessian matrix.
Following the derivation of SR$1$ for real variables,
we derive a complex plane SR$1$ that is similar to the real one.
\Cref{alg:SR1} presents the implementation details of SR$1$ in the complex plane. 
{  We found that using $\gamma>1$ is crucial
to ensure that $\mB_k$ is Hermitian positive definite in our setting
because otherwise $\left<\vm_k-\mH_0\vs_k,\vs_k\right>$ can become negative,
causing $\mB_k$ to turn indefinite.}
In our numerical experiments, we found that a fixed $\gamma>1$ worked well.   
}

\begin{algorithm}[!t]        
\caption{Proposed complex quasi-Newton proximal method.}    
\label{alg:WPMs} 
\begin{algorithmic}[1]
\REQUIRE $\vx_1$.
\ENSURE 
\FOR {$k=1,2,\dots$}
\STATE pick the step-size $a_k$ and the weighting $\mB_k$.
\STATE $\vx_{k+1}\leftarrow \mrm{prox}_{a_k\lambda h}^{\mB_k}\left(\vx_{k}-a_k\mB_k^{-1}\nabla_{\vx} f(\vx_{k})\right)$. \label{alg:WPMs:EvWPM}
\ENDFOR
\end{algorithmic}
\end{algorithm}

\begin{algorithm}[!t]        
\caption{SR1 updating.}         
\label{alg:SR1} 
\begin{algorithmic}[1]
\REQUIRE $\gamma>1 $, $\delta = 10^{-8}$, $\Xi>0$ a fixed real scalar, $\vx_k$, $\vx_{k-1}$, $\nabla f(\vx_k)$, and $\nabla f(\vx_{k-1})$.
\IF {$k=1$}
\STATE $\mB_k\leftarrow \Xi \mI$.
\ELSE
\STATE Set $\vs_k \leftarrow \vx_k-\vx_{k-1}$ and $\vm_k \leftarrow \nabla f(\vx_k)-\nabla f(\vx_{k-1})$.
\STATE Compute $\tau_k \leftarrow \gamma\frac{\|\vm_k\|_2^2}{\left<\vs_k,\vm_k\right>}$. ~~~~~ \% $\left<\va,\vb\right>=\vb^\mathcal H\va$
\IF {$\tau<0$}
\STATE $\mB_k\leftarrow\Xi\mI$.
\ELSE
\STATE $\mH_0 \leftarrow \tau_k\mI$.
\STATE $\vu_k\leftarrow \vm_k-\mH_0\vs_k$.
\IF {$|\left<\vu_k,\vs_k\right>|\leq \delta \|\vs_k\|_2\|\vu_k\|_2$}
\STATE $\vu_k\leftarrow \bm 0$.
\ENDIF
\STATE $\mB_k \leftarrow \mH_0+\frac{\vu_k\vu_k^\mathcal H}{\left<\vm_k-\mH_0\vs_k,\vs_k\right>}$.\label{alg:SR1:HessUpdat}
\ENDIF
\ENDIF
\STATE {\bf Return:} $\mB_k$
\end{algorithmic}
\end{algorithm}

\subsection{Compute Weighted Proximal Mapping}
The dominant computation in \Cref{alg:WPMs}
is computing the WPM at Step \ref{alg:WPMs:EvWPM}
which could be as hard as solving \eqref{eq:mainMin} for a general $\mB_k$.
However, we find one can compute $\mrm{prox}^{\mB_k}_{a_k\lambda h}(\cdot)$
as easily as the case when $\mB_k=\mI_N$
by using the structure of $\mB_k$.

To compute the WPM  $\mrm{prox}_{\bar{\lambda} h}^{\mB_k}(\vv_k)$ at $k$th iteration,
we need to solve the following problem
\begin{equation}
\min_{\vx\in\mbb C^N}
\|\vx-\vv_k\|_{\mB_k}^2 + 2\bar{\lambda} \big[\alpha\|\mT\vx\|_1 + (1-\alpha)\mrm{TV}(\vx)\big]
,\label{eq:TVWavReco:WPM}
\end{equation}
where $\vv_k=\vx_k-a_k\mB_k^{-1}\nabla_{\vx} f(\vx_k)$ and $\bar{\lambda}=a_k\lambda$.
A difficulty of \eqref{eq:TVWavReco:WPM} is the nonsmoothness of $\|\cdot\|_1$ and $\mrm{TV}(\cdot)$.
To address this difficulty, we consider a dual approach for \eqref{eq:TVWavReco:WPM}
that is similar to Chambolle's approach
for TV-based image reconstruction \cite{chambolle2004algorithm}. {\cb
Our method only uses one inner iteration to compute the WPM,
and the related gradient is computed easily.}
Proposition~\ref{prop:CQNP:TVWav:Dual}
describes the dual problem of \eqref{eq:TVWavReco:WPM}
and the relation between the primal and dual optimal solutions.

\begin{prop}
\label{prop:CQNP:TVWav:Dual}
Let 
\begin{equation}
\label{eq:TVWavReco:WPM:DualFinal:Main}    
(\vz^*,\mP^*,\mQ^*)=\argmin\limits_{\substack{\vz\in\SetZ\\(\mP,\mQ)\in\SetP}}
\left\|\vw_k(\vz,\mP,\mQ)\right\|^2_{\mB_k}
\end{equation}
where
$\vw_k(\vz,\mP,\mQ)=\vv_k-\bar{\lambda}{\mB}_k^{-1}
\left(\alpha\mT^\mathcal H\vz+(1-\alpha)\mrm{vec}\left(\mathcal L(\mP,\mQ)\right)\right)
$
and
$\SetP = \SetP_1~\text{or}~\SetP_2$ depending on which TV is used.
Then the optimal solution of \eqref{eq:TVWavReco:WPM} is given by
$\vx_{k+1}=\vw_k(\vz^*,\mP^*,\mQ^*).$
\end{prop}
\begin{proof}
{\cb See \Cref{App:proof:prop:CQNP:TVWav:Dual}.}
\end{proof}

Using Proposition \ref{prop:CQNP:TVWav:Dual},
we can apply the FISTA \cite{nesterov1983method,beck2009fast}
to solve \eqref{eq:TVWavReco:WPM:DualFinal:Main}
for computing $\mrm{prox}_{\bar\lambda h}^{\mB_k}$
since \eqref{eq:TVWavReco:WPM:DualFinal:Main} is convex and continuously differentiable.
\Cref{lemma:gradient:LipschitzConstant:dualPro}
specifies
the corresponding gradient and Lipschitz constant of \eqref{eq:TVWavReco:WPM:DualFinal:Main}.
\begin{lemma}
	\label{lemma:gradient:LipschitzConstant:dualPro}
The gradient of \eqref{eq:TVWavReco:WPM:DualFinal} is
\begin{equation}
	\label{eq:TVWavReco:WPM:DualFinal:gradient}
-2\bar{\lambda} \bmat
 \alpha\mT\\
	(1-\alpha)\mathcal L^\mathcal{T}
 \emat \vw_k(\vz,\mP,\mQ)
\end{equation}
and the corresponding Lipschitz constant is
$$L_c=2\sigma_{\mrm{min}}
\bar{\lambda}^2(\alpha^2\|\mT\|^2+8(1-\alpha)^2)$$
where $\sigma_{\mrm{min}}$ is the smallest eigenvalue of ${\mB}_k$.
\end{lemma}
\begin{proof}
{\cb See \Cref{App:proof:lemma:gradient:LipschitzConstant:dualPro}.}
\end{proof}

According to the formulation of $\mB_k$ proposed in \Cref{alg:SR1},
we can obtain $\sigma_{min}$ easily through%
\footnote{We note that $\left<\vm_k-\mH_0\vs_k,\vs_k\right>$
is real in our setting, see Observation I.}
$$
\sigma_{min}=\left\{\begin{array}{ll}
    \Xi& \mrm{if} ~~\tau<0,  \\
    \tau&  \mrm{if} ~~\left<\vm_k-\mH_0\vs_k,\vs_k\right>>0,\\
    \tau+\frac{\vu^\mathcal H\vu}{\left<\vm_k-\mH_0\vs_k,\vs_k\right>}&\mrm{if} ~~\left<\vm_k-\mH_0\vs_k,\vs_k\right><0.
\end{array}\right.
$$
The value of $\|\mT\|$ depends on the choice of wavelets which can be computed in advance,
so the computational cost of obtaining the Lipschitz constant of \eqref{eq:TVWavReco:WPM:DualFinal} is cheap.
For completeness,
\Cref{alg:FPMDualProx} presents the implementation details of FISTA
for solving \eqref{eq:TVWavReco:WPM:DualFinal}.
{\cb We terminate \Cref{alg:FPMDualProx} when it reaches a maximal number of iterations
or a given accuracy tolerance.
The initial value $(\vz_1,\mP_1,\mQ_1)$ in \Cref{alg:FPMDualProx}
uses the final solution of the previous iteration.}


\begin{algorithm}[!htb]        
\caption{FISTA for solving \eqref{eq:TVWavReco:WPM:DualFinal}.}    
\label{alg:FPMDualProx} 
\begin{algorithmic}[1]
\REQUIRE $\mB_k$, $\vv_k$, $\bar{\lambda}>0$, $\alpha\in[0,1]$, Lipschitz constant $L_c$, maximal iteration $\mrm{Max}\_\mrm{Iter}$, tolerance $\epsilon>0$, and initial values $\vz_1,\mP_1,
\mQ_1$.
\ENSURE 
\STATE $t_1 \leftarrow 1$.
\STATE $(\bar{\vz}_1,\bar{\mP}_1,\bar{\mQ}_1)\leftarrow({\vz}_1,{\mP}_1,{\mQ}_1)$.
\FOR {$s=1,2,\dots,\mrm{Max}\_\mrm{Iter}$}
\STATE Compute $\bar{\vw}\leftarrow\vw_k(\bar{\vz}_s,\bar{\mP}_s,
\bar{\mQ}_s).$
\IF {$\alpha\neq 0$}
\STATE {$\vz_{s+1}\leftarrow \mrm{Proj}_{\mathcal Z}\big(\bar{\vz}_s+\frac{2\bar\lambda\alpha}{L_c}\mT\bar{\vw}\big).$}
\ELSE
\STATE Set $\vz_{s+1}$ empty.
\ENDIF
\IF {$\alpha\neq 1$}
\STATE $(\mP_{s+1},\mQ_{s+1})\leftarrow \mrm{Proj}_{\mathcal P}\big((\bar{\mP}_s,\bar{\mQ}_s)+\frac{2\bar\lambda(1-\alpha)}{L_c}\mathcal L^\mathcal T\bar{\vw}\big)$.
\ELSE
\STATE Set $\mP_{s+1}$ and $\mQ_{s+1}$ empty.
\ENDIF
\IF{$\|(\vz_{s+1}-\vz_s,\mP_{s+1}-\mP_s,\mQ_{s+1}-\mQ_s)\|\leq\epsilon$}
\STATE break.
\ENDIF
\STATE $t_{s+1}\leftarrow \frac{1+\sqrt{1+4t_s^2}}{2}$.
\STATE $(\bar{\vz}_{s+1},\bar{\mP}_{s+1},\bar{\mQ}_{s+1})\leftarrow\begin{array}{l} \frac{t_{s+1}+t_s-1}{t_{s+1}}({\vz}_{s+1},{\mP}_{s+1},{\mQ}_{s+1})\\ [6pt]
~~~-\frac{t_s-1}{t_{s+1}}({\vz}_s,{\mP}_s,{\mQ}_s).\end{array}$
\STATE $t_s\leftarrow t_{s+1}$.
\ENDFOR
\end{algorithmic}
\end{algorithm}

\begin{remark}
    \label{remark:FPMDualProx}
    Compared with APMs for addressing \eqref{eq:mainMin},
    the additional cost of CQNPM is applying $\mB_k^{-1}$
    in computing $\vv_k$ and $\vw_k$ in \Cref{alg:WPMs,alg:FPMDualProx}.
    This inversion can be computed cheaply through the Woodbury matrix identity.
    Moreover, computing the projectors $\mrm{Proj}_{\mathcal Z}(\cdot)$
    and $\mrm{Proj}_{\mathcal P}(\cdot)$ is also cheap
    and identical to the one shown in \cite{beck2009fastTV}, so we omit the details here.
    The step-size $a_k$ in \Cref{alg:WPMs} can be simply set to be $1$. 
\end{remark}

\subsection{Compute the Weighted Proximal Mapping when \texorpdfstring{$\alpha=1$}{alpha}}
For $\alpha=1$, running \Cref{alg:FPMDualProx} to compute the WPM would be inefficient
since we would have to apply wavelet transform many times at each outer iteration.
However, if $\mT$ {\red is left invertible that $\mT^\mathcal H\mT=\mI_N$ }, we can solve the following problem instead of \eqref{eq:mainMinWavTV}
to avoid using \Cref{alg:FPMDualProx} to compute the WPM:
\begin{equation}
\bar{\vx}^*=\argmin_{\bar{\vx}\in\mathbb C^{\tilde N}}\underbrace{\frac{1}{2}\|\mA\mT^\mathcal H\bar{\vx}-\vy\|_2^2}_{f(\bar{\vx})}+\lambda \|\bar{\vx}\|_1
.\label{eq:WaveletReco}
\end{equation}
Then the recovered image is $\vx^*=\mT^\mathcal H\bar{\vx}^*$.
Now the corresponding WPM becomes
\begin{equation}
\mrm{prox}_{\bar{\lambda}\|\cdot\|_1}^{\mB_k}(\vv_k)=\argmin_{\bar{\vx}\in\mathbb C^{\tilde N}}\|\bar{\vx}-\vv_k\|_{\mB_k}^2+2\bar{\lambda}\|\bar{\vx}\|_1.\label{eq:WavReco:WPM}
\end{equation}
{ Note that \eqref{eq:mainMinWavTV} and \eqref{eq:WaveletReco}
represent the analysis-based and synthetic-based priors, respectively.
For a detailed discussion of their relations and equivalence, see \cite{elad2007analysis}.}

Let $\mW\in\mathbb{C}^{\tilde N\times \tilde N}: = \mD\pm\vu\vu^\mathcal H$
where $\mD\in \mathbb R^{\tilde N\times \tilde N}$ is a diagonal matrix
and $\vu\in\mathbb{C}^{\tilde N}$.
Becker et al. proposed the following theorem that relates
$\mrm{prox}_{\lambda h}^{\mW}(\vx)$ and $\mrm{prox}_{\lambda h}^{\mD}(\vx)$.
\begin{theorem}[Theorem 3.4, \cite{becker2019quasi}\footnote{The theorem is proved in real plane but it is also valid in complex plane.}]
\label{them:structuredWPM:evaluation}
	Let $\mW = \mD\pm\vu\vu^\mathcal H$. Then,
$$\mrm{prox}_{\lambda h}^{\mW}(\vx)=\mrm{prox}_{\lambda h}^{\mD}(\vx\mp\mD^{-1}\vu\beta^*),$$
where $\beta^*\in\mathbb C$ is the unique zero of the following nonlinear equation
$$
\mathbb J(\beta):\vu^\mathcal H\left(\vx-\mrm{prox}_{\lambda h}^{\mD}(\vx\mp\mD^{-1}\vu\beta)\right)+\beta.
$$
\end{theorem}
\red{Using the notation in \Cref{alg:SR1},}
we have the following observation:
\begin{namedthm*}{Observation I}
$\tau$ and $\left<\vm_k-\mH_0\vs_k,\vs_k\right>$ in \Cref{alg:SR1} are real.
\end{namedthm*}
\begin{proof}
{\cb
Note that $f(\vx)=\frac{1}{2}\|\mA\vx-\vy\|_2^2$.
Then we have $\vm_k=\mA^\mathcal H\mA\vs_k$,
so $\left<\vs_k,\vm_k\right>$ is real.} 
\end{proof}
Since $\left<\vm_k-\mH_0\vs_k,\vs_k\right>$ is real, we  rewrite $\mB_k$ as
\begin{equation}
    \mB_k = \mH_0+\mrm{sgn}\left(\left<\vm_k-\mH_0\vs_k,\vs_k\right>\right)\tilde{\vu}_k\tilde{\vu}_k^\mathcal H, \label{eq:B_k:form}
\end{equation}
where
$\tilde{\vu}_k=\frac{\vu_k}{\sqrt{\left<\vm_k-\mH_0\vs_k,\vs_k\right>}}$
and $\mrm{sgn}(\cdot)$ denotes the sign function
such that $\mB_k$ holds the same structure as $\mW$ in \Cref{them:structuredWPM:evaluation}.
So, instead of solving \eqref{eq:WavReco:WPM} directly,
we first solve $\mathbb J(\beta)=0$
and then use \Cref{them:structuredWPM:evaluation}
to obtain $\mrm{prox}_{\bar\lambda\|\cdot\|_1}^{\mB_k}(\vv_k)$.
In this paper, we solve
$\mbb J(\beta) = 0$
using ``SciPy'' library in Python.

\subsection{Partial Smoothing}

For $\alpha\in(0,1)$, \Cref{alg:FPMDualProx} still requires
applying many wavelet transforms,
which can dominate the computational cost.
An alternative way is to use the idea proposed in \cite{beck2012smoothing}
where one partially smooths the objective and then applies \Cref{alg:WPMs}.
For comparison purposes, 
we apply \Cref{alg:WPMs} to the following problem 
\begin{equation}
  \min_{\vx\in\mathbb C^N} \underbrace{ \frac{1}{2}\|\mA\vx-\vy\|_2^2+\lambda\alpha\cdot\mrm{S}^{\eta}\big(\|\mT\vx\|_1\big)}_{f(\vx)}
  +\underbrace{\lambda(1-\alpha)\mrm{TV}(\vx)}_{h(\vx)},\label{eq:mainMin:PartialSmooth}
\end{equation}
such that each outer iteration needs only two wavelet transforms.
For the comparisons in this paper,
we used
$\mrm{S}^{\eta}(\|\vx\|_1)=\sum_{n=1}^N\sqrt{\vx_n^2+\eta}$
with $\eta>0$ so that $f(\vx)$ in \eqref{eq:mainMin:PartialSmooth} is differentiable.
Our numerical experiments compare the performance of such a partial smoothing approach
to methods based on the original cost function
for image reconstruction in CS MRI.%
\footnote{One could instead partially smooth the TV regularizer.
However, in our settings, we found that smoothing $\|\mT\vx\|_1$ led to better qualifty
than TV smoothing.}

\section{Numerical Experiments}
\label{sec:numericalExp}

This section studies the performance of our algorithm
for image reconstruction in CS MRI with {\cb non-Cartesian sampling trajectories.
Specifically, we consider the radial and spiral trajectories.
Moreover, we also study the robustness of our algorithm
to the choice of $\gamma$ and Max\_Iter in \Cref{alg:SR1,alg:FPMDualProx}, respectively.} 
Similar to \cite{lustig2007sparse}, we focus on wavelet and TV regularizers.
{\cb We first present our experimental and algorithmic settings
and then show our reconstruction results.}

\noindent {\bf {\emph{Experimental Settings}}:}
We took complex k-space data
from the brain and knee training datasets
(one each)
in the NYU fastMRI dataset \cite{zbontar2018fastmri}
to generate the simulated k-space data.
We applied the ESPIRiT algorithm \cite{uecker2014espirit} to recover the complex images
and then cropped the images to size $256\times 256$
to define the ground-truth images,
with maximum magnitude scaled to one.
\Cref{fig:GT_imag} shows the magnitude of the complex-valued ground-truth images.
Following \cite{iyer2022accelerating},
we used
$32$ interleaves, $1688$ readout points, and $12$-coils
(respectively, $96$ radial projections, $512$ readout points, and $12$ coils)
for the spiral (respectively, radial) trajectory to define the forward model $\mA$.
{ \Cref{fig:Trj_sampling} presents the used trajectories in this paper.
For clarity, we plot only every 4th sample of the trajectories.}
Applying the used forward model to the ground truth image generated the noiseless multi-coil k-space data. 
We added complex i.i.d Gaussian noise with mean zero and variance $10^{-2}$
to all coils to form the measurements, $\vy$.
{\Mcb The data input SNR was below $7$dB.
We also studied a higher data input SNR case of around $30$dB.} 
Our implementation used Python programming language with SigPy library \cite{ong2019sigpy}.
The reconstructions ran on a workstation with 2.3GHz AMD EPYC 7402.
Our \emph{code} is available on~\url{https://github.com/hongtao-argmin/CQNPCS_MRIReco}.
{\Mcb The supplementary material provides additional experimental results
and a comparison with a Plug-and-Play reconstruction method
using BM3D and a deep denoiser \cite{zhang2021plug}.}

\input{Figs/GTFig.tex}

\input{Figs/TrjFig.tex}

\noindent {\bf{\emph{Algorithmic Settings}}:}
For APM, we precomputed the Lipschitz constant for all experiments.
For CQNPM, we set $a_k=1$ and $\gamma=1.7$.
Denote by S-APM (respectively, S-CQNPM) when APM (respectively, CQNPM) is used to solve \eqref{eq:mainMin:PartialSmooth}.
We chose the step-size in S-APM
using a backtracking strategy \cite{beck2017first}. {\cb Moreover, we also compared our method with primal-dual (PD) methods \cite{sidky2012convex}.}
The tradeoff parameters $\lambda$ and $\alpha$ were chosen to reach the \emph{highest} peak signal-to-noise ratio (PSNR)
when running enough iterations of APM.
We set $\eta = 10^{-5}$ in our experiments. {\cb The maximal number of iterations and tolerance in} \Cref{alg:FPMDualProx} {\cb  are set to be $20$ and $10^{-6}$ for both CQNPM and APM.}

\subsection{Radial Acquisition MRI Reconstruction}
\Cref{fig:brain:radial:cost,fig:RadialBrain:wavelet} show the performance of \Cref{alg:WPMs}
for the wavelet based reconstruction of the brain image
and the comparison with APM \cite{beck2009fast} and PD \cite{sidky2012convex}. {\cb Here, we used  \Cref{them:structuredWPM:evaluation} to compute the WPM.}
Clearly, CQNPM converged faster than APM {\cb and PD} in terms of iterations.
Compared with the cost of computing the proximal mapping,
the additional cost of computing WPM with our method is insignificant.
{\cb \Cref{fig:brain:radial:cost,fig:RadialBrain:wavelet} show that
the computational costs of CQNPM and APM per iteration are similar}.
The comparison of PSNR versus CPU time in \Cref{fig:RadialBrain:wavelet}
also shows that CQNPM reached a higher PSNR with less CPU time,
illustrating the fast convergence of CQNPM.
The reconstructed images at $3$, $10$, $13$, and $16$th iteration
illustrate that CQNPM yielded a clearer image than APM for the same number of iterations.
{\cb Since PD led to a much lower PSNR than APM,
we do not present the reconstructed images of PD.}
Similar observations apply to the knee image
{\cb and the related results are provided in the supplemental material}.

\input{Figs/BrainWaveFigRadial.tex}

We also studied the performance of our algorithm when using both wavelet and TV regularizers. {\cb Here, we used \Cref{alg:FPMDualProx} to compute the the proximal mapping and WPM.}
Since ADMM is a classical method for \eqref{eq:mainMinWavTV}
with $h(\vx)=\alpha \|\mT\vx\|_1+(1-\alpha)\mrm{TV}(\vx)$,
we include a comparison with ADMM.
Moreover, we also studied the performance of the partial smoothing technique.
{\cb Although PD does not require any inner iteration,
unlike ADMM, APM, and our method,
our method is still faster than PD in terms of iterations and CPU time.}

\Cref{fig:brain:radial:WavTV:cost,fig:RadialBrain:WavTV}
present the results for the reconstruction of the brain image.
{\cb CQNPM reduced the cost faster than APM in terms of iterations and CPU time.}
Although we solved \eqref{eq:mainMin:PartialSmooth} instead of \eqref{eq:mainMinWavTV}
for the partial smoothing method, the cost is still computed with \eqref{eq:mainMinWavTV}. 
Surprisingly, {\cb in this setting,} we see that, for the cost values versus iterations,
S-APM (respectively, S-CQNPM) converged similar to APM (respectively, CQNPM) {\cb in terms of iterations}.
However, from the cost values versus CPU time plot,
{\cb S-CQNPM converged faster than CQNPM, as expected
since the partial smoothing method requires only two wavelet transforms per outer iteration.
However, S-APM converged slower than APM in terms of CPU time
because S-APM requires applying a line search to choose the step-size,
increasing the computational cost.}
{\MRcb{Although CQNPM/S-CQNPM require an iterative method to solve the WPM,
\Cref{sub:exp:sub:ChoiceMaxiter} demonstrated that the WPM can be solved inexactly,
and the computation for solving the WPM is relatively inexpensive
compared to executing $\mA\vx$ in CS MRI reconstruction.
Thus, CQNPM/S-CQNPM converged faster than ADMM/PD
both in terms of iteration numbers and in CPU time.
Note that ADMM requires solving a least-squares problem at each iteration,
which involves applying $\mA\vx$ multiple times,
leading to significantly slower convergence in terms of CPU time.}}

The PSNR versus CPU time plot {\cb in \Cref{fig:RadialBrain:WavTV}}
also demonstrates the fast convergence of CQNPM and S-CQNPM.
{\cb Compared with the previous experiments that only used a wavelet regularizer,
we see an improved PSNR here, confirming the benefit of using both wavelet and TV regularizers.} 
The reconstructed images at $3$, $10$, $13$, and $16$th iteration for each method%
\footnote{We do not show the reconstructed image of ADMM
since it yielded a much lower PSNR than other methods.}
illustrate that the partial smoothing method works as well as the nonsmoothing one.
In summary,
the proposed method converged faster than other methods in terms of iterations and CPU time,
and S-CQNPM is the best algorithm for \eqref{eq:mainMinWavTV} in this setting.
{\cb We also tested our algorithm on the knee image
and provided the results in the supplementary material.}

\input{Figs/BrainWaveTVFigRadial.tex}

\subsection{Spiral Acquisition MRI Reconstruction}
This part studies the reconstruction with spiral acquisition that used
$32$ interleaves, $1688$ readout points, and $12$ coils.
\Cref{fig:knee:Spiral:WavTV:cost,fig:SpiralKnee:WavTV}
show the results of the knee image with wavelet and TV regularizers. The trends are similar to the radial acquisition case.  {\cb Note that CQNPM reduced the cost values faster than S-CQNPM in terms of iterations and CPU time in this setting. However, S-CQNPM reached a higher PSNR than CQNPM with same CPU time.} {\cb We provided the reconstruction of the brain and knee images with wavelet regularizer and the brain image with wavelet and TV regularizers in the supplementary material.}

\input{Figs/BrainWavTVGamma.tex}
\input{Figs/BrainWavTVInner.tex}

\input{Figs/KneeWaveTVFigSpiral.tex}

\subsection{The Choice of \texorpdfstring{$\gamma$}{gamma}}
\label{sec:numericalExp:sub:gamma}
We tried several different $\gamma$ values
to study how $\gamma$ affects the convergence of CQNPM.
We reconstructed the brain image with wavelet and TV regularizers and radial acquisition.
\Cref{fig:RadialBrain:WavTVGamma} presents the results
that show that CQNPM is quite robust to different $\gamma$ values,
and $\gamma=1.7$ worked slightly better than the others.
So we simply set $\gamma=1.7$ for all experiments.

\subsection{The Choice of \texorpdfstring{Max\_Iter in \Cref{alg:FPMDualProx}}{maxiter}}
\label{sub:exp:sub:ChoiceMaxiter}
{\cb Following the setting used in \Cref{sec:numericalExp:sub:gamma}, we studied how the choice of Max\_Iter in \Cref{alg:FPMDualProx} affects the converge of CQNPM.  \Cref{fig:RadialBrain:WavTVInner} presents the cost values versus iteration with different values of Max\_Iter. Clearly we see that CQNPM is quite robust to the choice of Max\_Iter. However, a small Max\_Iter (e.g., Max\_Iter$=10$) can slightly increase the cost and Max\_Iter$=20,50$ converged faster than other vaules. In our experiments, we found that Max\_Iter$=20$ is sufficient.}

\subsection{Reconstruction with High Data Input SNR}
{\Mcb This part studies the reconstruction
for complex additive Gaussian noise with mean zero and lower variance $4\times10^{-5}$,
yielding around $30$dB data input SNR.
\Cref{fig:SpiralBrainHigh:WavTV} displays the reconstructed results using spiral acquisition and
$h(\vx)=\alpha \|\mT\vx\|_1+(1-\alpha)\mathrm{TV}(\vx)$.
The reconstructed images are much clearer than those in the low data input SNR cases.
Moreover, the convergence trends of different algorithms
are similar to those observed in low data input SNR reconstructions.
The supplementary material provides the reconstructed results of the knee image
that align with the observations made from the brain image presented here.
} 

\input{Figs/BrainWaveTVFigSpiralhighSNR.tex}


\section{Conclusions and Future Work}
\label{Sec:conclusion}

This paper proposes complex quasi-Newton proximal methods for solving \eqref{eq:mainMinWavTV}
that led to faster convergence than APMs.
By using the structure of $\mB_k$,
we develop efficient approaches for computing the WPM {\Mcb by considering wavelet and TV regularizers}.
Compared with computing the proximal mapping in APMs,
i.e., $\mB_k=\mI_N$,
the increased computational cost in computing the WPM is insignificant,
as illustrated by our comparisons in terms of CPU time.
CQNPM is appealing for large-scale problems
because CQNPM requires fewer iterations than APMs to converge,
reducing the times of computing $\nabla f(\vx)$ that it is expensive in large-scale settings.
Interestingly, in our setting,
we found the partial smoothing method worked pretty well
when both wavelet and TV regularizers are used.
So the partial smoothing approach may be a good method for solving problems with two nonsmooth terms.
{\Mcb
To adapt CQNPM to other regularizers,
one must find an efficient approach
to address the WPM for the chosen regularizer
to preserve the computational efficiency.}


Clearly, $\mB_k$ plays an important role in our algorithm
and a more accurate $\mB_k$ can accelerate the convergence further.
Since the Hessian matrix in CS MRI is known, i.e., $\mA^\mathcal H\mA$,
we plan to learn a fixed weighting $\mB$ to approximate $\mA^\mathcal H\mA$ accurately for future work.
However, $\mB$ must be easy to invert
so that $\mB$ should have some special structures, e.g., $\mB=\mD\pm\mU\mU^\mathcal H$, {\cb and finding such a $\mB$ should be computationally cheap since $\mA$ is different from each acquisition because the sensitivity mapping is patient dependent.} Moreover, with such a fixed $\mB$, we can adopt the accelerated manner used in APMs for \Cref{alg:WPMs}
and obtain an even faster algorithm than the one presented here.  
\section{Acknowledgements}
This work was funded by National Institutes of Health grant R01NS112233.

\ifCLASSOPTIONcaptionsoff 
\newpage
\fi


\appendices

\section{Proof of Proposition \ref{prop:CQNP:TVWav:Dual}}
\label{App:proof:prop:CQNP:TVWav:Dual}
Similar to \cite{beck2009fastTV} for the real case,
one can prove the following relations for complex numbers $x,y\in \mathbb C$

\begin{equation*}
    \begin{array}{c}
  \sqrt{|x|^2+|y|^2}=\max\limits_{p_1,p_2\in\mathbb C}\left\{\Re\left(p_1^*x+p_2^*y\right):|p_1|^2+|p_2|^2\leq 1\right\}  \\ [5pt]
         |x|=\max\limits_{p\in\mathbb C}\left\{\Re\left(p^*x\right): |p|\leq 1\right\} ~~~~~~~~~~~~
    \end{array}
\end{equation*}
where $^*$ denotes the conjugate operator and $\Re(\cdot)$ represents an operator to take the real part. With these relations and the definition of TV functions, we can rewrite $\mrm{TV}(\vx)$ and  $\|\mT\vx\|_1$ as
\begin{align*}
    \mrm{TV}(\vx)=&\max_{(\mP,\mQ)\in \mathcal P} \Re\left\{\mrm{vec}\left(\mathcal L\left(\mP,\mQ\right)\right)^\mathcal H\vx\right\},\\
    \|\mT\vx\|_1 =& \max\limits_{\vz\in\mathcal Z}\Re\left\{\vz^\mathcal H\mT\vx\right\},
\end{align*}
where $\mathcal P=\mathcal P_1$ (respectively, $\mathcal P_2$) for $\mathrm{TV}_{\mrm{iso}}$ (respectively, $\mathrm{TV}_{\ell_1}$). Hence, we represent \eqref{eq:TVWavReco:WPM} as   
\begin{equation}
 \min_{\vx\in\mathbb C^N}\max\limits_{\substack{\vz\in\mathcal Z \\(\mP,\mQ)\in \mathcal P}} \|\vx-\vv_k\|_{\mB_k}^2 +2\bar{\lambda} g(\vx,\vz,\mP,\mQ),\label{eq:TVWavReco:WPM:DualOrig}
\end{equation}
where $$g(\vx,\vz,\mP,\mQ)= \Re\left\{\alpha\left<\mT\vx,\vz\right>+ (1-\alpha)\mrm{vec}\left(\mathcal L\left(\mP,\mQ\right)\right)^\mathcal H\vx\right\}.$$
Reorganizing \eqref{eq:TVWavReco:WPM:DualOrig}, we get
\begin{equation}
\begin{array}{cl}
\min\limits_{\vx\in\mathbb C^N} \max\limits_{\substack{\vz\in\mathcal Z,\\(\mP,\mQ)\in\mathcal P}}&\left\|\vx-\vw_k(\vz,\mP,\mQ)\right\|_{{\mB}_k}^2-\left\|\vw_k(\vz,\mP,\mQ)
\right\|_{{\mB}_k}^2,
\label{eq:TVWavReco:WPM:DualFormI} 
\end{array}
\end{equation}
where $$\vw_k(\vz,\mP,\mQ)=\vv_k-\bar{\lambda} {\mB}_k^{-1}\left(\alpha\mT^\mathcal H\vz+(1-\alpha) \mrm{vec}\left(\mathcal L(\mP,\mQ)\right)\right).$$
Since \eqref{eq:TVWavReco:WPM:DualFormI} is convex in $\vx$ and concave in $(\vz,\mP,\mQ)$, we interchange the minimum and maximum and then get
\begin{equation}
\begin{array}{cl}
 \max\limits_{\substack{\vz\in\mathcal Z\\(\mP,\mQ)\in\mathcal P}}\min\limits_{\vx\in\mathbb C^N}&\left\|\vx-\vw_k(\vz,\mP,\mQ)\right\|_{ {\mB}_k}^2-\left\|\vw_k(\vz,\mP,\mQ)
\right\|_{{\mB}_k}^2.
\label{eq:TVWavReco:WPM:DualFormII} 
\end{array}
\end{equation}

Note that $\vx$ only appears in the first term of \eqref{eq:TVWavReco:WPM:DualFormII} so that the optimal solution of the minimum part is 
\begin{equation}
	\vx^*=\vw_k(\vz,\mP,\mQ).
	\label{eq:TVWavReco:WPM:Dual:OptSolutionInner}
\end{equation}
Substituting \eqref{eq:TVWavReco:WPM:Dual:OptSolutionInner} into \eqref{eq:TVWavReco:WPM:DualFormII},
we get the following dual problem that contains only unknown dual variables $(\vz,\mP,\mQ)$
\begin{equation}
\begin{array}{rl}
 (\vz^*,\mP^*,\mQ^*)=\argmin\limits_{\substack{\vz\in\mathcal Z,\\(\mP,\mQ)\in\mathcal P}}&\left\|\vw_k(\vz,\mP,\mQ)\right\|^2_{\mB_k}.
\label{eq:TVWavReco:WPM:DualFinal} 
\end{array}
\end{equation}
After solving \eqref{eq:TVWavReco:WPM:DualFinal},
the primal variable update is
$\vx_{k+1}=\vw_k(\vz^*,\mP^*,\mQ^*)$. This completes the proof.

\section{Proof of Lemma \ref{lemma:gradient:LipschitzConstant:dualPro}}
\label{App:proof:lemma:gradient:LipschitzConstant:dualPro}

Denote by $h(\vz,\mP,\mQ)\triangleq \left\|\vw_k(\vz,\mP,\mQ)\right\|^2_{{\mB}_k}$.
Applying the chain rule, we get  
$$
     \nabla h(\vz,\mP,\mQ)=-2\bar{\lambda} \bmat
 \alpha\mT\\
	(1-\alpha)\mathcal L^\mathcal{T}
 \emat \vw_k(\vz,\mP,\mQ).
$$
Now, we compute the Lipschitz constant of $h(\vz,\mP,\mQ)$. 
For every two pairs of $(\vz_1,\mP_1,\mQ_1)$ and $(\vz_2,\mP_2,\mQ_2)$, we have
$$
\begin{array}{l}
	\|\nabla h(\vz_1,\mP_1,\mQ_1)-\nabla h(\vz_2,\mP_2,\mQ_2)\|\\[6pt]
	~~~=2\bar{\lambda} ^2\bigg\|\bmat
 \alpha\mT\\
	(1-\alpha)\mathcal L^\mathcal{T}
 \emat \mB_k^{-1} \bmat
 \alpha\mT^\mathcal H&
	(1-\alpha)\mathcal L
 \emat \\[10pt]
 ~~~~~~~~~~~~~~~~\left[\left(\vz_1,\mP_1,\mQ_1\right)-\left(\vz_2,\mP_2,\mQ_2\right)\right]\bigg\|\\[10pt]
 
~~~\leq 2\bar{\lambda}^2\Big\|\alpha^2 \mT^\mathcal H\mT+(1-\alpha)^2\mathcal L^\mathcal T\mathcal L\Big\|\Big\|\mB_k^{-1}\Big\| \\ [10pt]
~~~~~~~~~~~~~\Big\| \left[\left(\vz_1,\mP_1,\mQ_1\right)-\left(\vz_2,\mP_2,\mQ_2\right)\right]\Big\|\\ [6pt]
~~~\leq 2\bar{\lambda}^2(\alpha^2\|\mT\|^2+(1-\alpha)^2\|\mathcal L\|^2)\sigma_{\mrm{min}} \\ [6pt]
~~~~~~~~~~~~\left\| \Big[\left(\vz_1,\mP_1,\mQ_1\right)-\left(\vz_2,\mP_2,\mQ_2\right)\Big]\right\|,
\end{array}
$$
where $\sigma_{\mrm{min}}$ is the smallest eigenvalue of $\mB_k$. With the proof of \cite[Lemma 4.2]{beck2009fastTV}, we know $\|\mathcal L\|=\sqrt{8}$ such that the Lipschitz constant of $h(\vz,\mP,\mQ)$ is $L_c=2\sigma_{\mrm{min}}
\bar{\lambda}^2(\alpha^2\|\mT\|^2+8(1-\alpha)^2)$. This completes the proof.

\section{Radial Acquisition MRI Reconstruction}

We provided the additional results with radial trajectory here. {\Cref{fig:RadialBrain:TVwavelet} describes the error maps of the zoom in regions of Fig. 6 in the main manuscript}. \Cref{fig:knee:radial:cost,fig:RadialKnee:wavelet} present the results of using wavelet regularizer with different algorithms for the knee image. Clearly, we observe similar trends as we tested on the brain image. 

\Cref{fig:knee:radial:WavTV:cost,fig:RadialKnee:WavTV} display the results of using wavelet and TV reguarizers for the knee image.
Our method again outperforms other algorithms.
{\cb \Cref{fig:knee:radial:WavTV:cost} shows that APM and CQNPM reduced the cost values
slightly faster than S-APM and S-CQNPM in terms of iterations and CPU time,
which is slightly different from our test on brain image.
However, it is unsurprising because S-APM and S-CQNPM solved the partial smoothing cost function,
rather than original one,
yet we still computed the original non-smooth cost.}

\input{Figs/KneeWaveFigRadial.tex}
\input{Figs/KneeWaveTVFigRadial.tex}
\section{Spiral Acquisition MRI Reconstruction}
We provided the reconstruction with spiral acquisition that used
$32$ interleaves, $1688$ readout points, and $12$ coils.

{\Mcb \Cref{fig:SpiralBrainHighSNR:TVwavelet} describes the error maps of the zoomed-in regions in Fig. 11 of the main manuscript}. \Cref{fig:brain:spiral:cost,fig:SpiralBrain:wavelet,fig:knee:spiral:cost,fig:SpiralKnee:wavelet} present the results of the brain and knee images with wavelet regularizer. \Cref{fig:brain:spiral:WavTV:cost,fig:SpiralBrain:WavTV} show the results of the brain image with wavelet and TV regularizers. {\Cref{fig:SpiralBrain:TVwavelet} describes the error maps of the zoom in regions of \Cref{fig:SpiralBrain:WavTV}}. {\Cref{fig:SpiralKneeHigh:WavTV} depicts the knee image with wavelet and TV regularizers on $30$dB input SNR.} Clearly, the trends are similar to the radial acquisition described in our paper. 

\input{Figs/BrainWaveFigSpiral.tex}
\input{Figs/KneeWaveFigSpiral.tex}

\input{Figs/BrainWaveTVFigSpiral.tex}

\section{Comparison with Plug-and-Play}

{\color{black}
This section compares the reconstruction performance between wavelet and total variation regularizers,
and Plug-and-Play (PnP) with a deep denoiser prior \cite{zhang2021plug} and BM3D denoiser \cite{dabov2007image}. 
To obtain the deep denoiser prior, we employed the DnCNN \cite{zhang2017beyond} denoiser,
training it on the brain dataset as described in \cite{aggarwal2018modl}. We trained on three different noise variance $\sigma=\{0.1,1,5\}$ yielding three different DnCNN denoisers. For testing, we selected one brain image from the test dataset
and generated the k-space data using the spiral trajectory outlined in the manuscript.
We then added complex i.i.d. Gaussian noise, resulting in a $\sim 30$dB data input SNR.}

{\color{black} In the reconstruction process, we applied 35 iterations for both APM and CQNPM. For the PnP, we continued the iterations until the change in PSNR between two consecutive iterations was less than $0.01$dB. For the PnP with the DnCNN denoiser, we chose the DnCNN trained with $\sigma=0.1$ as it yielded the best performance. \Cref{fig:SpiralKneeHighDeep:WavTV} presents the reconstructed images
that achieved the highest PSNR during the reconstruction process for each method.
Here,
PnP-BM3D and
the wavelet and total variation regularizers
led to higher SNR images than PnP with a deep denoiser prior.
As described in \cite{xu2020boosting},
the deep denoiser may suffer from a noise scaling problem,
which may explain its inferior SNR performance observed here.
Further investigation of this issue is reserved for future work.
}
\bibliographystyle{IEEEtran}
\bibliography{Refs}

\end{document}

%% file: defs.tex
\usepackage{subfigure}
\usepackage{hyperref}
\usepackage{graphicx}
\usepackage{url}
\usepackage{pslatex}
\usepackage{bm}
\usepackage[comma,numbers,square,sort&compress]{natbib}
\usepackage{amsmath,amssymb,amsthm}  
\usepackage{paralist}

\usepackage{algorithm}
\usepackage{algorithmic}      
\usepackage{multirow}
\renewcommand{\algorithmicrequire}{\textbf{Initialization:}}
\renewcommand{\algorithmicensure}{\textbf{Iteration:}}
\renewcommand{\algorithmiclastcon}{\textbf{Output:}} 

\newtheorem{lemma}{Lemma}

\newtheorem{theorem}{Theorem}

\newtheorem{prop}{Proposition}

\newtheorem{remark}{Remark}
\newcommand{\e}{\begin{equation}}
\newcommand{\ee}{\end{equation}}
\newcommand{\en}{\begin{equation*}}
\newcommand{\een}{\end{equation*}}
\newcommand{\eqn}{\begin{eqnarray}}
\newcommand{\eeqn}{\end{eqnarray}}
\newcommand{\bmat}{\begin{bmatrix}}
\newcommand{\emat}{\end{bmatrix}}
\newcommand{\BIT}{\begin{itemize}}
\newcommand{\EIT}{\end{itemize}}

\newcommand{\mrm}{\mathrm}
\newcommand{\mbb}{\mathbb}

\newcommand{\argmin}{\mathop{\rm argmin}}

\newcommand{\va}{\bm a}
\newcommand{\vb}{\bm b}

\newcommand{\vm}{\bm m}

\newcommand{\vq}{\bm q}

\newcommand{\vs}{\bm s}

\newcommand{\vu}{\bm u}
\newcommand{\vv}{\bm v}
\newcommand{\vw}{\bm w}
\newcommand{\vx}{\bm x}
\newcommand{\vy}{\bm y}
\newcommand{\vz}{\bm z}

\newcommand{\mA}{\bm A}
\newcommand{\mB}{\bm B}

\newcommand{\mD}{\bm D}

\newcommand{\mF}{\bm F}

\newcommand{\mH}{\bm H}
\newcommand{\mI}{\bm I}

\newcommand{\mP}{\bm P}
\newcommand{\mQ}{\bm Q}

\newcommand{\mS}{\bm S}
\newcommand{\mT}{\bm T}
\newcommand{\mU}{\bm U}

\newcommand{\mW}{\bm W}
\newcommand{\mX}{\bm X}

\newcommand{\mPhi}{\bm \Phi}

\newcommand{\SetL}{\mathcal L}

\newcommand{\SetP}{\mathcal P}

\newcommand{\SetT}{\mathcal T}

\newcommand{\SetZ}{\mathcal Z}

\newcounter{oursection}

%% file: Figs/GTFig.tex
\begin{figure}[!t]
\centering


\begin{tikzpicture}
    \begin{axis}[at={(0,0)},anchor = north west,
    xmin = 0,xmax = 216,ymin = 0,ymax = 75, width=0.6\textwidth,
        scale only axis,
        enlargelimits=false,
       axis line style={draw=none},
       tick style={draw=none},
        axis equal image,
        xticklabels={,,},yticklabels={,,},
        colormap={whiteblue}{color=(black) color=(white)}, colorbar horizontal,
    colorbar style={
        width= 3.5cm,
        height = 0.25cm,
        yshift = 0.5cm,
        xshift = 0.1cm,
        xticklabel style ={xshift=0.2cm,yshift=0.2cm},
        xtick={0,1},
        rotate=90,
        at={(0.55,0.95)}
    },   point meta min=0,
           point meta max=1]%
   \node[inner sep=0pt, anchor = south west] (GT_xy_1) at (0,0) {\includegraphics[width=0.18\textwidth]{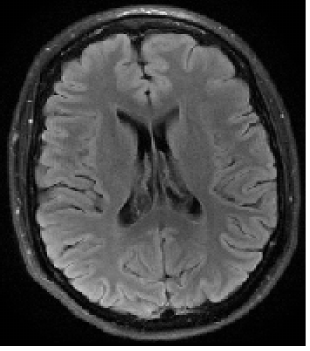}};
   
\node[inner sep=8pt, anchor = west] (GT_xy_2) at (GT_xy_1.east) {\includegraphics[width=0.2\textwidth,angle=270]{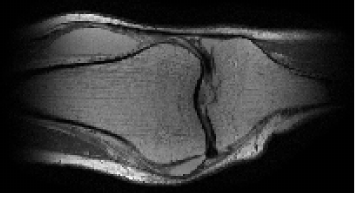}};
 \end{axis}  
  \node at (2,-4.1) {{\small (a) Brain}};
  \node at (4.5,-4.1) {{\small (b) Knee}};
\end{tikzpicture}
	\caption{The magnitude of the complex-valued ground truth images.}
	\label{fig:GT_imag}
\end{figure}

%% file: Figs/TrjFig.tex
\begin{figure}[!t]
\centering


\begin{tikzpicture}
    \begin{axis}[at={(0,0)},anchor = north west,
    xmin = 0,xmax = 216,ymin = 0,ymax = 75, width=0.6\textwidth,
        scale only axis,
        enlargelimits=false,
       axis line style={draw=none},
       tick style={draw=none},
        axis equal image,
   xticklabels={,,},yticklabels={,,},
        ]%
   \node[inner sep=0pt, anchor = south west] (GT_xy_1) at (0,0) {\includegraphics[width=0.2\textwidth]{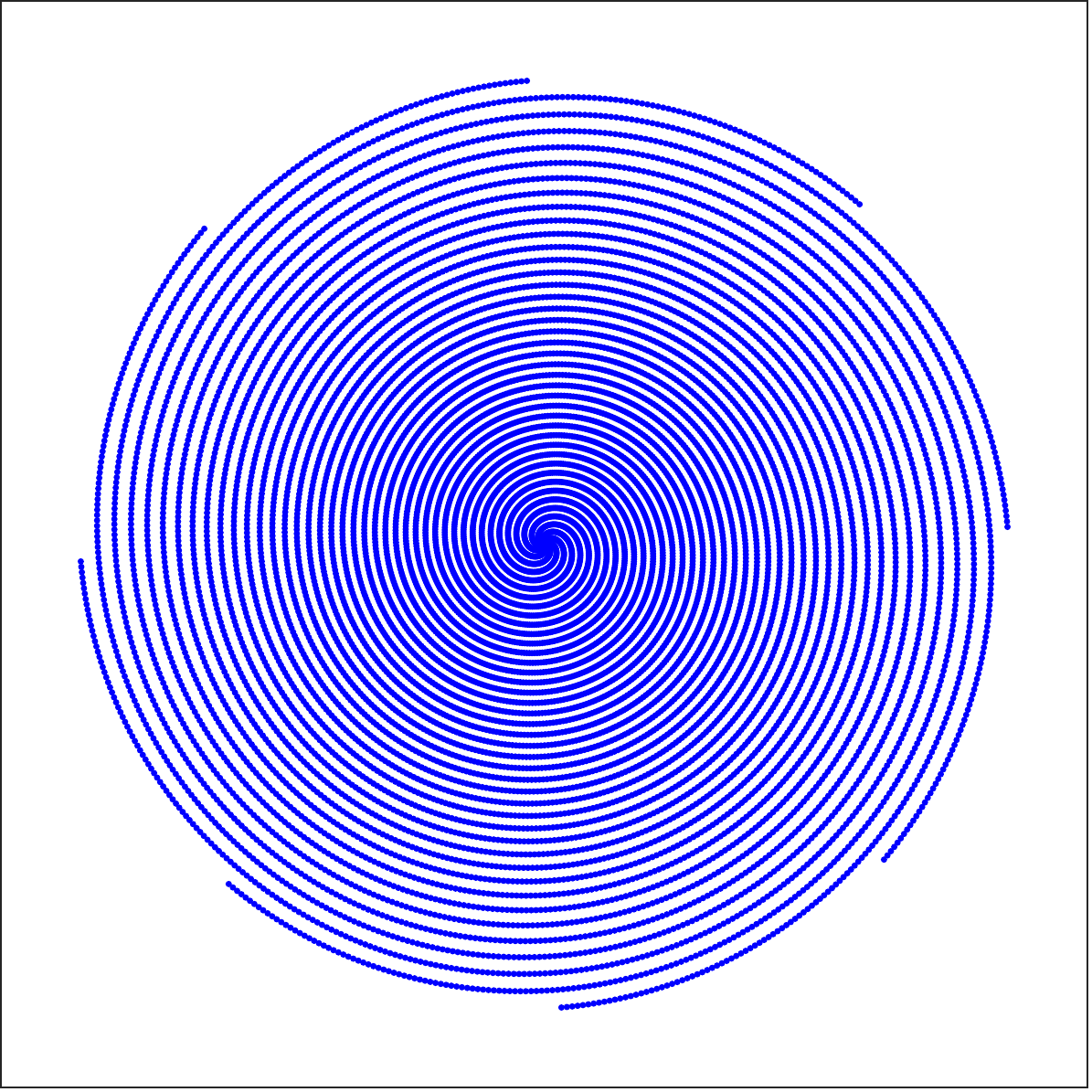}};
   
\node[inner sep=8pt, anchor = west] (GT_xy_2) at (GT_xy_1.east) {\includegraphics[width=0.2\textwidth]{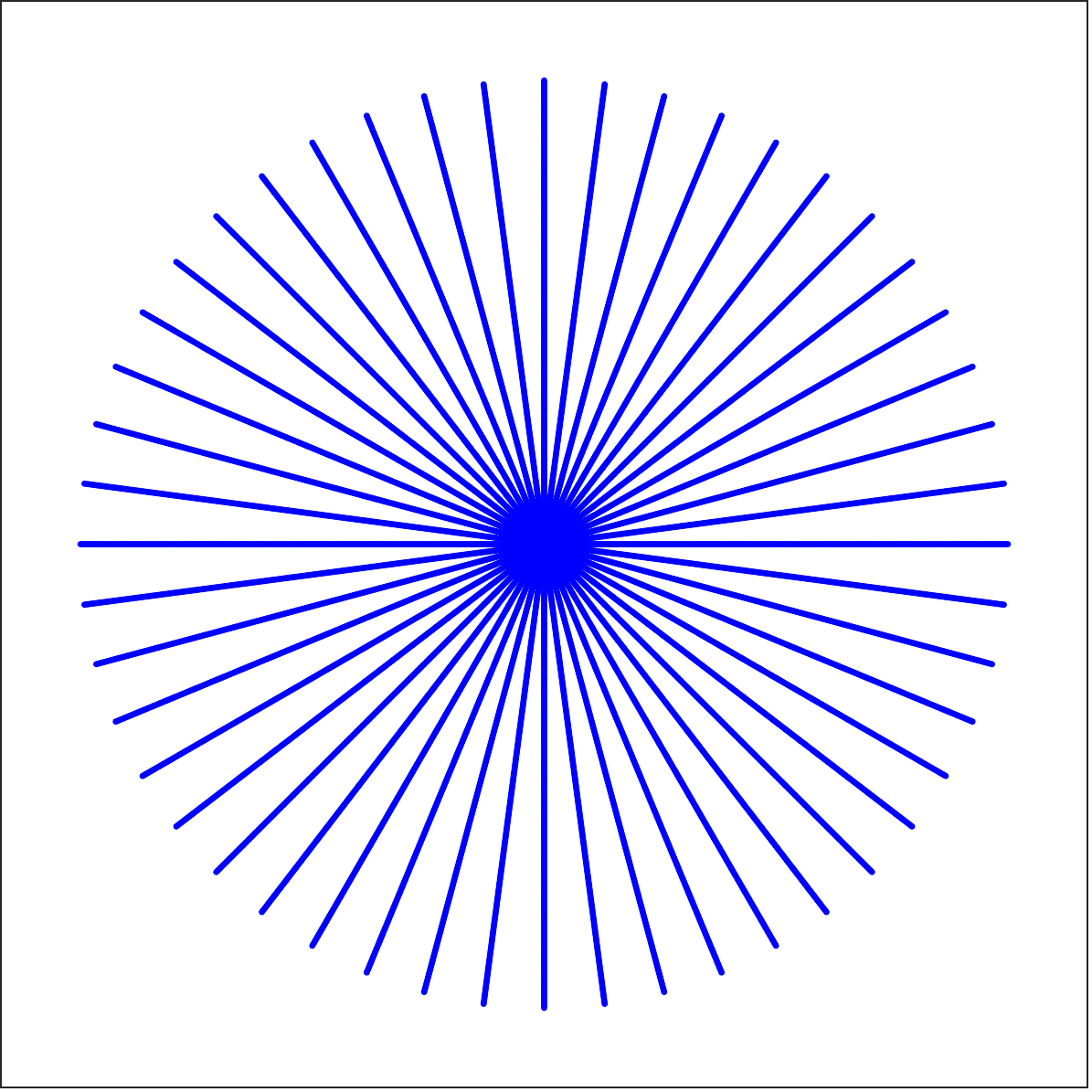}};
 \end{axis}  
  \node at (2,-4.1) {{\small (a) Spiral}};
  \node at (4.8,-4.1) {{\small (b) Radial}};
\end{tikzpicture}
	\caption{The non-Cartesian MRI trajectories used in this paper.}
	\label{fig:Trj_sampling}
\end{figure}

%% file: Figs/BrainWaveFigRadial.tex

\begin{figure}[!t]
   \centering
	\begin{tikzpicture}

 \draw [draw=blue] (2.6,0) rectangle (3.62,0.3);

 \draw [-to,color=blue](3,0.3) -- (3.93,2.2);
  \draw [-to,color=blue](3,0) -- (3.93,0);

 \draw [draw=blue] (2.75,-3.2) rectangle (3.62,-2.9);

 \draw [-to,color=blue](3.2,-2.9) -- (3.93,-1);

  \draw [-to,color=blue](3.2,-3.2) -- (3.93,-3.2);
 

 	\pgfplotsset{every axis legend/.append style={legend pos=south east,anchor=south east,font=\normalsize, legend cell align={left}},every tick label/.append style={font=\footnotesize}}
    \pgfplotsset{grid style={dotted, gray}}
 	\begin{groupplot}[enlargelimits=false,scale only axis, group style={group size=2 by 2,x descriptions at=edge bottom,group name=mygroup},
 	width=0.2*\textwidth,
 	height=0.12*\textwidth,
 	every axis/.append style={font=\small,title style={anchor=base,yshift=-1mm}, x label style={yshift = 0.5em}, y label style={yshift = -.5em}, grid = both}, legend style={at={(0.8,1)},anchor=north east,font =\footnotesize,text opacity = 1,fill opacity=0.6}
 	]
\nextgroupplot[ylabel={Cost},xlabel={Iteration},xmax = 16,ymax = 140,ymin=55,xtick={0,4,8,12,16},
xticklabels={0,4,8,12,16},tick align=inside,ytick={55,100,140},yticklabels={55,100,140},tick pos=left]

   \addplot[dotted,very thick,black,line width=1pt] table [search path={fig/Radial_results/Brain},x expr=\coordindex, y=PD_cost, col sep=comma] {Brain_Wav_PD.csv}; 
      \addplot[dashed,very thick,blue,line width=1pt] table [search path={fig/Radial_results/Brain},x expr=\coordindex, y=fista_cost, col sep=comma] {Brain_Wav_fista.csv};   
      \addplot[solid,very thick,red,line width=1pt] table [search path={fig/Radial_results/Brain},x expr=\coordindex, y=QNP_cost, col sep=comma] {Brain_Wav_QNP.csv};

    \legend{PD,APM,CQNPM};  

\nextgroupplot[ylabel={},xlabel={Iteration},xmin=12,xmax = 16,ymin=59,ymax = 62.5,xtick={12,14,16},
xticklabels={,14,16},tick align=inside,ytick={},yticklabels={},
tick pos=left,xshift={-7mm}]

   \addplot[dotted,very thick,black,line width=1pt] table [search path={fig/Radial_results/Brain},x expr=\coordindex, y=PD_cost, col sep=comma] {Brain_Wav_PD.csv}; 
   
     \addplot[dashed,very thick,blue,line width=1pt] table [search path={fig/Radial_results/Brain},x expr=\coordindex, y=fista_cost, col sep=comma] {Brain_Wav_fista.csv};   
     \addplot[solid,very thick,red,line width=1pt] table [search path={fig/Radial_results/Brain},x expr=\coordindex, y=QNP_cost, col sep=comma] {Brain_Wav_QNP.csv};
     
\nextgroupplot[xlabel={CPU Time (Seconds)},xmax = 6,ymax = 140,ymin=55,xtick={0,2,4,6},
xticklabels={0,2,4,6},xtick align=inside,
tick pos=left,ytick={55,100,140},yticklabels={55,100,140},ylabel={Cost}]

 \addplot[dotted,very thick,black,line width=1pt] table [search path={fig/Radial_results/Brain},x = PD_time, y=PD_cost, col sep=comma] {Brain_Wav_PD.csv}; 
   
\addplot[dashed,very thick,blue,line width=1pt] table [search path={fig/Radial_results/Brain},x = fista_time, y=fista_cost, col sep=comma] {Brain_Wav_fista.csv};
  
\addplot[solid,very thick,red,line width=1pt] table [search path={fig/Radial_results/Brain},x = QNP_time, y=QNP_cost, col sep=comma] {Brain_Wav_QNP.csv};

\nextgroupplot[xlabel={CPU Time (Seconds)},xmax = 6,xmin = 4.5,xtick={4.5,5,6},
xticklabels={,5,6},xtick align=inside,
tick pos=left,ytick={},yticklabels={},ymin=59,ymax = 62.5]

 \addplot[dotted,very thick,black,line width=1pt] table [search path={fig/Radial_results/Brain},x = PD_time, y=PD_cost, col sep=comma] {Brain_Wav_PD.csv}; 
 
\addplot[dashed,very thick,blue,line width=1pt] table [search path={fig/Radial_results/Brain},x = fista_time, y=fista_cost, col sep=comma] {Brain_Wav_fista.csv};

\addplot[solid,very thick,red,line width=1pt] table [search path={fig/Radial_results/Brain},x = QNP_time, y=QNP_cost, col sep=comma] {Brain_Wav_QNP.csv};

\end{groupplot}
\end{tikzpicture} 
\caption{\cb Cost values versus iteration (top) and CPU time (bottom) of the brain image with regularizer
$h(\vx) = \|\mT \vx\|_1$ and $\lambda=5\times10^{-4}$
for a left invertible wavelet transform $\mT$
with $5$ levels. Acquisition: radial trajectory with
$96$ projections, $512$ readout points, and $12$ coils.}
 \label{fig:brain:radial:cost}
\end{figure}

\begin{figure*}[!t]
	\centering

\begin{tikzpicture}
    \node (Brain_GT) at (-5,1.3) {\includegraphics[ width=0.18\textwidth]{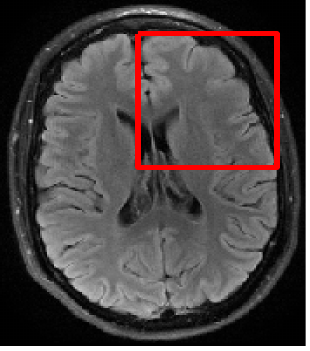}};
   \node at (-6.35,2.9) {\color{white} GT};

    \node (Brain_GT_Crop) at (-1.9,1.3) {\includegraphics[ width=0.15\textwidth]{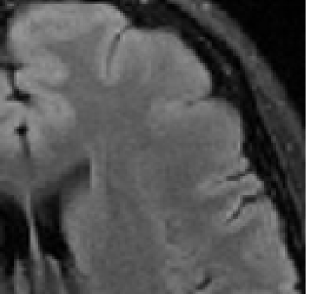}};

       \node at (-3.9,2.63) {\color{white} (a)};
   \node at (-0.8,2.35) {\color{white} (a)};

 	\pgfplotsset{every axis legend/.append style={legend pos=south east,anchor=south east,font=\normalsize, legend cell align={left}},every tick label/.append style={font=\footnotesize}}
    \pgfplotsset{grid style={dotted, gray}}
    
    \hspace{0.45cm}
 	
  \begin{groupplot}[enlargelimits=false,scale only axis, group style={group size=1 by 1,x descriptions at=edge bottom,group name=mygroup},
 	width=0.4*\textwidth,
 	height=0.18*\textwidth,
 	every axis/.append style={font=\small,title style={anchor=base,yshift=-1mm}, x label style={yshift = 0.5em}, y label style={yshift = -.9em}, grid = both}, legend style={at={(1,0.5)},anchor=north east}
 	]
     \nextgroupplot[ylabel={PSNR},xlabel={CPU Time (Seconds)},xmax = 6,ymax = 29,xtick={0,2,4,6},
xticklabels={0,2,4,6},tick align=inside,
tick pos=left,ytick={20,25,28.9},
yticklabels={20,25,28.9}]

     \addplot[solid,very thick,red,line width=1pt] table [search path={fig/Radial_results/Brain},x=QNP_time, y=QNP_PSNR, col sep=comma] {Brain_Wav_QNP.csv};

     \addplot[dashed,very thick,blue,line width=1pt] table [search path={fig/Radial_results/Brain},x=fista_time, y=fista_PSNR, col sep=comma] {Brain_Wav_fista.csv};   
     
\addplot[dotted,black,very thick,line width=1pt] table [search path={fig/Radial_results/Brain},x=PD_time, y=PD_PSNR, col sep=comma] {Brain_Wav_PD.csv};

    \addplot[mark=none, darkblue,dashed,line width=1pt]
    coordinates {(3.25,28.25) (6,28.25)};
   \addplot[mark=none, red,dashed,line width=1pt] coordinates {(3.25,18.3379697161268) (3.25,28.25)};
   
\legend{CQNPM,APM,PD}; 
 \end{groupplot}
\end{tikzpicture}

\vspace{-2.5cm}
\begin{tikzpicture}
  \hspace{1.8cm} 
   \begin{axis}[at={(0,0)}, anchor = north,ylabel = APM,
    xmin = 0,xmax = 216,ymin = 0,ymax = 75, width=1\textwidth,
        scale only axis,
        enlargelimits=false,
      y label style = {yshift = -0.2cm,xshift=-1.5cm},
        axis line style={draw=none},
        tick style={draw=none},
        axis equal image,
        xticklabels={,,},yticklabels={,,},
        ]
        
    \node[inner sep=0pt, anchor = south west] (FISTA_1) at (0,0) {\includegraphics[ width=0.18\textwidth]{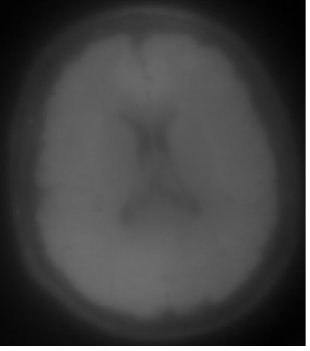}};
    
    \node at (8,41.5) {\color{white} $\text{iter.}=3$};
    \node at (9,3) {\color{red} $20.85$dB};
    
    \node[inner sep=0pt,anchor = west] (FISTA_2) at (FISTA_1.east) {\includegraphics[ width=0.18\textwidth]{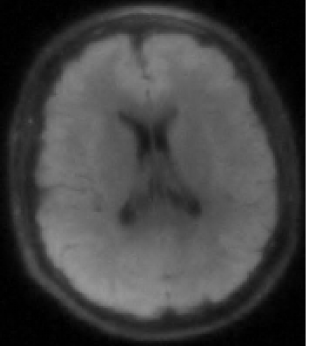}};
    
    \node at (43,41.5) {\color{white} $10$};
    \node at (48,3) {\color{red} $25.42$dB};
    
    \node[inner sep=0pt, anchor = west] (FISTA_3) at (FISTA_2.east) {\includegraphics[ width=0.18\textwidth]{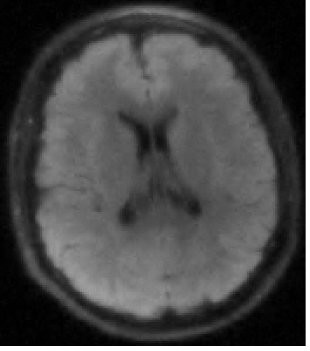}};
    
    \node at (81,41.5) {\color{white} $13$};
    \node at (87,3) {\color{red} $26.75$dB};

    \node[inner sep=0pt, anchor = west] (FISTA_4) at (FISTA_3.east) {\includegraphics[ width=0.18\textwidth]{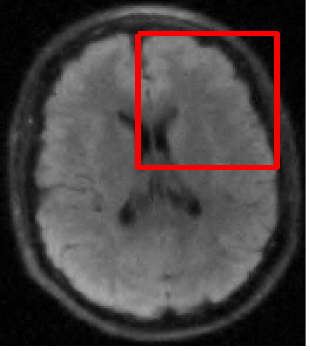}};
    \node at (120,41.5) {\color{white} $16$};
    \node at (125,3) {\color{red} $27.85$dB};
    \node at (150,37) {\Large\color{white} I};
    \end{axis}
    \begin{axis}[at={(FISTA_1.south west)},anchor = north west,ylabel = CQNPM,
    xmin = 0,xmax = 216,ymin = 0,ymax = 75, width=1\textwidth,
        scale only axis,
        enlargelimits=false,
        yshift = 2.6cm,
        y label style = {yshift = -0.2cm,xshift=-1.3cm},
       axis line style={draw=none},
       tick style={draw=none},
        axis equal image,
        xticklabels={,,},yticklabels={,,}
       ]
   \node[inner sep=0pt, anchor = south west] (QN_1) at (0,0) {\includegraphics[ width=0.18\textwidth]{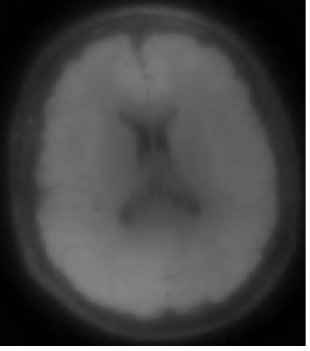}};
    \node at (9,3) {\color{red} $23.06$dB};
    
    \node[inner sep=0pt, anchor = west] (QN_2) at (QN_1.east) {\includegraphics[ width=0.18\textwidth]{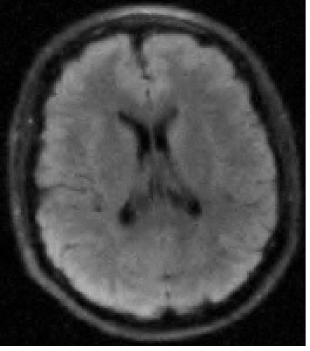}};
    \node at (48,3) {\color{red} $28.48$dB};
    
    \node[inner sep=0pt, anchor = west] (QN_3) at (QN_2.east) {\includegraphics[ width=0.18\textwidth]{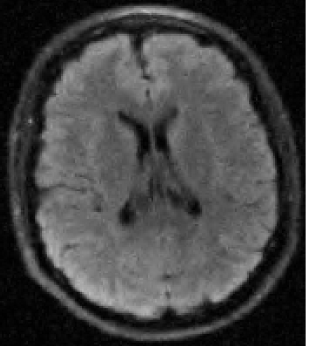}};
    \node at (87,3) {\color{red} $28.84$dB};

    \node[inner sep=0pt, anchor = west] (QN_4) at (QN_3.east) {\includegraphics[ width=0.18\textwidth]{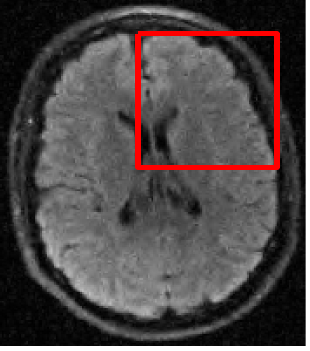}};
     \node at (125,3) {\color{red} $28.59$dB};  
    \node at (149,37) {\Large\color{white} II};
 \end{axis}

\begin{axis}[at={(QN_1.south west)},anchor = north west,
    xmin = 0,xmax = 216,ymin = 0,ymax = 75, width=1\textwidth,
        scale only axis,
    enlargelimits=false,
    yshift = 3.65cm,
       axis line style={draw=none},
       tick style={draw=none},
        axis equal image,
        xticklabels={,,},yticklabels={,,}
       ]
       
   \node[inner sep=0pt,anchor = south west] (Resi_FISTA) at (0,0) {\includegraphics[ width=0.15\textwidth]{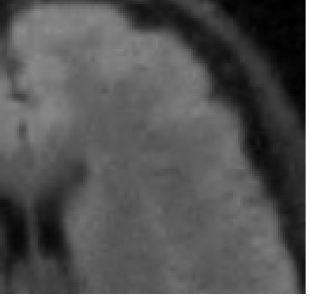}};    

    \node[inner sep=8pt, anchor = west] (Resi_FISTA_1) at (Resi_FISTA.east) {\includegraphics[ width=0.15\textwidth]{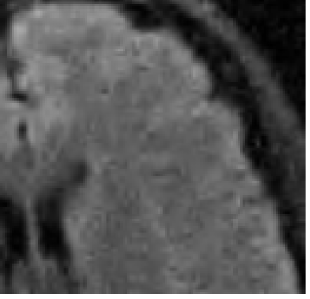}};

    \node[inner sep=15pt, anchor = west] (Resi_FISTA_2) at (Resi_FISTA_1.east) {\includegraphics[ width=0.15\textwidth]{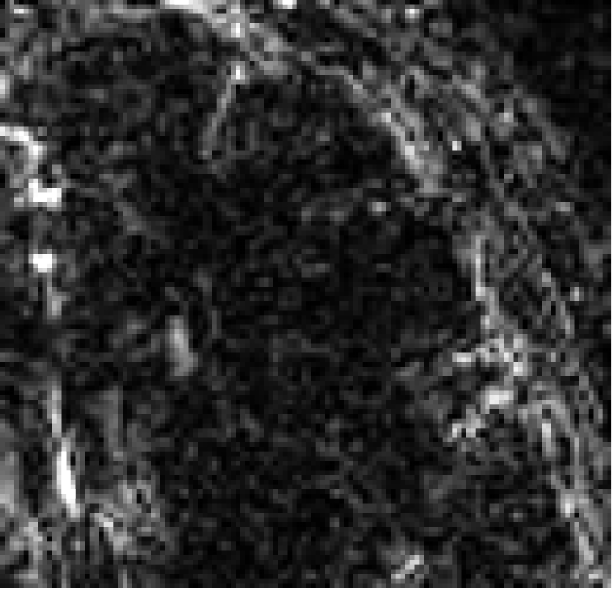}};
   
    \node[inner sep=5pt, anchor = west] (Resi_FISTA_3) at (Resi_FISTA_2.east) {\includegraphics[ width=0.15\textwidth]{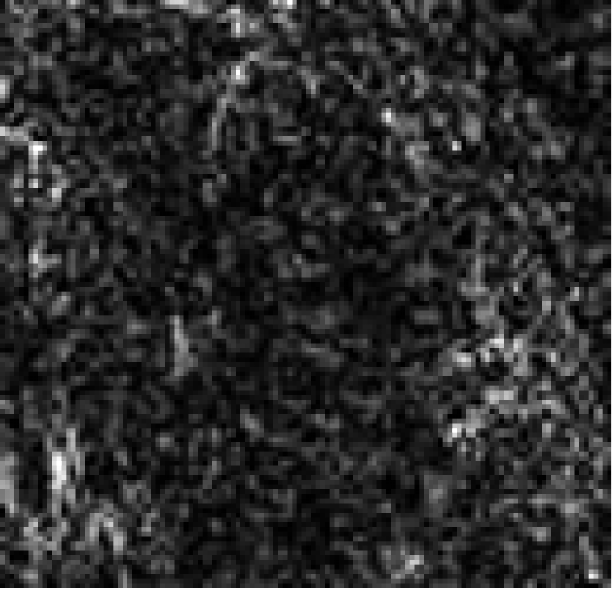}};
    \node at (30,28) {\Large\color{white} I};
    \node at (65,28) {\Large\color{white} II};
    \node at (108,28) {\Large\color{white} I};
    \node at (146,28) {\Large\color{white} II};
\node at (114.5,17) {\Large\color{red} $\times 5$};
\node at (156,17) {\Large\color{red} $\times 5$};
    
 \end{axis}
 
\end{tikzpicture} 
\caption{{First row: the ground truth image and PSNR values versus CPU time; second to third row: the reconstructed brain images at $3$, $10$, $13$, and $16$th iteration with
\Cref{fig:brain:radial:cost} setting; 
 fourth row: the zoomed-in regions
and the corresponding error maps ($\times 5$)
of the $16$th iteration reconstructed images.}
}
\label{fig:RadialBrain:wavelet}
\end{figure*}


%% file: Figs/BrainWaveTVFigRadial.tex
\begin{figure}[!t]
   \centering
	\begin{tikzpicture}

 \draw [draw=blue] (2.7,0) rectangle (3.62,0.3);

 \draw [-to,color=blue](3,0.3) -- (3.92,2.2);

 \draw [-to,color=blue](3,0) -- (3.98,0);

 \draw [draw=blue] (2.9,-3.2) rectangle (3.62,-2.9);

 \draw [-to,color=blue](3.2,-2.9) -- (3.92,-0.98);

  \draw [-to,color=blue](3.2,-3.2) -- (3.95,-3.2);
 
 

 	\pgfplotsset{every axis legend/.append style={legend pos=south east,anchor=south east,font=\normalsize, legend cell align={left}},every tick label/.append style={font=\footnotesize}}
    \pgfplotsset{grid style={dotted, gray}}
 	\begin{groupplot}[enlargelimits=false,scale only axis, group style={group size=2 by 2,x descriptions at=edge bottom,group name=mygroup},
 	width=0.2*\textwidth,
 	height=0.12*\textwidth,
 	every axis/.append style={font=\small,title style={anchor=base,yshift=-1mm}, x label style={yshift = 0.5em}, y label style={yshift = -.5em}, grid = both}, legend style={at={(0.7,1)},anchor=north east,font=\tiny,text opacity = 1,fill opacity=0.6}
 	]
  
\nextgroupplot[ylabel={Cost},xlabel={Iteration},xmax = 16,ymin=55,ymax = 140,xtick={0,4,8,12,16},
xticklabels={0,4,8,12,16},tick align=inside,ytick={55,100,140},
yticklabels={55,100,140},
tick pos=left,mark repeat = 5,mark size = 3pt]

 \addplot[dotted,very thick,black,line width=1pt] table [search path={fig/Radial_results/Brain},x expr=\coordindex, y=PD_cost, col sep=comma] {Brain_WavTV_PD.csv};

  \addplot[dash pattern=on 2pt off 3pt on 4pt off 4pt,very thick,goldenbrown,line width=1pt] table [search path={fig/Radial_results/Brain},x expr=\coordindex, y=ADMM_cost, col sep=comma] {Brain_WavTV_ADMM.csv};
  
  \addplot[dashed,very thick,blue,line width=1pt] table [search path={fig/Radial_results/Brain},x expr=\coordindex, y=fista_cost, col sep=comma] {Brain_WavTV_fista.csv};  

\addplot[dashed,mark=star,very thick,darkblue,line width=1pt] table [search path={fig/Radial_results/Brain},x expr=\coordindex, y=fista_Smooth_cost, col sep=comma] {Brain_WavTV_fistaSmooth.csv}; 

\addplot[solid,very thick,red,line width=1pt] table [search path={fig/Radial_results/Brain},x expr=\coordindex, y=QNP_cost, col sep=comma] {Brain_WavTV_QNP.csv};

\addplot[solid,mark=star,very thick,red,line width=1pt] table [search path={fig/Radial_results/Brain},x expr=\coordindex, y=QNP_Smooth_cost, col sep=comma] {Brain_WavTV_QNPSmooth.csv};

\legend{PD,ADMM,APM,S-APM,CQNPM,S-CQNPM};

\nextgroupplot[ylabel={},xlabel={Iteration},xmin=12,xmax = 16,ymin=59,ymax = 62.8,xtick={12,14,16},
xticklabels={,14,16},tick align=inside,ytick={},yticklabels={},
tick pos=left,xshift={-7mm}]

 \addplot[dotted,very thick,black,line width=1pt] table [search path={fig/Radial_results/Brain},x expr=\coordindex, y=PD_cost, col sep=comma] {Brain_WavTV_PD.csv};

 \addplot[dash pattern=on 2pt off 3pt on 4pt off 4pt,very thick,goldenbrown,line width=1pt] table [search path={fig/Radial_results/Brain},x expr=\coordindex, y=ADMM_cost, col sep=comma] {Brain_WavTV_ADMM.csv};

\addplot[dashed,mark=star,very thick,darkblue,line width=1pt] table [search path={fig/Radial_results/Brain},x expr=\coordindex, y=fista_Smooth_cost, col sep=comma] {Brain_WavTV_fistaSmooth.csv}; 

\addplot[solid,very thick,red,line width=1pt] table [search path={fig/Radial_results/Brain},x expr=\coordindex, y=QNP_cost, col sep=comma] {Brain_WavTV_QNP.csv};
\addplot[solid,mark=star,very thick,red,line width=1pt] table [search path={fig/Radial_results/Brain},x expr=\coordindex, y=QNP_Smooth_cost, col sep=comma] {Brain_WavTV_QNPSmooth.csv};

\nextgroupplot[xlabel={CPU Time (Seconds)},xmax = 12,ymax = 140,ymin=55,xtick={0,4,8,12},
xticklabels={0,4,8,12},xtick align=inside,
ylabel={Cost},ytick={55,100,140},
yticklabels={55,100,140},mark repeat = 5,mark size = 3pt,ylabel={Cost}]

 \addplot[dotted,very thick,black,line width=1pt] table [search path={fig/Radial_results/Brain},x = PD_time, y=PD_cost, col sep=comma] {Brain_WavTV_PD.csv};
 
 \addplot[dash pattern=on 2pt off 3pt on 4pt off 4pt,very thick,goldenbrown,line width=1pt] table [search path={fig/Radial_results/Brain},x = ADMM_time, y=ADMM_cost, col sep=comma] {Brain_WavTV_ADMM.csv};

\addplot[dashed,very thick,blue,line width=1pt] table [search path={fig/Radial_results/Brain},x = fista_time, y=fista_cost, col sep=comma] {Brain_WavTV_fista.csv};
\addplot[dashed,mark=star,very thick,darkblue,line width=1pt] table [search path={fig/Radial_results/Brain},x =fista_Smooth_time, y=fista_Smooth_cost, col sep=comma] {Brain_WavTV_fistaSmooth.csv}; 

\addplot[solid,very thick,red,line width=1pt] table [search path={fig/Radial_results/Brain},x = QNP_time, y=QNP_cost, col sep=comma] {Brain_WavTV_QNP.csv};

\addplot[solid,mark=star,very thick,red,line width=1pt] table [search path={fig/Radial_results/Brain},x = QNP_Smooth_time, y=QNP_Smooth_cost, col sep=comma] {Brain_WavTV_QNPSmooth.csv};

\nextgroupplot[xlabel={CPU Time (Seconds)},xmax = 12,xmin = 10,xtick={10,11,12},
xticklabels={,11,12},xtick align=inside,
tick pos=left,ytick={},yticklabels={},ymin=59,ymax = 61.05]

 \addplot[dotted,very thick,black,line width=1pt] table [search path={fig/Radial_results/Brain},x = PD_time, y=PD_cost, col sep=comma] {Brain_WavTV_PD.csv};
 
 \addplot[dash pattern=on 2pt off 3pt on 4pt off 4pt,very thick,goldenbrown,line width=1pt] table [search path={fig/Radial_results/Brain},x = ADMM_time, y=ADMM_cost, col sep=comma] {Brain_WavTV_ADMM.csv};

\addplot[dashed,very thick,blue,line width=1pt] table [search path={fig/Radial_results/Brain},x = fista_time, y=fista_cost, col sep=comma] {Brain_WavTV_fista.csv};

\addplot[dashed,mark=star,very thick,darkblue,line width=1pt] table [search path={fig/Radial_results/Brain},x =fista_Smooth_time, y=fista_Smooth_cost, col sep=comma] {Brain_WavTV_fistaSmooth.csv}; 

\addplot[solid,very thick,red,line width=1pt] table [search path={fig/Radial_results/Brain},x = QNP_time, y=QNP_cost, col sep=comma] {Brain_WavTV_QNP.csv};

\addplot[solid,mark=star,very thick,red,line width=1pt] table [search path={fig/Radial_results/Brain},x = QNP_Smooth_time, y=QNP_Smooth_cost, col sep=comma] {Brain_WavTV_QNPSmooth.csv};

\end{groupplot}
\end{tikzpicture} 
\caption{\cb Cost values versus iteration (top) and CPU time (bottom) of the brain image with regularizer
$h(\vx)=\alpha \|\mT\vx\|_1+(1-\alpha)\mrm{TV}(\vx)$ and same acquisition as \Cref{fig:brain:radial:cost}.
The parameters were $\lambda = 6\times 10^{-4}$ and $\alpha = \frac{1}{6}$.}
 \label{fig:brain:radial:WavTV:cost}
\end{figure}

\begin{figure*}[!t]
	\centering
\begin{tikzpicture}

    \node (Brain_GT) at (-5,1.3) {\includegraphics[ width=0.18\textwidth]{fig/CropSquare_Brain_GT.pdf}};
   \node at (-6.35,2.9) {\color{white} GT};

    \node (Brain_GT_Crop) at (-1.9,1.3) {\includegraphics[ width=0.15\textwidth]{fig/Crop_Brain_GT.pdf}};
    
    \node at (-3.9,2.63) {\color{white} (a)};

   \node at (-0.8,2.35) {\color{white} (a)};

 	\pgfplotsset{every axis legend/.append style={legend pos=south east,anchor=south east,font=\normalsize, legend cell align={left}}}
    \pgfplotsset{grid style={dotted, gray}}
    
    \hspace{0.45cm}
    
 	\begin{groupplot}[enlargelimits=false,scale only axis, group style={group size=1 by 1,x descriptions at=edge bottom,group name=mygroup},
 	width=0.4*\textwidth,
 	height=0.19*\textwidth,
 	every axis/.append style={font=\small,title style={anchor=base,yshift=-1mm}, x label style={yshift = 0.5em}, y label style={yshift = -.5em}, grid = both}, legend style={at={(1,0.65)},anchor=north east,font =\tiny,text opacity = 1,fill opacity=0.6}
 	]
  
\nextgroupplot[ylabel={PSNR},xlabel={CPU Time (Seconds)},xmax = 12,ymax = 30.5,xtick={0,4,8,12},
xticklabels={0,4,8,12},tick align=inside,
tick pos=left,mark repeat = 5,mark size = 3pt]

\addplot[solid,mark=star,very thick,red,line width=1pt] table [search path={fig/Radial_results/Brain},x =QNP_Smooth_time, y=QNP_Smooth_PSNR, col sep=comma] {Brain_WavTV_QNPSmooth.csv};

\addplot[solid,very thick,red,line width=1pt] table [search path={fig/Radial_results/Brain},x =QNP_time, y=QNP_PSNR, col sep=comma] {Brain_WavTV_QNP.csv};

\addplot[dashed,very thick,blue,line width=1pt] table [search path={fig/Radial_results/Brain},x =fista_time, y=fista_PSNR, col sep=comma] {Brain_WavTV_fista.csv};  
    \addplot[dashed,mark=star,very thick,darkblue,line width=1pt] table [search path={fig/Radial_results/Brain},x =fista_Smooth_time, y=fista_Smooth_PSNR, col sep=comma] {Brain_WavTV_fistaSmooth.csv}; 
    
 \addplot[dotted,very thick,black,line width=1pt] table [search path={fig/Radial_results/Brain},x = PD_time, y=PD_PSNR, col sep=comma] {Brain_WavTV_PD.csv};
 
 \addplot[dash pattern=on 2pt off 3pt on 4pt off 4pt,very thick,goldenbrown,line width=1pt] table [search path={fig/Radial_results/Brain},x = ADMM_time, y=ADMM_PSNR, col sep=comma] {Brain_WavTV_ADMM.csv};
    
     \legend{S-CQNPM,CQNPM,APM,S-APM,PD,ADMM};
     
    \addplot[mark=none, darkblue,dashed,line width=1pt]
    coordinates {(4.8,28.7) (12,28.7)};
   \addplot[mark=none, red,dashed,line width=1pt] coordinates {(4.8, 18.3395693080118) (4.8,28.7)};
   
   \addplot[mark=none, red,dashed,line width=1pt] coordinates {(6.3, 18.3395693080118) (6.3,28.7)};

   
 \end{groupplot}
\end{tikzpicture}

\vspace{-2cm}

\begin{tikzpicture}

 \hspace{0.8cm} 
 
     \begin{axis}[at={(0,0)},anchor = north west, ylabel = PD,
    xmin = 0,xmax = 216,ymin = 0,ymax = 75, width=1\textwidth,
        scale only axis,
        enlargelimits=false,
      y label style = {yshift = -0.2cm,xshift=-1.5cm},
        axis line style={draw=none},
        tick style={draw=none},
        axis equal image,
        xticklabels={,,},yticklabels={,,},
        ]
    \node[inner sep=0pt, anchor = south west] (PD_1) at (0,0) {\includegraphics[ width=0.18\textwidth]{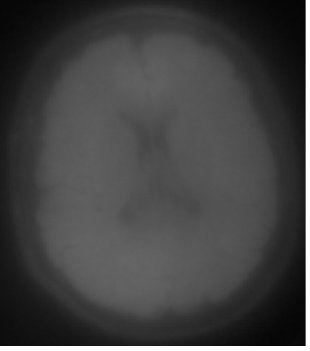}};
    
    \node at (8,41.5) {\color{white} $\text{iter.}=3$};
    \node at (9,3) {\color{red} $19.97$dB};
    

    \node[inner sep=0pt, anchor = west] (PD_2) at (PD_1.east) {\includegraphics[ width=0.18\textwidth]{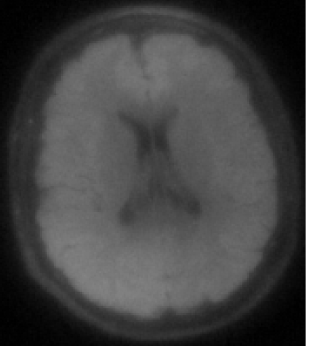}};
    
    \node at (42,41.5) {\color{white} $10$};
    \node at (47,3) {\color{red} $23.76$dB};
    
    \node[inner sep=0pt, anchor = west] (PD_3) at (PD_2.east) {\includegraphics[ width=0.18\textwidth]{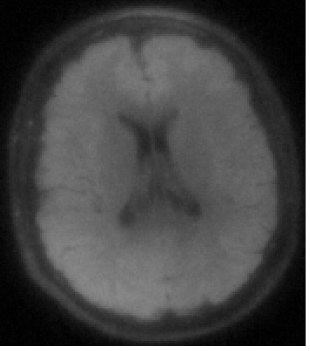}};
    
    \node at (82,41.5) {\color{white} $13$};
    \node at (87,3) {\color{red} $24.48$dB};

    \node[inner sep=0pt, anchor = west] (PD_4) at (PD_3.east) {\includegraphics[ width=0.18\textwidth]{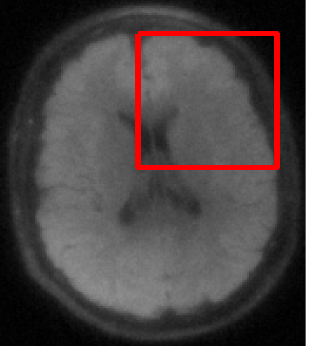}};
    \node at (120,41.5) {\color{white} $16$};
    \node at (125,3) {\color{red} $25.03$dB};
    \node[inner sep=0pt, anchor = west] (PD_5_crop) at (PD_4.east) {\includegraphics[ width=0.15\textwidth]{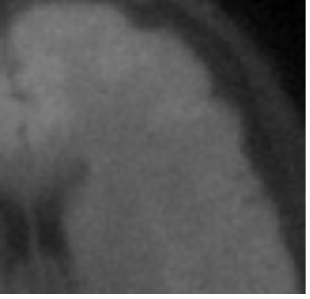}};

   \node at (150,37) {\Large\color{white} I};
   \node at (186,34) {\Large\color{white} I};


    \end{axis}

     \begin{axis}[at={(PD_1.south west)},anchor = north west, ylabel = APM,
    xmin = 0,xmax = 216,ymin = 0,ymax = 75, width=1\textwidth,
        scale only axis,
        enlargelimits=false,
        yshift=2.6cm,
      y label style = {yshift = -0.2cm,xshift=-1.3cm},
        axis line style={draw=none},
        tick style={draw=none},
        axis equal image,
        xticklabels={,,},yticklabels={,,},
        ]
    \node[inner sep=0pt, anchor = south west] (FISTA_1) at (0,0) {\includegraphics[ width=0.18\textwidth]{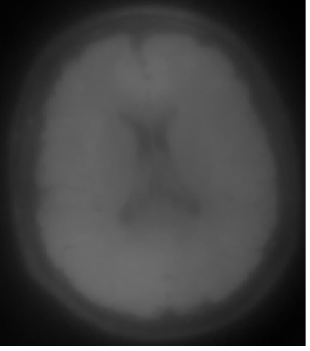}};
    
    \node at (9,3) {\color{red} $20.86$dB};
    

    \node[inner sep=0pt, anchor = west] (FISTA_2) at (FISTA_1.east) {\includegraphics[ width=0.18\textwidth]{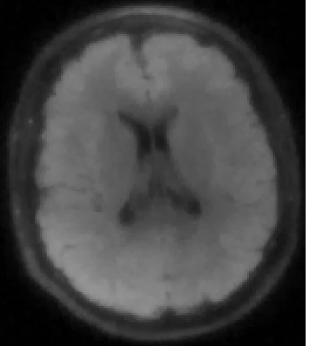}};
    
    \node at (48,3) {\color{red} $25.55$dB};
    
    \node[inner sep=0pt, anchor = west] (FISTA_3) at (FISTA_2.east) {\includegraphics[ width=0.18\textwidth]{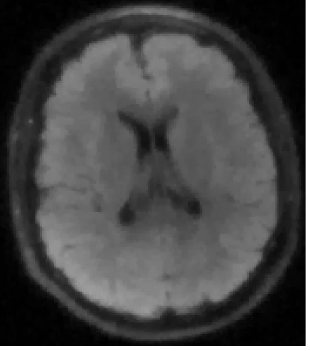}};
    
    \node at (87,3) {\color{red} $26.98$dB};

    \node[inner sep=0pt, anchor = west] (FISTA_4) at (FISTA_3.east) {\includegraphics[ width=0.18\textwidth]{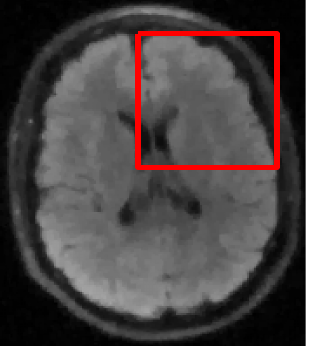}};
    \node at (125,3) {\color{red} $28.24$dB};

    \node[inner sep=0pt, anchor = west] (CropFISTA_5) at (FISTA_4.east) {\includegraphics[ width=0.15\textwidth]{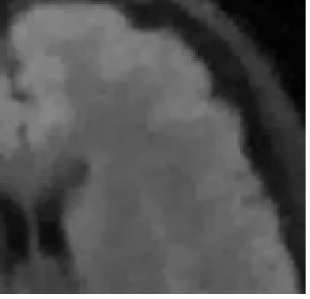}};

   \node at (150,37) {\Large\color{white} II};
   \node at (186,34) {\Large\color{white} II};
   

    \end{axis}

\begin{axis}[at={(FISTA_1.south west)},anchor =  north west, ylabel = S-APM,
    xmin = 0,xmax = 216,ymin = 0,ymax = 75, width=1\textwidth,
        scale only axis,
        enlargelimits=false,
        yshift = 2.6cm,
        y label style = {yshift = -0.2cm,xshift=-1.3cm},
        axis line style={draw=none},
        tick style={draw=none},
        axis equal image,
        xticklabels={,,},yticklabels={,,},
        ]
    \node[inner sep=0pt, anchor = south west] (S_FISTA_1) at (0,0) {\includegraphics[ width=0.18\textwidth]{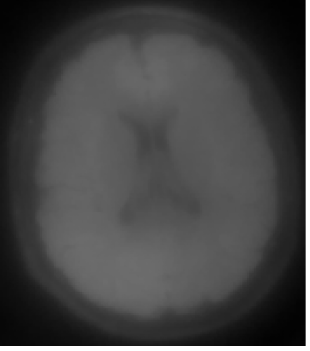}}; 
    \node at (9,3) {\color{red} $20.86$dB};
    \node[inner sep=0pt, anchor = west] (S_FISTA_2) at (S_FISTA_1.east) {\includegraphics[ width=0.18\textwidth]{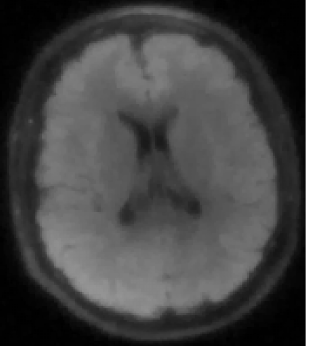}};

    \node at (49,3) {\color{red} $25.55$dB};

    \node[inner sep=0pt, anchor = west] (S_FISTA_3) at (S_FISTA_2.east) {\includegraphics[ width=0.18\textwidth]{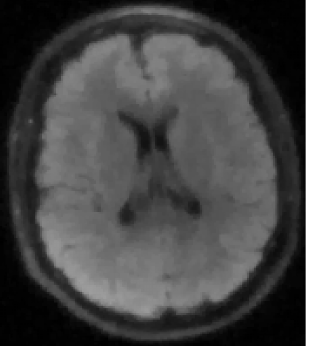}};
     \node at (87,3) {\color{red} $26.98$dB};
    \node[inner sep=0pt, anchor = west] (S_FISTA_4) at (S_FISTA_3.east) {\includegraphics[width=0.18\textwidth]{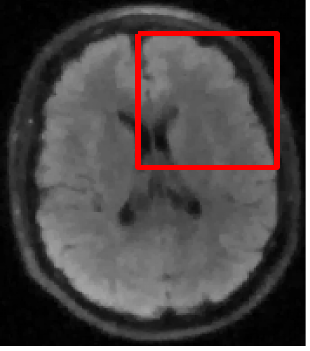}};
    \node at (125,3) {\color{red} $28.24$dB};

 \node[inner sep=0pt, anchor = west] (CropS_FISTA_4) at (S_FISTA_4.east) {\includegraphics[width=0.15\textwidth]{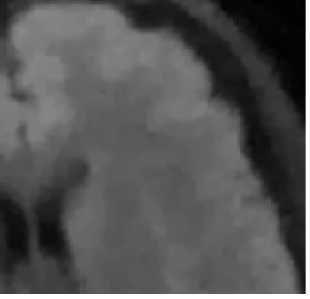}};

\node at (149,37) {\Large\color{white} III};
\node at (185,34) {\Large\color{white} III};
\end{axis}

\begin{axis}[at={(S_FISTA_1.south west)},anchor = north west,ylabel = CQNPM,
    xmin = 0,xmax = 216,ymin = 0,ymax = 75, width=1\textwidth,
        scale only axis,
        enlargelimits=false,
        yshift = 2.6cm,
        y label style = {yshift = -0.2cm,xshift=-1.3cm},
       axis line style={draw=none},
       tick style={draw=none},
        axis equal image,
        xticklabels={,,},yticklabels={,,}
       ]
   \node[inner sep=0pt, anchor = south west] (QN_1) at (0,0) {\includegraphics[ width=0.18\textwidth]{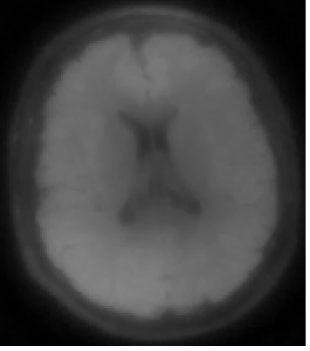}};
    \node at (9,3) {\color{red} $23.11$dB};
    
    \node[inner sep=0pt, anchor = west] (QN_2) at (QN_1.east) {\includegraphics[ width=0.18\textwidth]{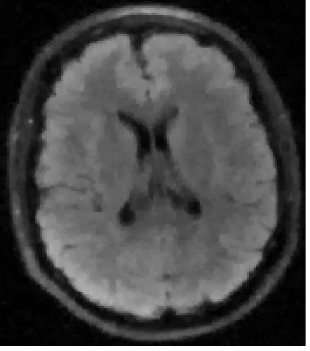}};
    \node at (49,3) {\color{red} $29.13$dB};
    
    \node[inner sep=0pt, anchor = west] (QN_3) at (QN_2.east) {\includegraphics[width=0.18\textwidth]{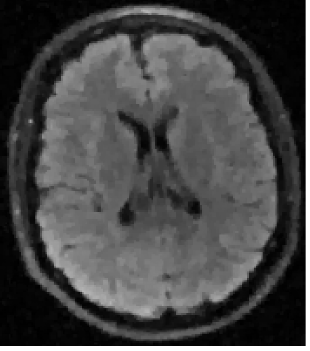}};
    \node at (87,3) {\color{red} $29.83$dB};

    \node[inner sep=0pt, anchor = west] (QN_4) at (QN_3.east) {\includegraphics[ width=0.18\textwidth]{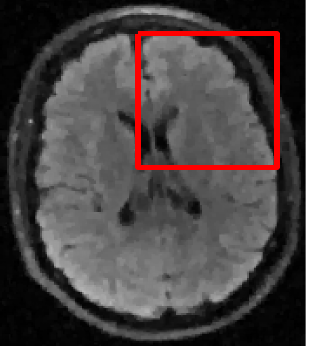}};
     \node at (125,3) {\color{red} $30.06$dB};  

    \node[inner sep=0pt, anchor = west] (CropQN_5) at (QN_4.east) {\includegraphics[width=0.15\textwidth]{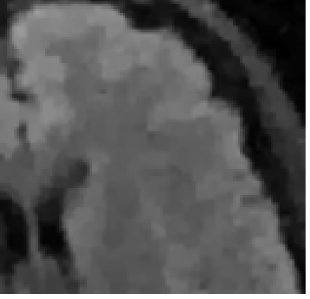}};

\node at (149,37) {\Large\color{white} IV};
\node at (185,34) {\Large\color{white} IV};    
 \end{axis}

    \begin{axis}[at={(QN_1.south west)},anchor = north west,ylabel = S-CQNPM,
    xmin = 0,xmax = 216,ymin = 0,ymax = 75, width=1\textwidth,
        scale only axis,
        enlargelimits=false,
          yshift = 2.6cm,
        y label style = {yshift = -0.2cm,xshift=-1.3cm},
       axis line style={draw=none},
       tick style={draw=none},
        axis equal image,
        xticklabels={,,},yticklabels={,,}
       ]
       
   \node[inner sep=0pt, anchor = south west] (S_QN_1) at (0,0) {\includegraphics[ width=0.18\textwidth]{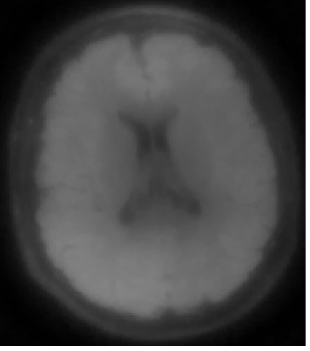}};
   
    \node at (9,3) {\color{red} $23.12$dB};

    \node[inner sep=0pt, anchor = west] (S_QN_2) at (S_QN_1.east) {\includegraphics[ width=0.18\textwidth]{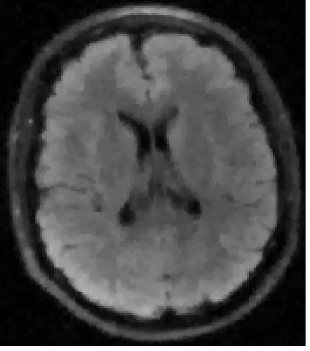}};
    \node at (49,3) {\color{red} $29.13$dB};
  
    \node[inner sep=0pt, anchor = west] (S_QN_3) at (S_QN_2.east) {\includegraphics[ width=0.18\textwidth]{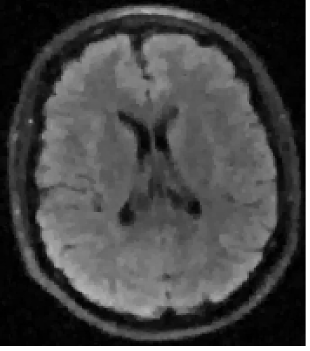}};
    
     \node at (87,3) {\color{red} $29.82$dB};
    \node[inner sep=0pt, anchor = west] (S_QN_4) at (S_QN_3.east) {\includegraphics[ width=0.18\textwidth]{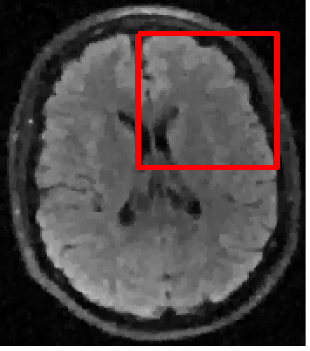}};
    \node at (125,3) {\color{red} $30.04$dB};  

       \node[inner sep=0pt, anchor = west] (CropS_QN_5) at (S_QN_4.east) {\includegraphics[ width=0.15\textwidth]{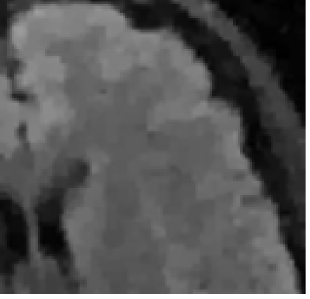}};

       \node at (149,37) {\Large\color{white} V};
\node at (185,34) {\Large\color{white} V};
 \end{axis}
 
\end{tikzpicture} 

\caption{
 First row: the ground truth image and PSNR values versus CPU time; second to sixth row: the reconstructed brain images  at $3$, $10$, $13$, and $16$th iteration with \Cref{fig:brain:radial:WavTV:cost} setting and the zoomed-in regions of the $16$th iteration reconstruction. We did not show the reconstructed image of ADMM since it yielded a much lower PSNR than other methods.
}
\label{fig:RadialBrain:WavTV}
\end{figure*}

%% file: Figs/BrainWavTVGamma.tex
\begin{figure}[!t]
	\centering
\begin{tikzpicture}

 \draw [draw=blue] (1.5,0) rectangle (3.8,0.23);

 \draw [-to,color=blue](3,0.23) -- (4,2.15);
 \draw [-to,color=blue](3,0) -- (4,0);
 
 	\pgfplotsset{every axis legend/.append style={legend pos=south east,anchor=south east,font=\normalsize, legend cell align={left}},every tick label/.append style={font=\footnotesize}}
    \pgfplotsset{grid style={dotted, gray}}
 	\begin{groupplot}[enlargelimits=false,scale only axis, group style={group size=2 by 2,x descriptions at=edge bottom,group name=mygroup},
   	width=0.21*\textwidth,
 	height=0.12*\textwidth,
 	every axis/.append style={font=\small,title style={anchor=base,yshift=-1mm}, x label style={yshift = 0.5em}, y label style={yshift = -.5em}, grid = both}, legend style={at={(0.8,1),font =\tiny,text opacity = 1,fill opacity=0.6},anchor=north east}
 	]
  
\nextgroupplot[ylabel={Cost},xlabel={Iteration},xmax = 16,ymax = 140,ymin=55,xtick={0,4,8,12,16},
xticklabels={0,4,8,12,16},ytick={55,90,140},yticklabels={55,90,140},tick align=inside,
tick pos=left,mark repeat =5,mark size = 3pt]

    \addplot[dotted,very thick,blue,line width=1pt] table [search path={fig/Radial_results/Brain},x expr=\coordindex, y=QNP_125_cost, col sep=comma] {Brain_WavTVGamma_QN.csv};   
     
     \addplot[solid,very thick,red,line width=1pt] table [search path={fig/Radial_results/Brain},x expr=\coordindex, y=QNP_17_cost, col sep=comma] {Brain_WavTVGamma_QN.csv};
     
    \addplot[dashdotted,magenta,very thick,black,line width=1pt] table [search path={fig/Radial_results/Brain},x expr=\coordindex, y=QNP_20_cost, col sep=comma] {Brain_WavTVGamma_QN.csv};
    
    \addplot[dashed,very thick,darkblue,line width=1pt] table [search path={fig/Radial_results/Brain},x expr=\coordindex, y=QNP_30_cost, col sep=comma] {Brain_WavTVGamma_QN.csv};
     \legend{$1.25$, $1.7$,$2$,$3$};

   \nextgroupplot[ylabel={},xlabel={Iteration},xmin=6,xmax = 16,ymin=59.15,ymax = 60.5,,xtick={6,8,11,14,16},
xticklabels={,8,11,14,16},ytick={},yticklabels={},tick align=inside,xshift=-0.8cm,
tick pos=left,mark repeat =5,mark size = 3pt]

     \addplot[dotted,very thick,blue,line width=1pt] table [search path={fig/Radial_results/Brain},x expr=\coordindex, y=QNP_125_cost, col sep=comma] {Brain_WavTVGamma_QN.csv};   
     
     \addplot[solid,very thick,red,line width=1pt] table [search path={fig/Radial_results/Brain},x expr=\coordindex, y=QNP_17_cost, col sep=comma] {Brain_WavTVGamma_QN.csv};
     
    \addplot[dashdotted,magenta,very thick,black,line width=1pt] table [search path={fig/Radial_results/Brain},x expr=\coordindex, y=QNP_20_cost, col sep=comma] {Brain_WavTVGamma_QN.csv};
    
    \addplot[dashed,very thick,darkblue,line width=1pt] table [search path={fig/Radial_results/Brain},x expr=\coordindex, y=QNP_30_cost, col sep=comma] {Brain_WavTVGamma_QN.csv};

 \end{groupplot}
\end{tikzpicture}

\caption{
{\cb Influence of $\gamma$ on the convergence of CQNPM.
Test on the brain image with wavelet and TV regularizers and radial acquisition. We set Max\_Iter$=20$.}
}
\label{fig:RadialBrain:WavTVGamma}
\end{figure}
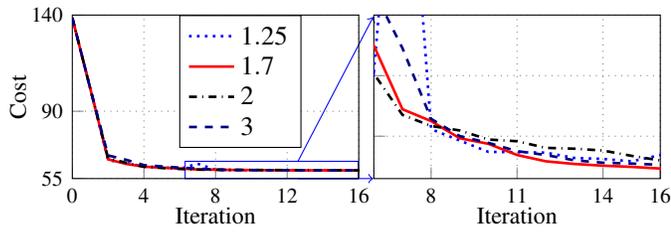

%% file: Figs/BrainWavTVInner.tex
\begin{figure}[!t]
	\centering
\begin{tikzpicture}

 \draw [draw=blue] (2.85,0) rectangle (3.8,0.3);

 \draw [-to,color=blue](3.2,0.3) -- (4,2.7);
 
 \draw [-to,color=blue](3.2,0) -- (4,0);

 \pgfplotsset{every axis legend/.append style={legend pos=south east,anchor=south east,font=\normalsize, legend cell align={left}},every tick label/.append style={font=\footnotesize}}
    \pgfplotsset{grid style={dotted, gray}}
 	\begin{groupplot}[enlargelimits=false,scale only axis, group style={group size=2 by 1,x descriptions at=edge bottom,group name=mygroup},
   	width=0.21*\textwidth,
 	height=0.15*\textwidth,
 	every axis/.append style={font=\small,title style={anchor=base,yshift=-1mm}, x label style={yshift = 0.5em}, y label style={yshift = -.5em}, grid = both}, legend style={at={(0.75,1),fontsize=\tiny,text opacity = 1,fill opacity=0.6},anchor=north east}
 	]
     \nextgroupplot[ylabel={Cost},xlabel={Iteration},xmax = 16,ymax = 140,ymin=55,xtick={0,4,8,12,16},
xticklabels={0,4,8,12,16},ytick={55,90,140},yticklabels={55,90,140},tick align=inside,
tick pos=left,mark repeat =5,mark size = 3pt]

     \addplot[dotted,very thick,blue,line width=1pt] table [search path={fig/Radial_results/Brain},x expr=\coordindex, y=QNP_1_cost, col sep=comma] {Brain_WavTVInner_QN.csv};

     \addplot[solid,very thick,red,line width=1pt] table [search path={fig/Radial_results/Brain},x expr=\coordindex, y=QNP_5_cost, col sep=comma] {Brain_WavTVInner_QN.csv};
     
     \addplot[very thick,dash pattern=on 2pt off 3pt on 4pt off 4pt,goldenbrown,line width=1pt] table [search path={fig/Radial_results/Brain},x expr=\coordindex, y=QNP_10_cost, col sep=comma] {Brain_WavTVInner_QN.csv};
          
    \addplot[dashdotted,magenta,very thick,black,line width=1pt] table [search path={fig/Radial_results/Brain},x expr=\coordindex, y=QNP_20_cost, col sep=comma] {Brain_WavTVInner_QN.csv};
    
    \addplot[dashed,very thick,darkblue,line width=1pt] table [search path={fig/Radial_results/Brain},x expr=\coordindex, y=QNP_50_cost, col sep=comma] {Brain_WavTVInner_QN.csv};
     \legend{$1$, $5$,$10$,$20$,$50$};

\nextgroupplot[ylabel={},xlabel={Iteration},xmin=12,xmax = 16,ymax = 139,tick align=inside,
tick pos=left,mark repeat =5,mark size = 3pt,xtick={12,14,16},
xticklabels={,14,16},ytick={},yticklabels={},ymin=59.2,ymax=59.66,xshift=-0.8cm]

     \addplot[dotted,very thick,blue,line width=1pt] table [search path={fig/Radial_results/Brain},x expr=\coordindex, y=QNP_1_cost, col sep=comma] {Brain_WavTVInner_QN.csv};   
     \addplot[solid,very thick,red,line width=1pt] table [search path={fig/Radial_results/Brain},x expr=\coordindex, y=QNP_5_cost, col sep=comma] {Brain_WavTVInner_QN.csv};
     
     \addplot[very thick,dash pattern=on 2pt off 3pt on 4pt off 4pt,goldenbrown,line width=1pt] table [search path={fig/Radial_results/Brain},x expr=\coordindex, y=QNP_10_cost, col sep=comma] {Brain_WavTVInner_QN.csv};
          
    \addplot[dashdotted,magenta,very thick,black,line width=1pt] table [search path={fig/Radial_results/Brain},x expr=\coordindex, y=QNP_20_cost, col sep=comma] {Brain_WavTVInner_QN.csv};
    
    \addplot[dashed,very thick,darkblue,line width=1pt] table [search path={fig/Radial_results/Brain},x expr=\coordindex, y=QNP_50_cost, col sep=comma] {Brain_WavTVInner_QN.csv};

 \end{groupplot}
\end{tikzpicture}

\caption{
{\cb Influence of Max\_Iter on the convergence of CQNPM.
Test on the same problem as \Cref{fig:RadialBrain:WavTVGamma} with $\gamma=1.7$.}
}
\label{fig:RadialBrain:WavTVInner}
\end{figure}
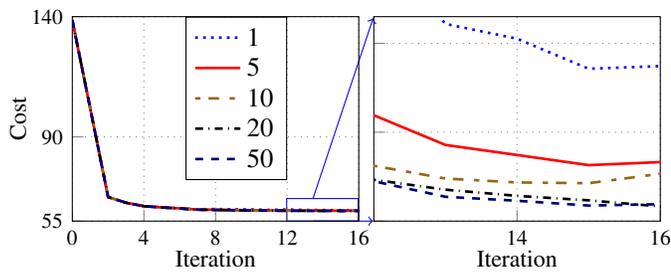

%% file: Figs/KneeWaveTVFigSpiral.tex
\begin{figure}[!t]
   \centering
	\begin{tikzpicture}

 \draw [draw=blue] (2.7,0) rectangle (3.62,0.5);

 \draw [-to,color=blue](3,0.5) -- (3.95,2.7);

 \draw [-to,color=blue](3,0) -- (3.95,0);

 \draw [draw=blue] (2.9,-3.3) rectangle (3.62,-3.75);

 \draw [-to,color=blue](3.2,-3.3) -- (3.92,-1);

 \draw [-to,color=blue](3.2,-3.75) -- (3.92,-3.73);
 
 

 	\pgfplotsset{every axis legend/.append style={legend pos=south east,anchor=south east,font=\normalsize, legend cell align={left}},every tick label/.append style={font=\footnotesize}}
    \pgfplotsset{grid style={dotted, gray}}
 	\begin{groupplot}[enlargelimits=false,scale only axis, group style={group size=2 by 2,x descriptions at=edge bottom,group name=mygroup},
 	width=0.2*\textwidth,
 	height=0.15*\textwidth,
 	every axis/.append style={font=\small,title style={anchor=base,yshift=-1mm}, x label style={yshift = 0.5em}, y label style={yshift = -.5em}, grid = both}, legend style={at={(0.7,1)},anchor=north east,font=\tiny,text opacity = 1,fill opacity=0.6}
 	]
  
\nextgroupplot[ylabel={Cost},xlabel={Iteration},xmax = 16,ymin=60,ymax = 123,xtick={0,4,8,12,16},
xticklabels={0,4,8,12,16},tick align=inside,ytick={60,90,123},
yticklabels={60,90,123},
tick pos=left,mark repeat = 5,mark size = 3pt]
    
 \addplot[dotted,very thick,black,line width=1pt] table [search path={fig/Spiral_results/Knee},x expr=\coordindex, y=PD_cost, col sep=comma] {Knee_WavTV_PD.csv};
 \addplot[dash pattern=on 2pt off 3pt on 4pt off 4pt,very thick,goldenbrown,line width=1pt] table [search path={fig/Spiral_results/Knee},x expr=\coordindex, y=ADMM_cost, col sep=comma] {Knee_WavTV_ADMM.csv};
 \addplot[dashed,very thick,blue,line width=1pt] table [search path={fig/Spiral_results/Knee},x expr=\coordindex, y=fista_cost, col sep=comma] {Knee_WavTV_fista.csv};  
 
\addplot[dashed,mark=star,very thick,darkblue,line width=1pt] table [search path={fig/Spiral_results/Knee},x expr=\coordindex, y=fista_Smooth_cost, col sep=comma] {Knee_WavTV_fistaSmooth.csv}; 

\addplot[solid,very thick,red,line width=1pt] table [search path={fig/Spiral_results/Knee},x expr=\coordindex, y=QNP_cost, col sep=comma] {Knee_WavTV_QNP.csv};
\addplot[solid,mark=star,very thick,red,line width=1pt] table [search path={fig/Spiral_results/Knee},x expr=\coordindex, y=QNP_Smooth_cost, col sep=comma] {Knee_WavTV_QNPSmooth.csv};

\legend{PD,ADMM,APM,S-APM,CQNPM,S-CQNPM};

\nextgroupplot[ylabel={},xlabel={Iteration},xmin=12,xmax = 16,ymin = 65,ymax = 67.2,xtick={12,14,16},
xticklabels={,14,16},tick align=inside,ytick={},yticklabels={},
tick pos=left,xshift={-7mm}]

 \addplot[dotted,very thick,black,line width=1pt] table [search path={fig/Spiral_results/Knee},x expr=\coordindex, y=PD_cost, col sep=comma] {Knee_WavTV_PD.csv};
 \addplot[dash pattern=on 2pt off 3pt on 4pt off 4pt,very thick,goldenbrown,line width=1pt] table [search path={fig/Spiral_results/Knee},x expr=\coordindex, y=ADMM_cost, col sep=comma] {Knee_WavTV_ADMM.csv};
 
\addplot[dashed,very thick,blue,line width=1pt] table [search path={fig/Spiral_results/Knee},x expr=\coordindex, y=fista_cost, col sep=comma] {Knee_WavTV_fista.csv};   
\addplot[dashed,mark=star,very thick,darkblue,line width=1pt] table [search path={fig/Spiral_results/Knee},x expr=\coordindex, y=fista_Smooth_cost, col sep=comma] {Knee_WavTV_fistaSmooth.csv}; 

\addplot[solid,very thick,red,line width=1pt] table [search path={fig/Spiral_results/Knee},x expr=\coordindex, y=QNP_cost, col sep=comma] {Knee_WavTV_QNP.csv};
\addplot[solid,mark=star,very thick,red,line width=1pt] table [search path={fig/Spiral_results/Knee},x expr=\coordindex, y=QNP_Smooth_cost, col sep=comma] {Knee_WavTV_QNPSmooth.csv};

\nextgroupplot[xlabel={CPU Time (Seconds)},xmax = 12,ymax = 123,ymin=60,xtick={0,4,8,12},
xticklabels={0,4,8,12},xtick align=inside,
ylabel={Cost},ytick={60,90,123},
yticklabels={60,90,123},mark repeat = 5,mark size = 3pt,ylabel={Cost}]

 \addplot[dotted,very thick,black,line width=1pt] table [search path={fig/Spiral_results/Knee},x = PD_time, y=PD_cost, col sep=comma] {Knee_WavTV_PD.csv};
 \addplot[dash pattern=on 2pt off 3pt on 4pt off 4pt,very thick,goldenbrown,line width=1pt] table [search path={fig/Spiral_results/Knee},x =ADMM_time, y=ADMM_cost, col sep=comma] {Knee_WavTV_ADMM.csv};

\addplot[dashed,very thick,blue,line width=1pt] table [search path={fig/Spiral_results/Knee},x = fista_time, y=fista_cost, col sep=comma] {Knee_WavTV_fista.csv};
\addplot[dashed,mark=star,very thick,darkblue,line width=1pt] table [search path={fig/Spiral_results/Knee},x =fista_Smooth_time, y=fista_Smooth_cost, col sep=comma] {Knee_WavTV_fistaSmooth.csv}; 

\addplot[solid,very thick,red,line width=1pt] table [search path={fig/Spiral_results/Knee},x = QNP_time, y=QNP_cost, col sep=comma] {Knee_WavTV_QNP.csv};

\addplot[solid,mark=star,very thick,red,line width=1pt] table [search path={fig/Spiral_results/Knee},x = QNP_Smooth_time, y=QNP_Smooth_cost, col sep=comma] {Knee_WavTV_QNPSmooth.csv};

\nextgroupplot[xlabel={CPU Time (Seconds)},xmax = 12,xmin = 10,xtick={10,11,12},
xticklabels={,11,12},xtick align=inside,
tick pos=left,ytick={},yticklabels={},ymin = 65.2,ymax = 66]

 \addplot[dotted,very thick,black,line width=1pt] table [search path={fig/Spiral_results/Knee},x = PD_time, y=PD_cost, col sep=comma] {Knee_WavTV_PD.csv};
 \addplot[dash pattern=on 2pt off 3pt on 4pt off 4pt,very thick,goldenbrown,line width=1pt] table [search path={fig/Spiral_results/Knee},x =ADMM_time, y=ADMM_cost, col sep=comma] {Knee_WavTV_ADMM.csv};

\addplot[dashed,very thick,blue,line width=1pt] table [search path={fig/Spiral_results/Knee},x = fista_time, y=fista_cost, col sep=comma] {Knee_WavTV_fista.csv};
\addplot[dashed,mark=star,very thick,darkblue,line width=1pt] table [search path={fig/Spiral_results/Knee},x =fista_Smooth_time, y=fista_Smooth_cost, col sep=comma] {Knee_WavTV_fistaSmooth.csv}; 

\addplot[solid,very thick,red,line width=1pt] table [search path={fig/Spiral_results/Knee},x = QNP_time, y=QNP_cost, col sep=comma] {Knee_WavTV_QNP.csv};

\addplot[solid,mark=star,very thick,red,line width=1pt] table [search path={fig/Spiral_results/Knee},x = QNP_Smooth_time, y=QNP_Smooth_cost, col sep=comma] {Knee_WavTV_QNPSmooth.csv};

\end{groupplot}
\end{tikzpicture} 
\caption{\cb  Cost values versus iteration (top) and CPU time (bottom) of the knee image with same regularizer as \Cref{fig:brain:radial:WavTV:cost} but with spiral acquisition:
$32$ intervals, $1688$ readout points, and $12$ coils. The parameters $\lambda$ and $\alpha$ were $10^{-3}$ and $\frac{1}{2}$.}
 \label{fig:knee:Spiral:WavTV:cost}
\end{figure}
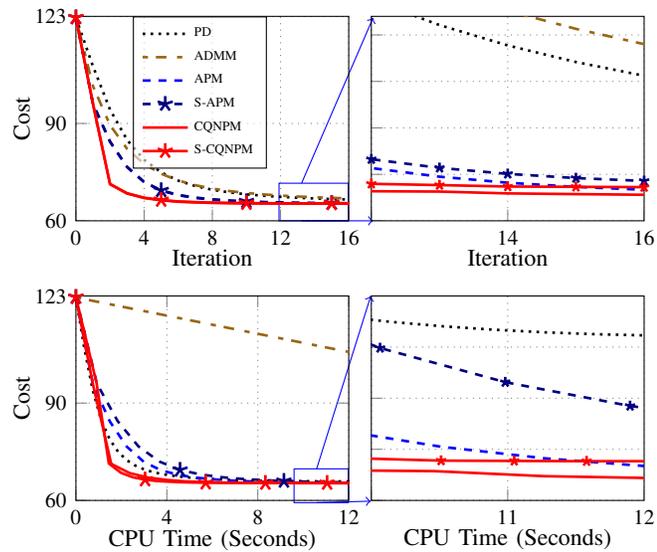

\begin{figure*}[!t]
	\centering
\begin{tikzpicture}


   \node (Knee_GT) at (-5,1.3) {\includegraphics[ width=0.22\textwidth]{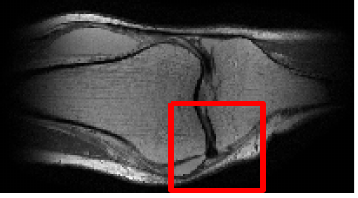}};
   \node at (-6.65,2.2) {\color{white} GT};

    \node (Knee_GT_Crop) at (-1.9,1.3) {\includegraphics[ width=0.11\textwidth]{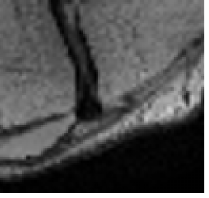}};
  \node at (-4.4,0.5) {\Large \color{white} (a)};
  \node at (-1.2,0.7) {\Large \color{white} (a)};
 	\pgfplotsset{every axis legend/.append style={legend pos=south east,anchor=south east,font=\normalsize, legend cell align={left}}}
    \pgfplotsset{grid style={dotted, gray}}

    \hspace{0.35cm}
      
 	\begin{groupplot}[enlargelimits=false,scale only axis, group style={group size=1 by 1,x descriptions at=edge bottom,group name=mygroup},
 	width=0.37*\textwidth,
 	height=0.2*\textwidth,
 	every axis/.append style={font=\small,title style={anchor=base,yshift=-1mm}, x label style={yshift = 0.5em}, y label style={yshift = -.5em}, grid = both}, legend style={at={(1,0.65)},anchor=north east,font =\tiny,text opacity = 1,fill opacity=0.6}
 	]
  
\nextgroupplot[ylabel={PSNR},xlabel={CPU Time (Seconds)},xmax = 12,ymax = 34,xtick={0,4,8,12},
xticklabels={0,4,8,12},tick align=inside,
tick pos=left,mark repeat = 5,mark size = 3pt]

     \addplot[solid,mark=star,very thick,red,line width=1pt] table [search path={fig/Spiral_results/Knee},x =QNP_Smooth_time, y=QNP_Smooth_PSNR, col sep=comma] {Knee_WavTV_QNPSmooth.csv};

    \addplot[solid,very thick,red,line width=1pt] table [search path={fig/Spiral_results/Knee},x =QNP_time, y=QNP_PSNR, col sep=comma] {Knee_WavTV_QNP.csv};

    \addplot[dashed,very thick,blue,line width=1pt] table [search path={fig/Spiral_results/Knee},x =fista_time, y=fista_PSNR, col sep=comma] {Knee_WavTV_fista.csv};  
    \addplot[dashed,mark=star,very thick,darkblue,line width=1pt] table [search path={fig/Spiral_results/Knee},x =fista_Smooth_time, y=fista_Smooth_PSNR, col sep=comma] {Knee_WavTV_fistaSmooth.csv}; 
   \addplot[dotted,very thick,black,line width=1pt] table [search path={fig/Spiral_results/Knee},x =PD_time, y=PD_PSNR, col sep=comma] {Knee_WavTV_PD.csv};
     
 \addplot[dash pattern=on 2pt off 3pt on 4pt off 4pt,very thick,goldenbrown,line width=1pt] table [search path={fig/Spiral_results/Knee},x =ADMM_time, y=ADMM_PSNR, col sep=comma] {Knee_WavTV_ADMM.csv};
    
     \legend{S-CQNPM,CQNPM,APM,S-APM,PD,ADMM};
     
     \addplot[mark=none, darkblue,dashed,line width=1pt]
     coordinates {(5.3,32.65) (12,32.65)};
     \addplot[mark=none, red,dashed,line width=1pt] coordinates {(5.3,20.5072513259172) (5.3,32.65)};
     \addplot[mark=none, red,dashed,line width=1pt] coordinates {(7.3,20.5072513259172) (7.3,32.65)};

   
 \end{groupplot}
\end{tikzpicture}

\vspace{-4cm}

\begin{tikzpicture}

    \hspace{0.8cm} 

     \begin{axis}[at={(0,0)},anchor = north west, ylabel = PD,
    xmin = 0,xmax = 216,ymin = 0,ymax = 75, width=1\textwidth,
        scale only axis,
        enlargelimits=false,
        yshift= 4cm,
      y label style = {yshift = -0.2cm,xshift=-2cm},
        axis line style={draw=none},
        tick style={draw=none},
        axis equal image,
        xticklabels={,,},yticklabels={,,},
        ]
    \node[inner sep=0pt, anchor = south west] (PD_1) at (0,0) {\includegraphics[ width=0.22\textwidth]{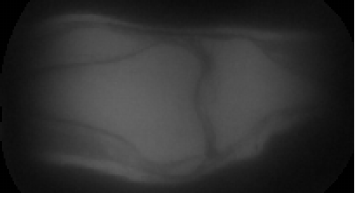}};
    
    \node at (9,24) {\color{white} $\text{iter.}=3$};
    \node at (9,3) {\color{red} $22.02$dB};
    

    \node[inner sep=0pt, anchor = west] (PD_2) at (PD_1.east) {\includegraphics[width=0.22\textwidth]{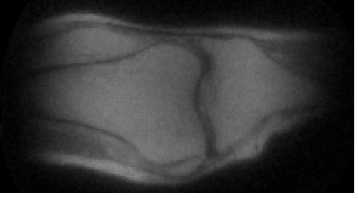}};
    
    \node at (51,24) {\color{white} $10$};
    \node at (56,3) {\color{red} $26.54$dB};
    
    \node[inner sep=0pt, anchor = west] (PD_3) at (PD_2.east) {\includegraphics[width=0.22\textwidth]{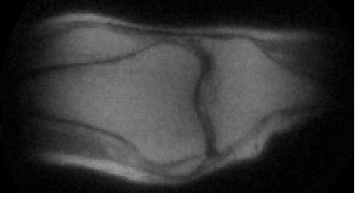}};
    
    \node at (99,24) {\color{white} $13$};
    \node at (103,3) {\color{red} $27.6$dB};

    \node[inner sep=0pt, anchor = west] (PD_4) at (PD_3.east) {\includegraphics[width=0.22\textwidth]{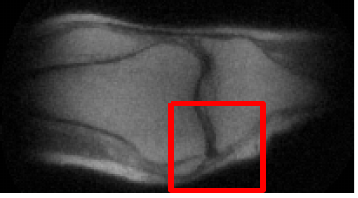}};
    \node at (148,24) {\color{white} $16$};
    \node at (152,3) {\color{red} $28.41$dB};
   \node at (176,4) {\Large \color{white} I};
    \end{axis}
    
     \begin{axis}[at={(PD_1.south west)},anchor = north west, ylabel = APM,
    xmin = 0,xmax = 216,ymin = 0,ymax = 75, width=1\textwidth,
        scale only axis,
        enlargelimits=false,
        yshift= 4.05cm,
      y label style = {yshift = -0.2cm,xshift=-2cm},
        axis line style={draw=none},
        tick style={draw=none},
        axis equal image,
        xticklabels={,,},yticklabels={,,},
        ]
    \node[inner sep=0pt, anchor = south west] (FISTA_1) at (0,0) {\includegraphics[ width=0.22\textwidth]{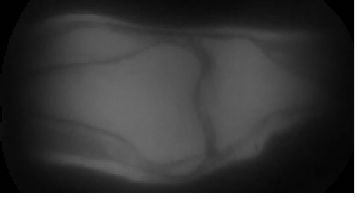}};
    
    \node at (9,3) {\color{red} $22.9$dB};
    

    \node[inner sep=0pt, anchor = west] (FISTA_2) at (FISTA_1.east) {\includegraphics[ width=0.22\textwidth]{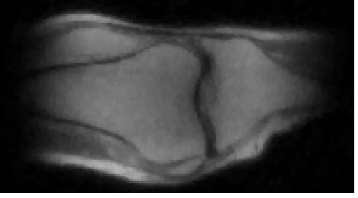}};
    
    \node at (56,3) {\color{red} $29.21$dB};
    
    \node[inner sep=0pt, anchor = west] (FISTA_3) at (FISTA_2.east) {\includegraphics[ width=0.22\textwidth]{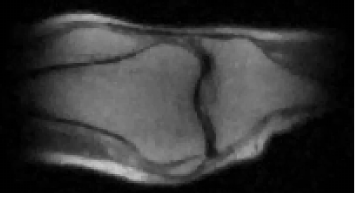}};
    
    \node at (104,3) {\color{red} $31.07$dB};

    \node[inner sep=0pt, anchor = west] (FISTA_4) at (FISTA_3.east) {\includegraphics[ width=0.22\textwidth]{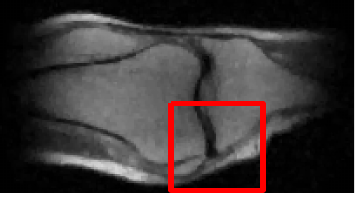}};
    
    \node at (152,3) {\color{red} $32.35$dB};
    \node at (175,4) {\Large \color{white} II};
    \end{axis}

\begin{axis}[at={(FISTA_1.south west)},anchor =  north west, ylabel = S-APM,
    xmin = 0,xmax = 216,ymin = 0,ymax = 75, width=1\textwidth,
        scale only axis,
        enlargelimits=false,
        yshift = 4.05cm,
        y label style = {yshift = -0.2cm,xshift=-2cm},
        axis line style={draw=none},
        tick style={draw=none},
        axis equal image,
        xticklabels={,,},yticklabels={,,},
        ]
    \node[inner sep=0pt, anchor = south west] (S_FISTA_1) at (0,0) {\includegraphics[ width=0.22\textwidth]{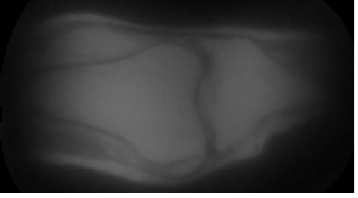}}; 
    \node at (9,3) {\color{red} $22.9$dB};
    \node[inner sep=0pt, anchor = west] (S_FISTA_2) at (S_FISTA_1.east) {\includegraphics[ width=0.22\textwidth]{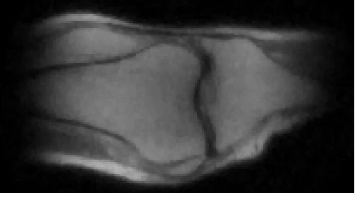}};

    \node at (56,3) {\color{red} $29.21$dB};

     \node[inner sep=0pt, anchor = west] (S_FISTA_3) at (S_FISTA_2.east) {\includegraphics[ width=0.22\textwidth]{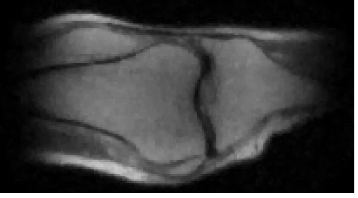}};
     \node at (104,3) {\color{red} $31.06$dB};
    \node[inner sep=0pt, anchor = west] (S_FISTA_4) at (S_FISTA_3.east) {\includegraphics[width=0.22\textwidth]{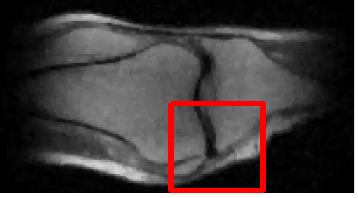}};
    \node at (152,3) {\color{red} $32.33$dB};
     \node at (175,4) {\Large \color{white} III};
    \end{axis}

\begin{axis}[at={(S_FISTA_1.south west)},anchor = north west,ylabel = CQNPM,
    xmin = 0,xmax = 216,ymin = 0,ymax = 75, width=1\textwidth,
        scale only axis,
        enlargelimits=false,
        yshift = 4.05cm,
        y label style = {yshift = -0.2cm,xshift=-2cm},
       axis line style={draw=none},
       tick style={draw=none},
        axis equal image,
        xticklabels={,,},yticklabels={,,}
       ]
   \node[inner sep=0pt, anchor = south west] (QN_1) at (0,0) {\includegraphics[ width=0.22\textwidth]{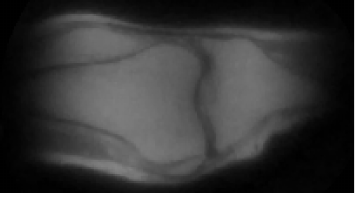}};
    \node at (9,3) {\color{red} $25.2$dB};
    
    \node[inner sep=0pt, anchor = west] (QN_2) at (QN_1.east) {\includegraphics[width=0.22\textwidth]{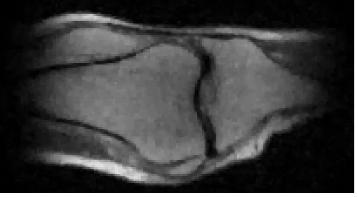}};
    \node at (56,3) {\color{red} $32.65$dB};
    
    \node[inner sep=0pt, anchor = west] (QN_3) at (QN_2.east) {\includegraphics[width=0.22\textwidth]{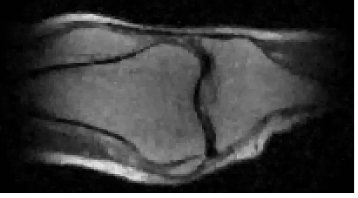}};
    \node at (104,3) {\color{red} $33.28$dB};

    \node[inner sep=0pt, anchor = west] (QN_4) at (QN_3.east) {\includegraphics[width=0.22\textwidth]{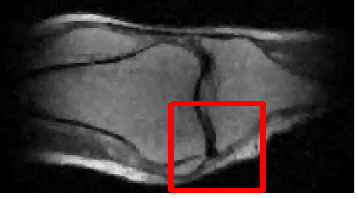}};
     \node at (152,3) {\color{red} $33.45$dB};

      \node at (175,4) {\Large \color{white} IV};
 \end{axis}

    \begin{axis}[at={(QN_1.south west)},anchor = north west,ylabel = S-CQNPM,
    xmin = 0,xmax = 216,ymin = 0,ymax = 75, width=1\textwidth,
        scale only axis,
        enlargelimits=false,
          yshift = 4.05cm,
        y label style = {yshift = -0.2cm,xshift=-2cm},
       axis line style={draw=none},
       tick style={draw=none},
        axis equal image,
        xticklabels={,,},yticklabels={,,}
       ]
       
   \node[inner sep=0pt, anchor = south west] (S_QN_1) at (0,0) {\includegraphics[ width=0.22\textwidth]{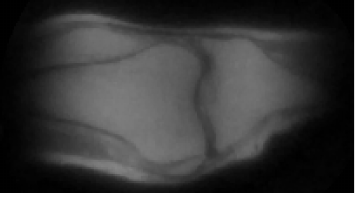}};
   
    \node at (9,3) {\color{red} $25.21$dB};

    \node[inner sep=0pt, anchor = west] (S_QN_2) at (S_QN_1.east) {\includegraphics[ width=0.22\textwidth]{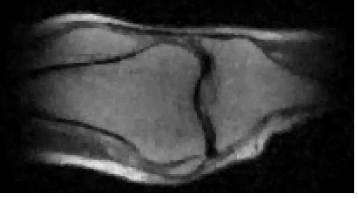}};
    \node at (56,3) {\color{red} $32.63$dB};
  
    \node[inner sep=0pt, anchor = west] (S_QN_3) at (S_QN_2.east) {\includegraphics[width=0.22\textwidth]{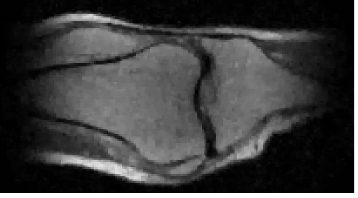}};
    
     \node at (104,3) {\color{red} $33.27$dB};
    \node[inner sep=0pt, anchor = west] (S_QN_4) at (S_QN_3.east) {\includegraphics[width=0.22\textwidth]{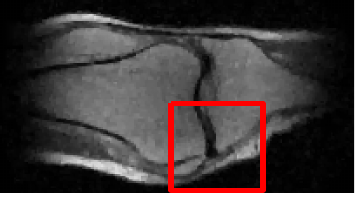}};
    \node at (152,3) {\color{red} $33.51$dB}; 
     \node at (176,4) {\Large \color{white} V};
 \end{axis}

     \begin{axis}[at={(S_QN_1.south west)},anchor = north west, 
    xmin = 0,xmax = 216,ymin = 0,ymax = 75, width=1\textwidth,
        scale only axis,
        enlargelimits=false,
        yshift=4.38cm,
      y label style = {yshift = -0.2cm,xshift=-2cm},
        axis line style={draw=none},
        tick style={draw=none},
        axis equal image,
        xticklabels={,,},yticklabels={,,},
        ]
        
    \node[inner sep=0pt, anchor = south west] (PDZoom_1) at (0,0) {\includegraphics[ width=0.11\textwidth]{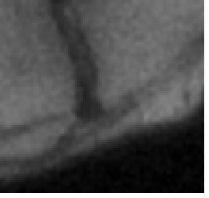}};
    
    \node[inner sep=0pt,yshift=-27pt,anchor = south west] (FISTAZoom_1) at (PDZoom_1.east) {\includegraphics[width=0.11\textwidth]{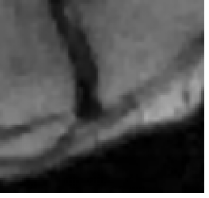}};
    
    \node[inner sep=0pt,yshift=-27pt,anchor = south west] (SFISTAZoom_1) at (FISTAZoom_1.east) {\includegraphics[width=0.11\textwidth]{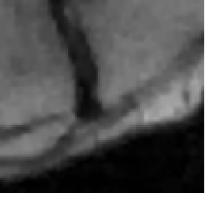}};
    
    \node[inner sep=0pt,yshift=-27pt,anchor = south west] (QNPZoom_1) at (SFISTAZoom_1.east) {\includegraphics[width=0.11\textwidth]{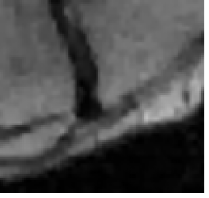}};

    \node[inner sep=0pt,yshift=-27pt,anchor = south west] (SQNPZoom_2) at (QNPZoom_1.east) {\includegraphics[width=0.11\textwidth]{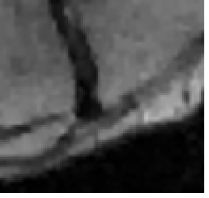}};
    \node at (22,3.5) {\Large \color{white} I};
    \node at (45,3.5) {\Large \color{white} II};
    \node at (68,3.5) {\Large \color{white} III};
    \node at (92,3.5) {\Large \color{white} IV};
    \node at (116,3.5) {\Large \color{white} V};
    \end{axis}

     \begin{axis}[at={(PDZoom_1.south west)},anchor = north west, 
    xmin = 0,xmax = 216,ymin = 0,ymax = 75, width=1\textwidth,
        scale only axis,
        yshift=4.38cm,
        enlargelimits=false,
      y label style = {yshift = -0.2cm,xshift=-2cm},
        axis line style={draw=none},
        tick style={draw=none},
        axis equal image,
        xticklabels={,,},yticklabels={,,},
        ]
        
    \node[inner sep=0pt, anchor = south west] (ResiPDZoom_1) at (0,0) {\includegraphics[ width=0.11\textwidth]{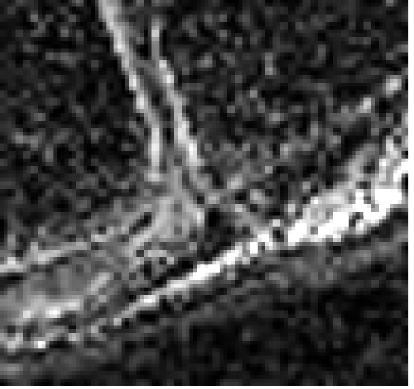}};
    
    \node[inner sep=0pt,yshift=-27pt,anchor = south west] (ResiFISTAZoom_1) at (ResiPDZoom_1.east) {\includegraphics[width=0.11\textwidth]{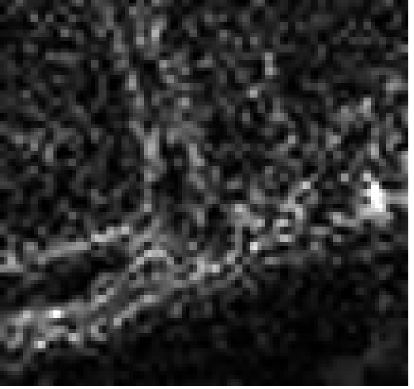}};

    \node[inner sep=0pt,yshift=-27pt,anchor = south west] (ResiSFISTAZoom_2) at (ResiFISTAZoom_1.east) {\includegraphics[width=0.11\textwidth]{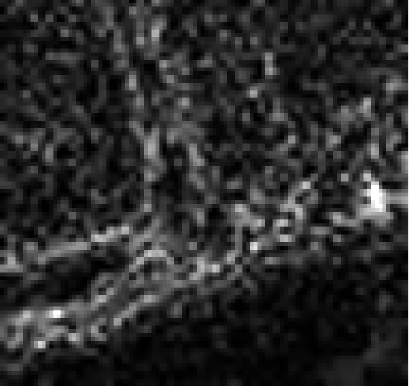}};
    
    \node[inner sep=0pt,yshift=-27pt,anchor = south west] (ResiQNPZoom_1) at (ResiSFISTAZoom_2.east) {\includegraphics[width=0.11\textwidth]{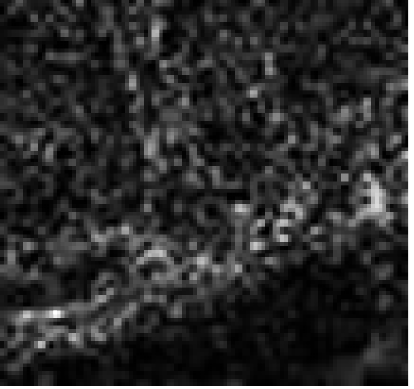}};

    \node[inner sep=0pt,yshift=-27pt,anchor = south west] (ResiSQNPZoom_2) at (ResiQNPZoom_1.east) {\includegraphics[width=0.11\textwidth]{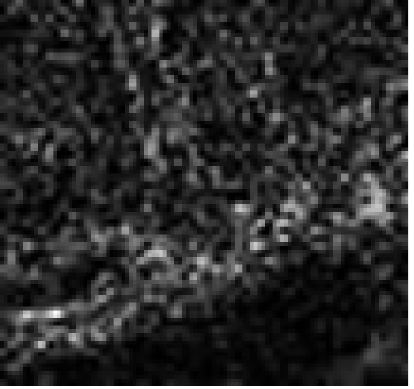}};
 \node at (127,13) {\LARGE \color{red}$\times 5$};

    \node at (22,3.5) {\Large \color{white} I};
    \node at (45,3.5) {\Large \color{white} II};
    \node at (68,3.5) {\Large \color{white} III};
    \node at (92,3.5) {\Large \color{white} IV};
    \node at (116,3.5) {\Large \color{white} V};
    
\end{axis}

\end{tikzpicture} 
\caption{{ First row: the ground truth image and
PSNR values versus CPU time; second to sixth row:  the reconstructed knee images at $3$, $10$, $13$, and $16$th iteration with \Cref{fig:knee:Spiral:WavTV:cost} setting. We did not show the reconstructed image of ADMM since it yielded a much lower PSNR than other methods. The seventh and eighth rows represent the zoomed-in regions and the corresponding error maps ($\times 5$) of the $16$th itertion reconstructed images with PD $\rightarrow$ APM $\rightarrow$ S-APM $\rightarrow$ CQNPM $\rightarrow$ S-CQNPM.}
}
	\label{fig:SpiralKnee:WavTV}
\end{figure*}

%% file: Figs/BrainWaveTVFigSpiralhighSNR.tex
\begin{figure*}[!t]
	\centering

\begin{tikzpicture}
    \node (Brain_GT) at (-5,1.3) {\includegraphics[ width=0.18\textwidth]{fig/CropSquare_Brain_GT.pdf}};
   \node at (-6.35,2.9) {\color{white} GT};

    \node (Brain_GT_Crop) at (-1.9,1.3) {\includegraphics[ width=0.15\textwidth]{fig/Crop_Brain_GT.pdf}};
    \node at (-3.9,2.63) {\color{white} (a)};

   \node at (-0.8,2.35) {\color{white} (a)};
   
 	\pgfplotsset{every axis legend/.append style={legend pos=south east,anchor=south east,font=\normalsize, legend cell align={left}}}
    \pgfplotsset{grid style={dotted, gray}}
    \hspace{0.45cm}
 	\begin{groupplot}[enlargelimits=false,scale only axis, group style={group size=1 by 1,x descriptions at=edge bottom,group name=mygroup},
 	width=0.5*\textwidth,
 	height=0.2*\textwidth,
 	every axis/.append style={font=\small,title style={anchor=base,yshift=-1mm}, x label style={yshift = 0.5em}, y label style={yshift = -.5em}, grid = both}, legend style={at={(1,0.57)},anchor=north east,font =\tiny,fill opacity=0.6, draw opacity=1}
 	]
  
\nextgroupplot[ylabel={PSNR},xlabel={CPU Time (Seconds)},xmax = 12.2,ymax = 37,xtick={0,6,10,12,12.2},
xticklabels={0,6,10,12,},tick align=inside,
tick pos=left,mark repeat = 5,mark size = 3pt]
    
  \addplot[solid,mark=star,very thick,red,line width=1pt] table [search path={fig/SpiralHighSNR_results/Brain},x =QNP_Smooth_time, y=QNP_Smooth_PSNR, col sep=comma] {Brain_WavTV_QNPSmooth.csv};
    
  \addplot[solid,very thick,red,line width=1pt] table [search path={fig/SpiralHighSNR_results/Brain},x =QNP_time, y=QNP_PSNR, col sep=comma] {Brain_WavTV_QNP.csv};

\addplot[dashed,very thick,blue,line width=1pt] table [search path={fig/SpiralHighSNR_results/Brain},x =fista_time, y=fista_PSNR, col sep=comma] {Brain_WavTV_fista.csv}; 
        
    \addplot[dashed,mark=star,very thick,darkblue,line width=1pt] table [search path={fig/SpiralHighSNR_results/Brain},x =fista_Smooth_time, y=fista_Smooth_PSNR, col sep=comma] {Brain_WavTV_fistaSmooth.csv}; 
 
  \legend{S-CQNPM,CQNPM,APM,S-APM};

     \addplot[mark=none, darkblue,dashed,line width=1pt]
     coordinates {(7.3,33.3) (12.2,33.3)};
     
     \addplot[mark=none, red,dashed,line width=1pt] coordinates {(10,18.6275510867208) (10,33.3)};
     \addplot[mark=none, red,dashed,line width=1pt] coordinates {(7.3,18.6275510867208) (7.3,33.3)};

 \end{groupplot}
\end{tikzpicture}

\vspace{-2cm}

\begin{tikzpicture}

 \hspace{0.8cm} 
 
     \begin{axis}[at={(0,0)},anchor = north west, ylabel = APM,
    xmin = 0,xmax = 216,ymin = 0,ymax = 75, width=1\textwidth,
        scale only axis,
        enlargelimits=false,
      y label style = {yshift = -0.2cm,xshift=-1.5cm},
        axis line style={draw=none},
        tick style={draw=none},
        axis equal image,
        xticklabels={,,},yticklabels={,,},
        ]

    \node[inner sep=0pt, anchor = south west] (FISTA_1) at (0,0) {\includegraphics[ width=0.18\textwidth]{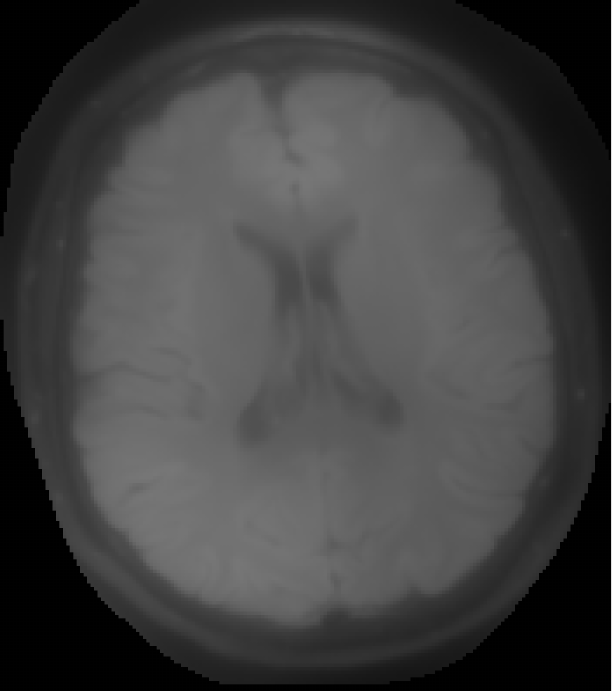}};
    
    \node at (9,3) {\color{red} $20.41$dB};

    \node[inner sep=0pt, anchor = west] (FISTA_2) at (FISTA_1.east) {\includegraphics[ width=0.18\textwidth]{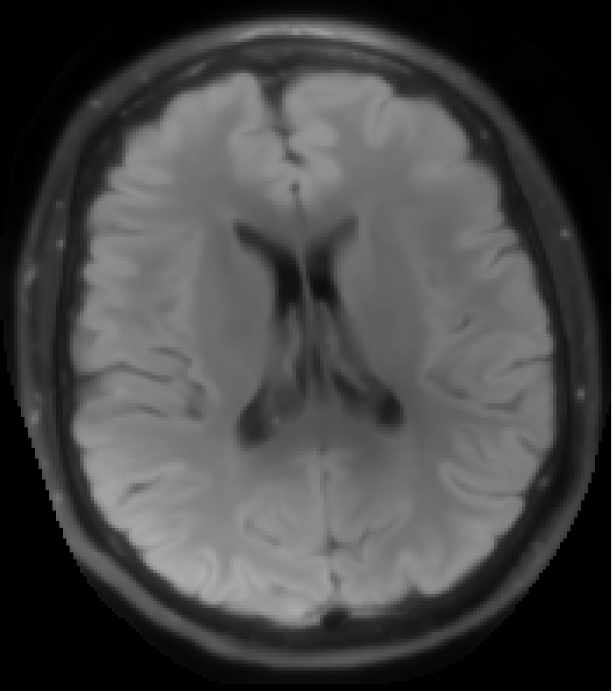}};

    \node at (48,3) {\color{red} $27.22$dB};
    
    \node[inner sep=0pt, anchor = west] (FISTA_3) at (FISTA_2.east) {\includegraphics[ width=0.18\textwidth]{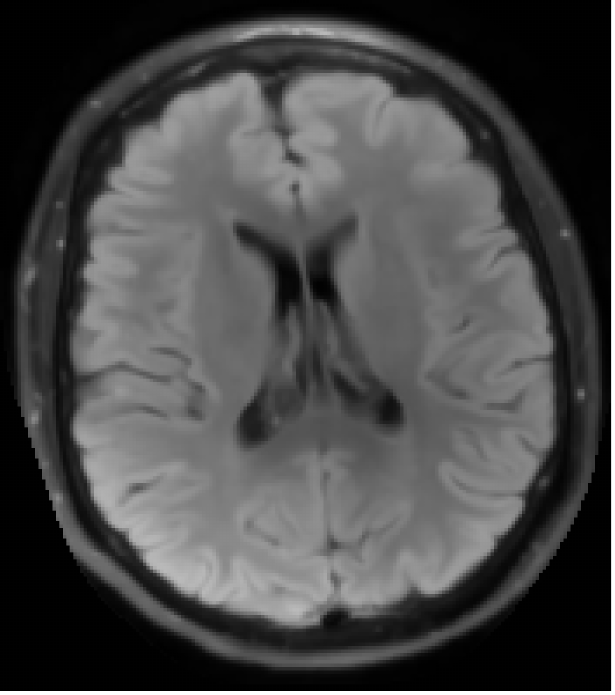}};
    
    \node at (88,3) {\color{red} $30.04$dB};

    \node[inner sep=0pt, anchor = west] (FISTA_4) at (FISTA_3.east) {\includegraphics[ width=0.18\textwidth]{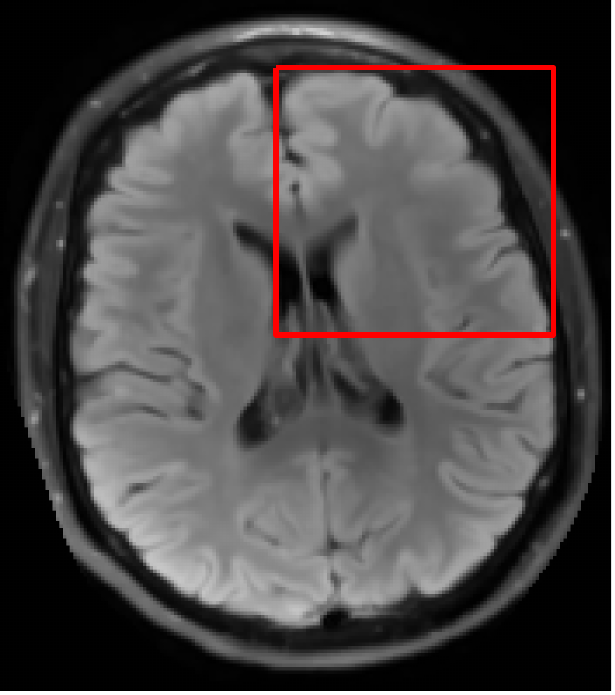}};
    \node at (125,3) {\color{red} $32.56$dB};

    \node[inner sep=0pt, anchor = west] (CropFISTA_5) at (FISTA_4.east) {\includegraphics[ width=0.15\textwidth]{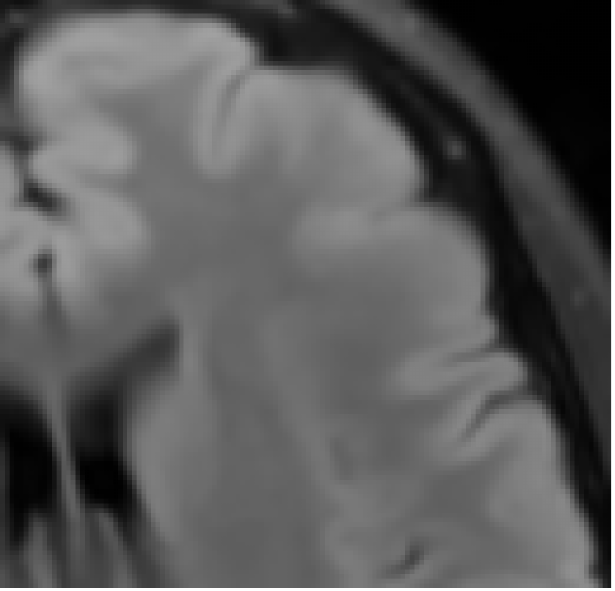}};
      \node at (150,37) {\Large\color{white} I};
   \node at (186,34) {\Large\color{white} I};
    \end{axis}

\begin{axis}[at={(FISTA_1.south west)},anchor =  north west, ylabel = S-APM,
    xmin = 0,xmax = 216,ymin = 0,ymax = 75, width=1\textwidth,
        scale only axis,
        enlargelimits=false,
        yshift = 2.6cm,
        y label style = {yshift = -0.2cm,xshift=-1.3cm},
        axis line style={draw=none},
        tick style={draw=none},
        axis equal image,
        xticklabels={,,},yticklabels={,,},
        ]
    \node[inner sep=0pt, anchor = south west] (S_FISTA_1) at (0,0) {\includegraphics[ width=0.18\textwidth]{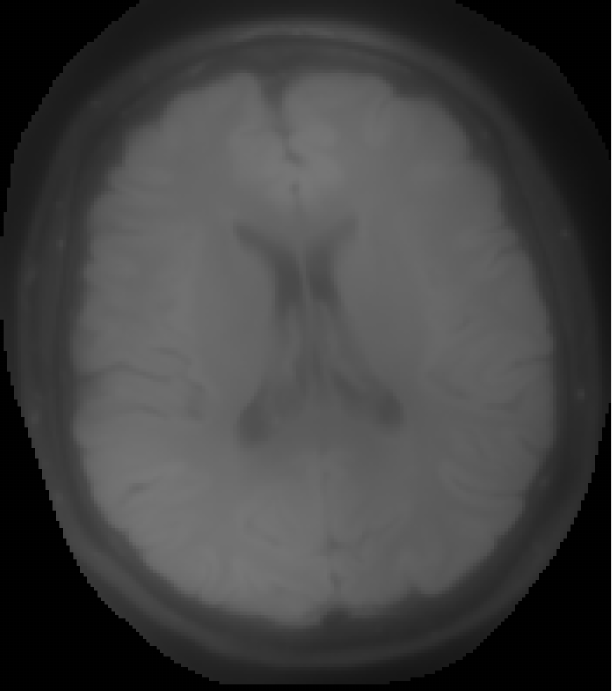}}; 
    \node at (9,3) {\color{red} $20.41$dB};
    \node[inner sep=0pt, anchor = west] (S_FISTA_2) at (S_FISTA_1.east) {\includegraphics[ width=0.18\textwidth]{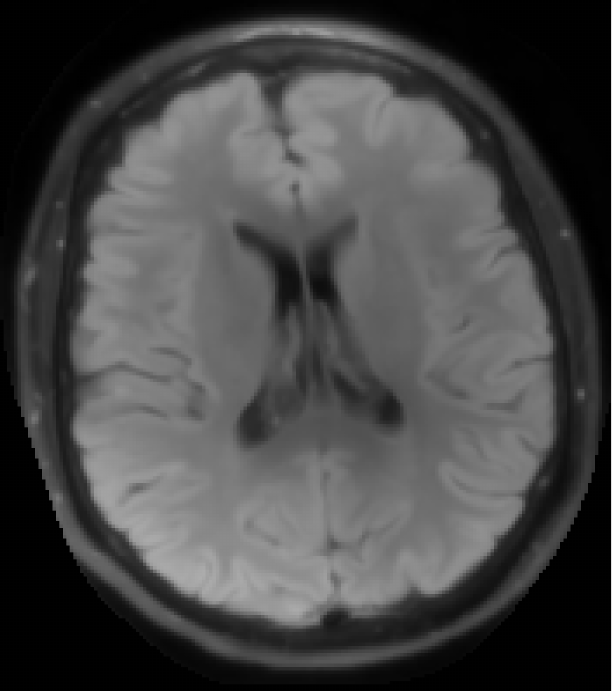}};

    \node at (49,3) {\color{red} $27.23$dB};

    \node[inner sep=0pt, anchor = west] (S_FISTA_3) at (S_FISTA_2.east) {\includegraphics[ width=0.18\textwidth]{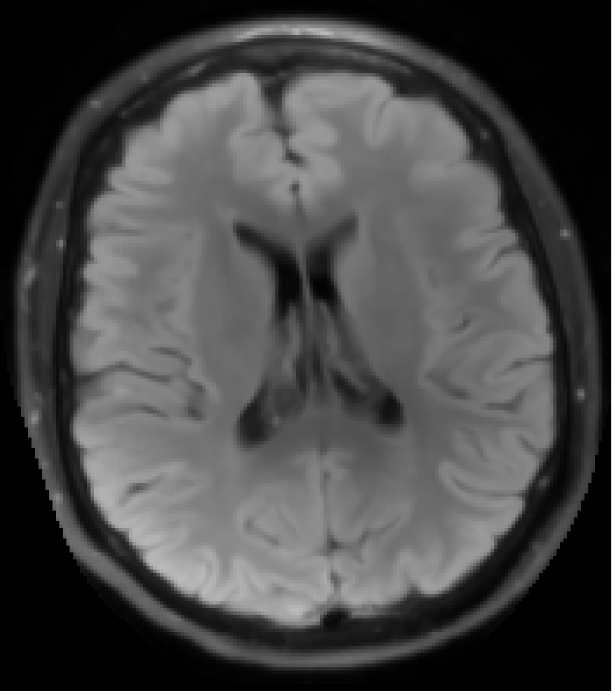}};
     \node at (87,3) {\color{red} $30.05$dB};
    \node[inner sep=0pt, anchor = west] (S_FISTA_4) at (S_FISTA_3.east) {\includegraphics[width=0.18\textwidth]{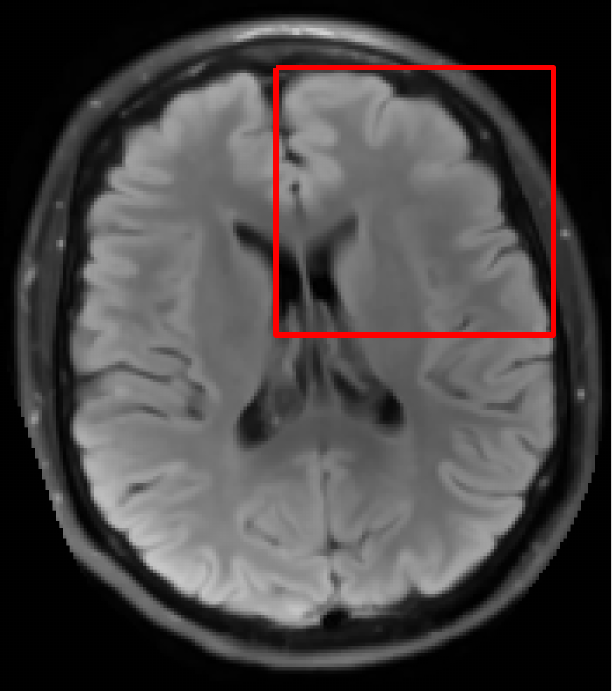}};
    \node at (125,3) {\color{red} $32.6$dB};

 \node[inner sep=0pt, anchor = west] (CropS_FISTA_4) at (S_FISTA_4.east) {\includegraphics[width=0.15\textwidth]{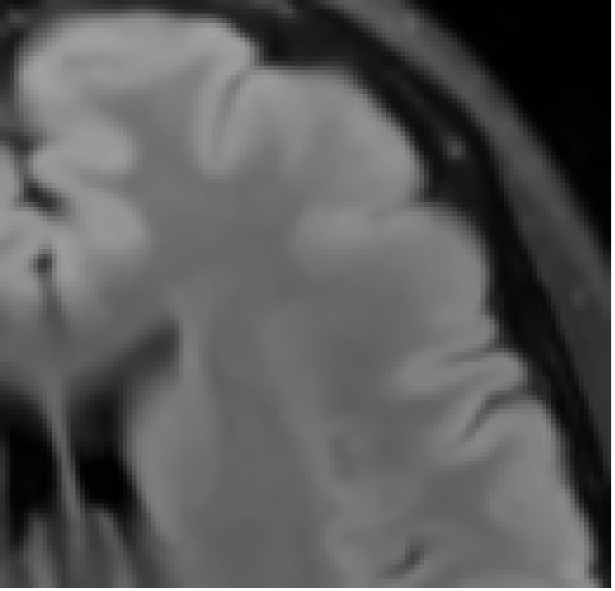}};
      \node at (149,37) {\Large\color{white} II};
   \node at (185,34) {\Large\color{white} II};
    \end{axis}

\begin{axis}[at={(S_FISTA_1.south west)},anchor = north west,ylabel = CQNPM,
    xmin = 0,xmax = 216,ymin = 0,ymax = 75, width=1\textwidth,
        scale only axis,
        enlargelimits=false,
        yshift = 2.6cm,
        y label style = {yshift = -0.2cm,xshift=-1.3cm},
       axis line style={draw=none},
       tick style={draw=none},
        axis equal image,
        xticklabels={,,},yticklabels={,,}
       ]
   \node[inner sep=0pt, anchor = south west] (QN_1) at (0,0) {\includegraphics[ width=0.18\textwidth]{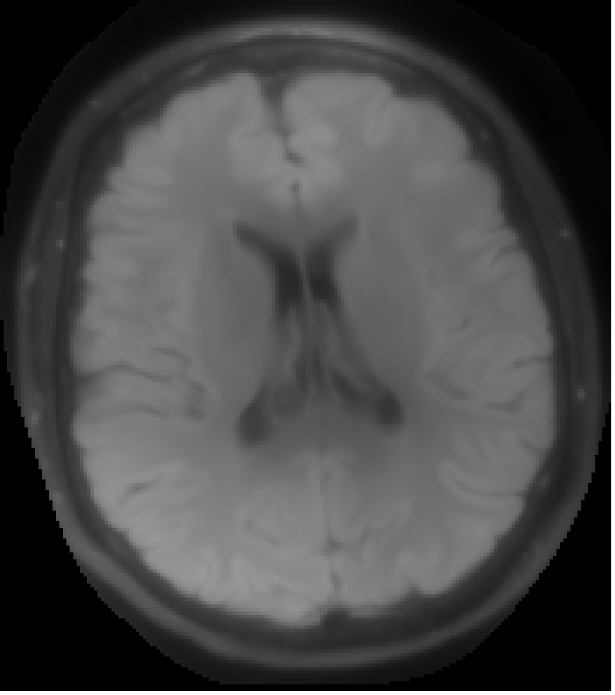}};
    \node at (9,3) {\color{red} $23.8$dB};
    
    \node[inner sep=0pt, anchor = west] (QN_2) at (QN_1.east) {\includegraphics[ width=0.18\textwidth]{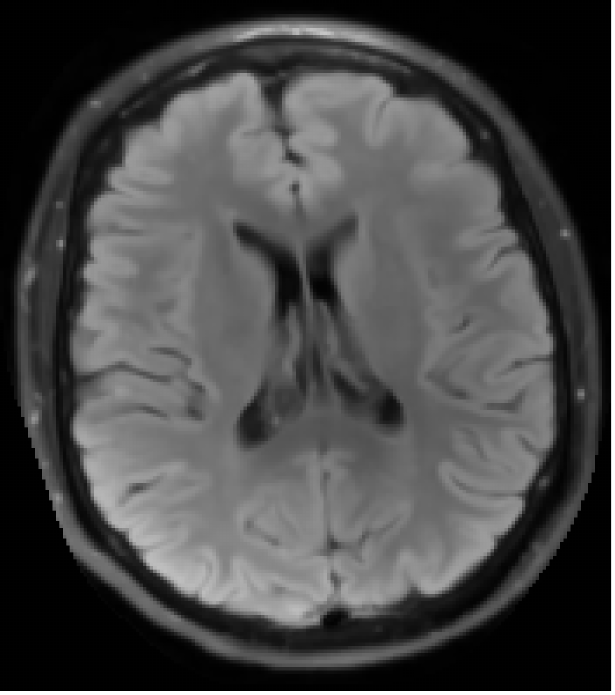}};
    \node at (49,3) {\color{red} $32.46$dB};
    
    \node[inner sep=0pt, anchor = west] (QN_3) at (QN_2.east) {\includegraphics[width=0.18\textwidth]{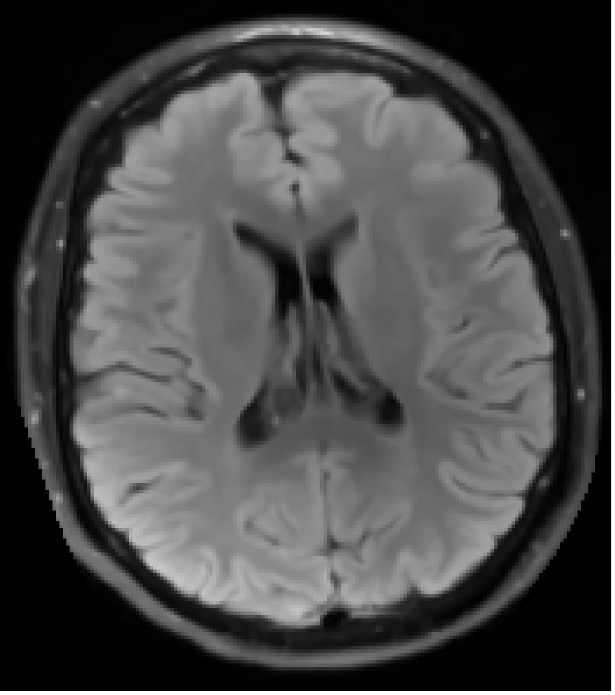}};
    \node at (88,3) {\color{red} $32.76$dB};

    \node[inner sep=0pt, anchor = west] (QN_4) at (QN_3.east) {\includegraphics[ width=0.18\textwidth]{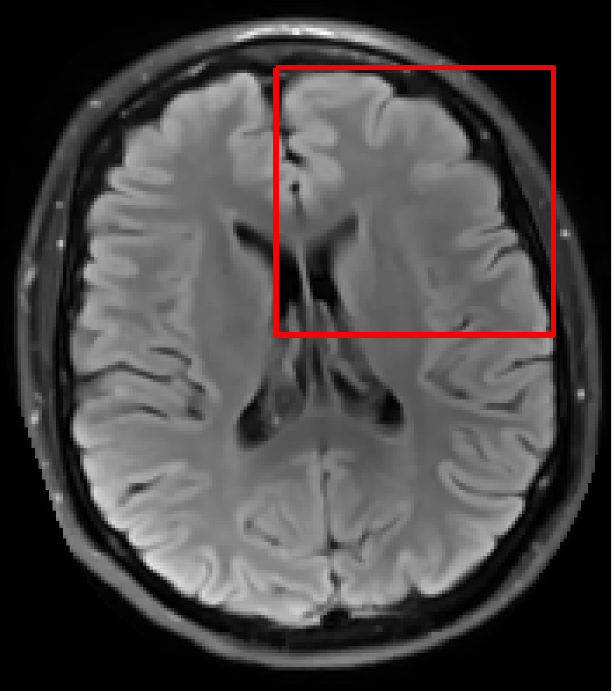}};
     \node at (125,3) {\color{red} $34.5$dB};  

    \node[inner sep=0pt, anchor = west] (CropQN_5) at (QN_4.east) {\includegraphics[width=0.15\textwidth]{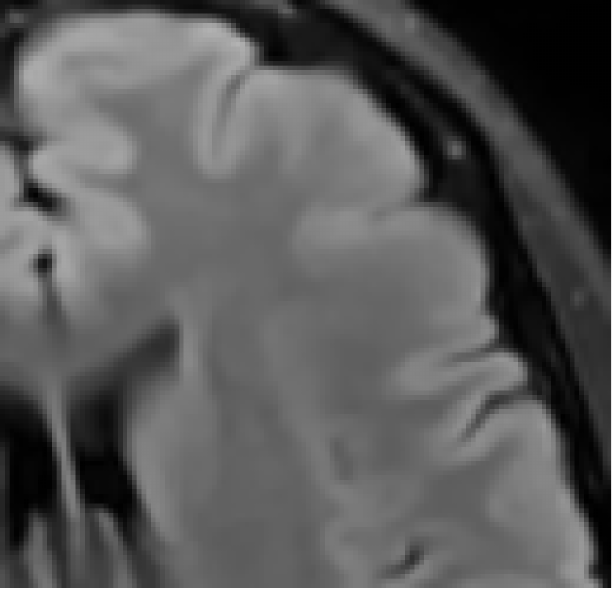}};
    
      \node at (149,37) {\Large\color{white} III};
   \node at (185,34) {\Large\color{white} III};
 \end{axis}

    \begin{axis}[at={(QN_1.south west)},anchor = north west,ylabel = S-CQNPM,
    xmin = 0,xmax = 216,ymin = 0,ymax = 75, width=1\textwidth,
        scale only axis,
        enlargelimits=false,
          yshift = 2.6cm,
        y label style = {yshift = -0.2cm,xshift=-1.3cm},
       axis line style={draw=none},
       tick style={draw=none},
        axis equal image,
        xticklabels={,,},yticklabels={,,}
       ]
       
   \node[inner sep=0pt, anchor = south west] (S_QN_1) at (0,0) {\includegraphics[ width=0.18\textwidth]{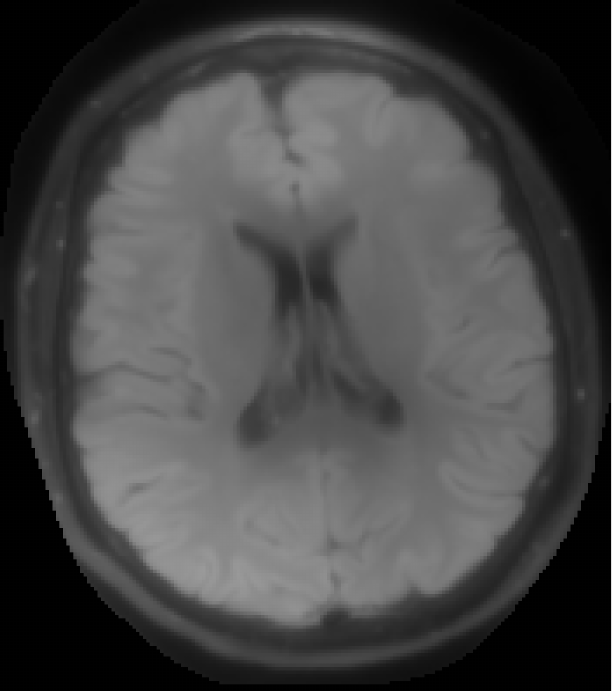}};
   
    \node at (9,3) {\color{red} $23.14$dB};

    \node[inner sep=0pt, anchor = west] (S_QN_2) at (S_QN_1.east) {\includegraphics[ width=0.18\textwidth]{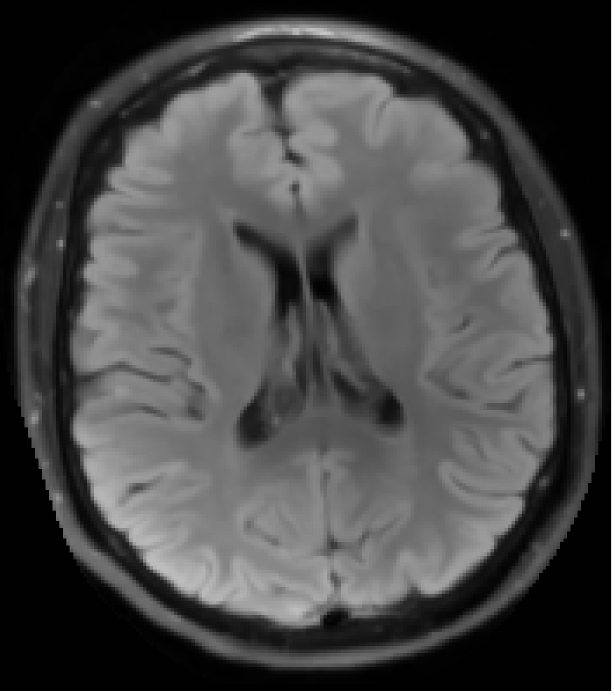}};
    \node at (49,3) {\color{red} $32.47$dB};
  
    \node[inner sep=0pt, anchor = west] (S_QN_3) at (S_QN_2.east) {\includegraphics[ width=0.18\textwidth]{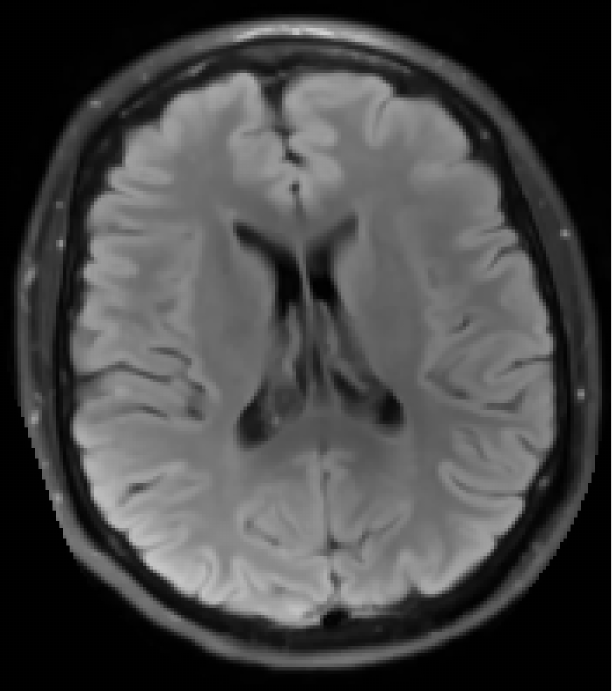}};
    
     \node at (88,3) {\color{red} $32.78$dB};
    \node[inner sep=0pt, anchor = west] (S_QN_4) at (S_QN_3.east) {\includegraphics[ width=0.18\textwidth]{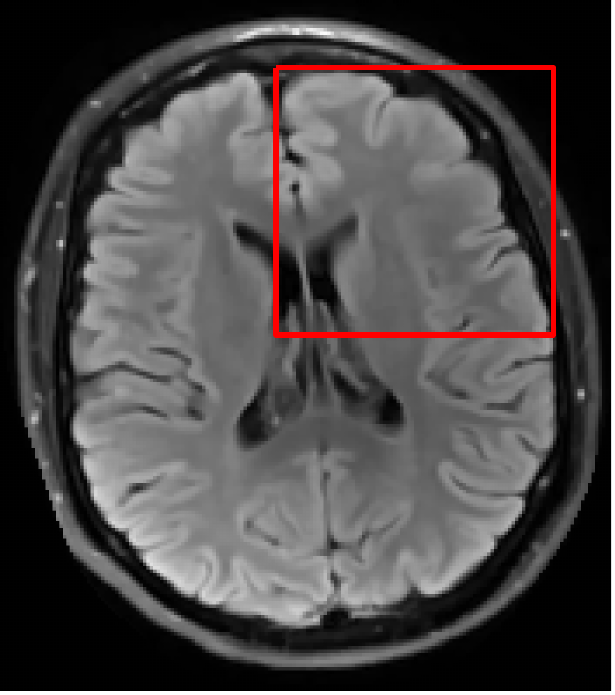}};
    \node at (125,3) {\color{red} $34.53$dB};  

       \node[inner sep=0pt, anchor = west] (CropS_QN_5) at (S_QN_4.east) {\includegraphics[ width=0.15\textwidth]{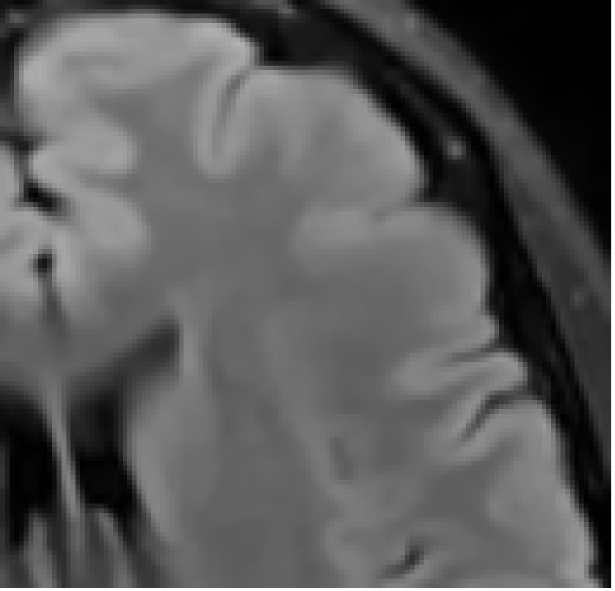}};
       
    \node at (149,37) {\Large\color{white} IV};
   \node at (185,34) {\Large\color{white} IV};
       
 \end{axis}
 
\end{tikzpicture} 
\caption{
\Mcb First row: the ground truth image and PSNR values versus CPU time; second to sixth row: the reconstructed brain images  at $3$, $10$, $13$, and $16$th iteration with spiral acquisition and $h(\vx)=\alpha\|\mT\vx\|_1+(1-\alpha)\mathrm{TV}(\vx)$ and the zoomed-in regions of the $16$th iteration reconstruction. The parameters were $\lambda=4\times10^{-5}$ and $\alpha=\frac{3}{4}.$
}
\label{fig:SpiralBrainHigh:WavTV}
\end{figure*}

%% file: Figs/KneeWaveFigRadial.tex

\begin{figure}[!t]
	\centering
\begin{tikzpicture}
     \begin{axis}[at={(0,0)},anchor = north west,
    xmin = 0,xmax = 216,ymin = 0,ymax = 75, width=1\textwidth,
        scale only axis,
        enlargelimits=false,
      y label style = {yshift = -0.2cm,xshift=-1.5cm},
        axis line style={draw=none},
        tick style={draw=none},
        axis equal image,
        xticklabels={,,},yticklabels={,,},
        ]
    \node[inner sep=0pt, anchor = south west] (ResiPD_1) at (0,0) {\includegraphics[ width=0.15\textwidth]{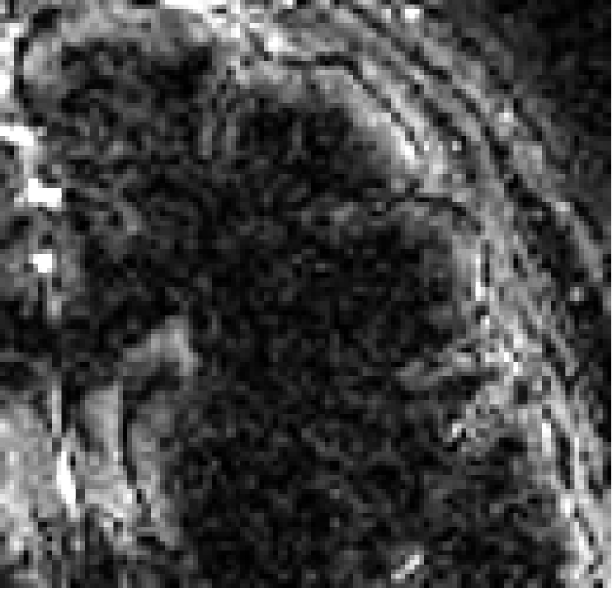}};
    

    \node[inner sep=0pt, anchor = west] (ResiFISTA_1) at (ResiPD_1.east) {\includegraphics[ width=0.15\textwidth]{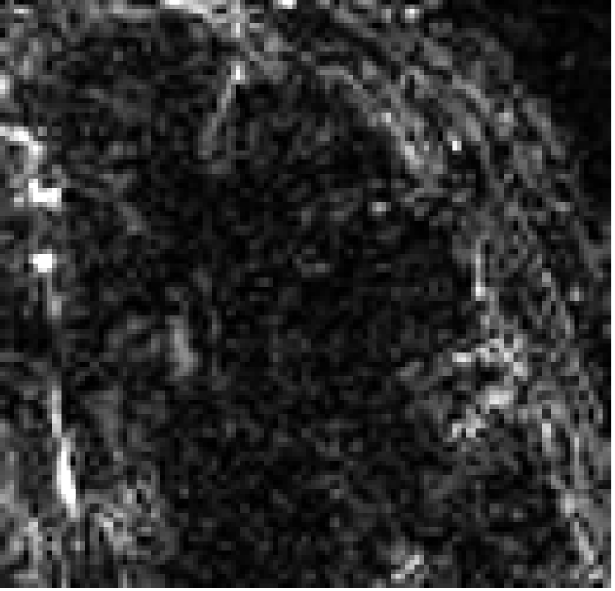}};
    
    \node[inner sep=0pt, anchor = west] (ResiSFISTA_1) at (ResiFISTA_1.east) {\includegraphics[ width=0.15\textwidth]{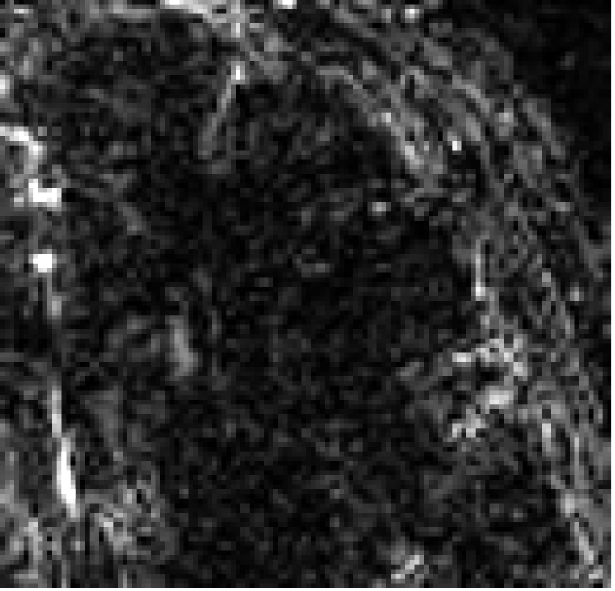}};
    

    \end{axis}

        \begin{axis}[at={(ResiPD_1.south west)},anchor = north west,
    xmin = 0,xmax = 216,ymin = 0,ymax = 75, width=1\textwidth,
        scale only axis,
        enlargelimits=false,
        yshift=3.6cm,
      y label style = {yshift = -0.2cm,xshift=-1.5cm},
        axis line style={draw=none},
        tick style={draw=none},
        axis equal image,
        xticklabels={,,},yticklabels={,,},
        ]
    
    \node[inner sep=0pt, anchor = south west] (ResiQNP_1) at (0,0) {\includegraphics[ width=0.15\textwidth]{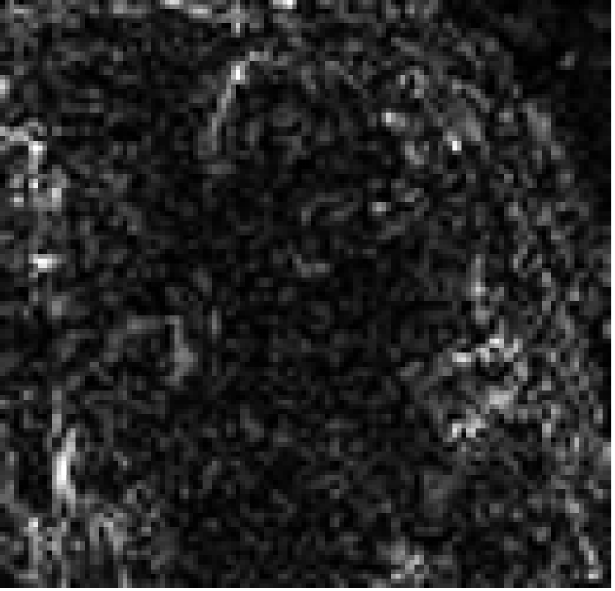}};

    \node[inner sep=0pt, anchor = west] (ResiSQNP_1) at (ResiQNP_1.east) {\includegraphics[ width=0.15\textwidth]{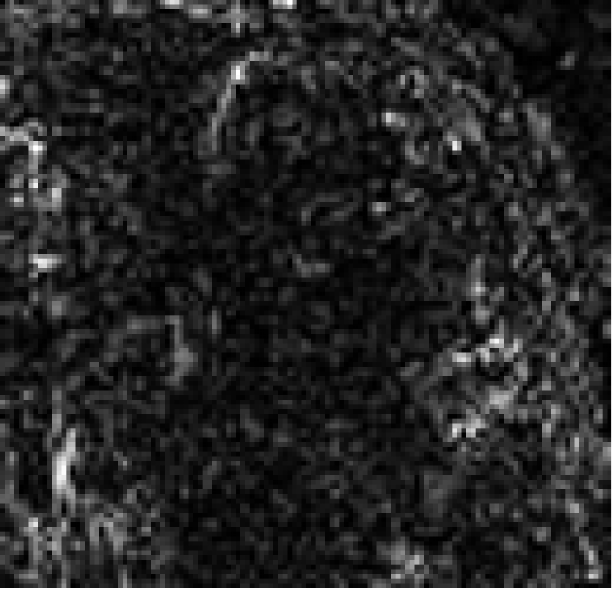}};
    \node at (70,16) {\LARGE \color{red}$\times 5$};
    \end{axis}
\end{tikzpicture} 

\caption{ {
 Error maps ($\times 5$) of the zoomed-in regions of the $16$th iteration reconstructed images in Fig. 6 of the main manuscript. From left to right, first row: PD $\rightarrow$ APM $\rightarrow$  S-APM;
 second row: CQNPM $\rightarrow$ S-CQNPM.}
}
\label{fig:RadialBrain:TVwavelet}
\end{figure}

\begin{figure}[!t]
	\centering
\begin{tikzpicture}
     \begin{axis}[at={(0,0)},anchor = north west,
    xmin = 0,xmax = 216,ymin = 0,ymax = 75, width=1\textwidth,
        scale only axis,
        enlargelimits=false,
      y label style = {yshift = -0.2cm,xshift=-1.5cm},
        axis line style={draw=none},
        tick style={draw=none},
        axis equal image,
        xticklabels={,,},yticklabels={,,},
        ]
    \node[inner sep=0pt, anchor = south west] (ResiPD_1) at (0,0) {\includegraphics[ width=0.15\textwidth]{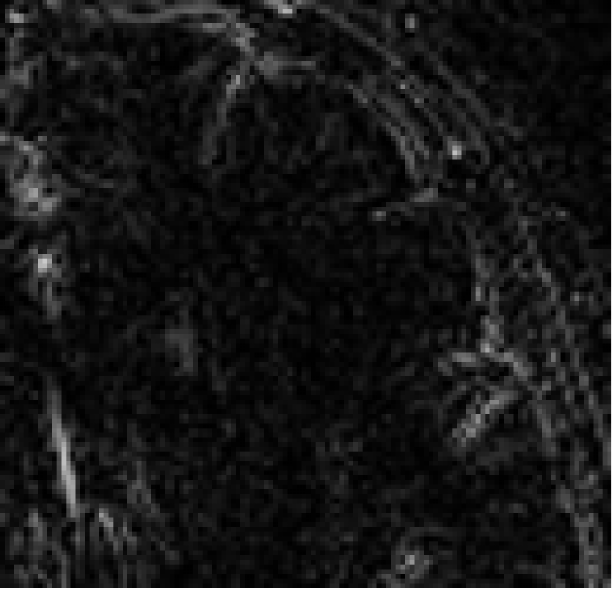}};

    \node[inner sep=0pt, anchor = west] (ResiFISTA_1) at (ResiPD_1.east) {\includegraphics[ width=0.15\textwidth]{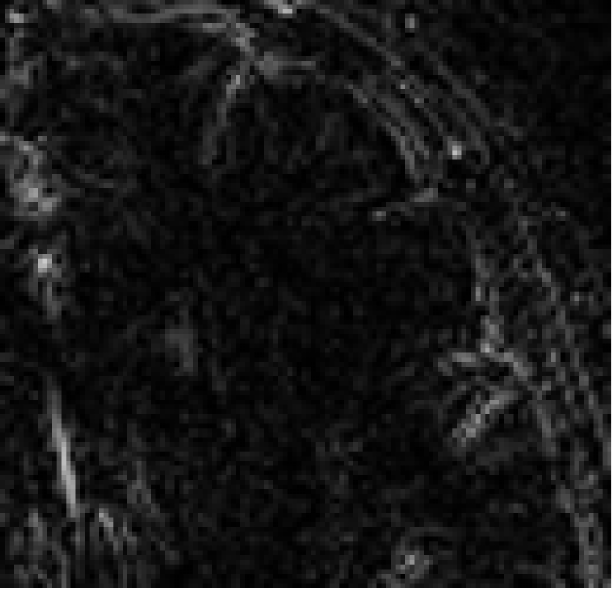}};
    
    \node[inner sep=0pt, anchor = west] (ResiSFISTA_1) at (ResiFISTA_1.east) {\includegraphics[ width=0.15\textwidth]{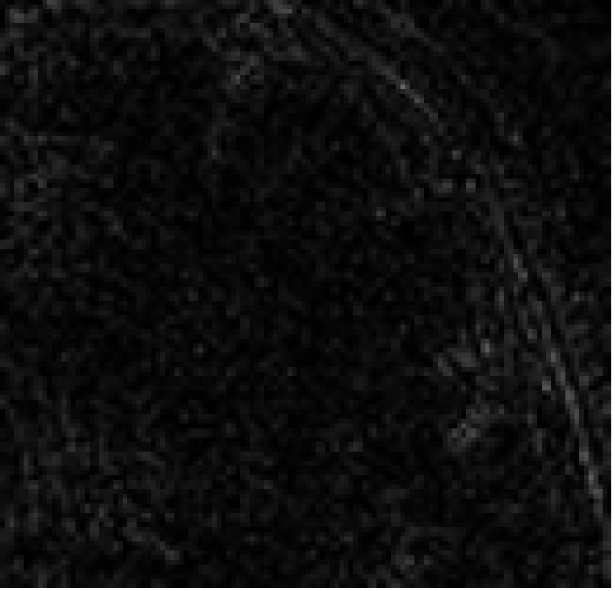}};    
    \end{axis}

        \begin{axis}[at={(ResiPD_1.south west)},anchor = north west,
    xmin = 0,xmax = 216,ymin = 0,ymax = 75, width=1\textwidth,
        scale only axis,
        enlargelimits=false,
        yshift=3.6cm,
      y label style = {yshift = -0.2cm,xshift=-1.5cm},
        axis line style={draw=none},
        tick style={draw=none},
        axis equal image,
        xticklabels={,,},yticklabels={,,},
        ]
    
    \node[inner sep=0pt, anchor = south west] (ResiQNP_1) at (0,0) {\includegraphics[ width=0.15\textwidth]{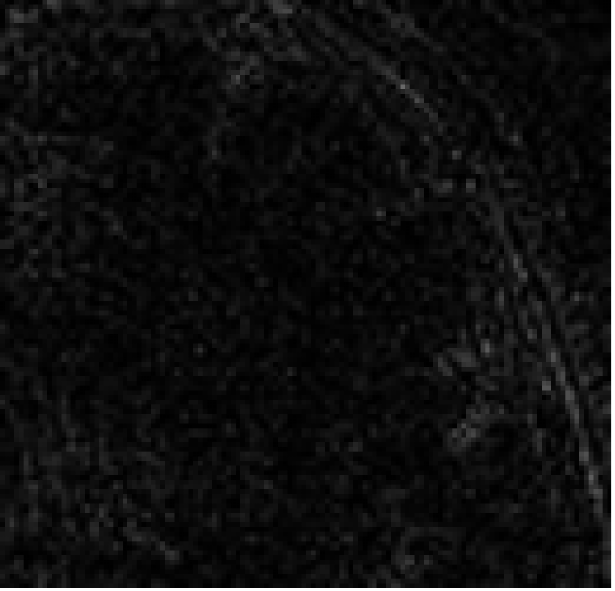}};

    \node at (45,18) {\LARGE \color{red}$\times 5$};
    \end{axis}
\end{tikzpicture} 

\caption{ {\Mcb
 Error maps ($\times 5$) of the zoomed-in in regions of the $16$ reconstructed images in Fig. 11 of the manuscript. From left to right, first row: APM $\rightarrow$ S-APM $\rightarrow$ CQNPM;
 second row: S-CQNPM.}
}
\label{fig:SpiralBrainHighSNR:TVwavelet}
\end{figure}

\begin{figure}[!t]
   \centering
	\begin{tikzpicture}

 \draw [draw=blue] (2.7,-0.1) rectangle (3.62,0.2);

 \draw [-,color=blue](3.2,0.2) -- (3.93,2.2);
 \draw [-,color=blue](3.2,-0.1) -- (4,0);

 \draw [draw=blue] (2.7,-3.3) rectangle (3.62,-3);

 \draw [-,color=blue](3.2,-3) -- (3.93,-1);
 \draw [-,color=blue](3.2,-3.3) -- (3.93,-3.2);
 
 

 	\pgfplotsset{every axis legend/.append style={legend pos=south east,anchor=south east,font=\normalsize, legend cell align={left}}}
    \pgfplotsset{grid style={dotted, gray}}
 	\begin{groupplot}[enlargelimits=false,scale only axis, group style={group size=2 by 2,x descriptions at=edge bottom,group name=mygroup},
 	width=0.2*\textwidth,
 	height=0.12*\textwidth,
 	every axis/.append style={font=\small,title style={anchor=base,yshift=-1mm}, x label style={yshift = 0.5em}, y label style={yshift = -.5em}, grid = both}, legend style={at={(0.9,1)},anchor=north east,font=\footnotesize}
 	]
\nextgroupplot[ylabel={Cost},xlabel={Iteration},xmax = 16,ymax = 145,xtick={0,4,8,12,16},
xticklabels={0,4,8,12,16},tick align=inside,ytick={80,110,145},
yticklabels={80,110,145},
tick pos=left]

     \addplot[dotted,very thick,black,line width=1pt] table [search path={fig/Radial_results/Knee},x expr=\coordindex, y=PD_cost, col sep=comma] {Knee_Wav_PD.csv}; 

     \addplot[dashed,very thick,blue,line width=1pt] table [search path={fig/Radial_results/Knee},x expr=\coordindex, y=fista_cost, col sep=comma] {Knee_Wav_fista.csv};   

     \addplot[solid,very thick,red,line width=1pt] table [search path={fig/Radial_results/Knee},x expr=\coordindex, y=QNP_cost, col sep=comma] {Knee_Wav_QNP.csv};
     
    \legend{PD,APM,CQNPM}; 

\nextgroupplot[ylabel={},xlabel={Iteration},xmin=12,xmax = 16,ymin=59.5,ymax = 62,xtick={12,14,16},
xticklabels={,14,16},tick align=inside,ytick={},yticklabels={},
tick pos=left,xshift={-7mm}]

     \addplot[dotted,very thick,black,line width=1pt] table [search path={fig/Radial_results/Knee},x expr=\coordindex, y=PD_cost, col sep=comma] {Knee_Wav_PD.csv}; 
     
     \addplot[dashed,very thick,blue,line width=1pt] table [search path={fig/Radial_results/Knee},x expr=\coordindex, y=fista_cost, col sep=comma] {Knee_Wav_fista.csv};
     \addplot[solid,very thick,red,line width=1pt] table [search path={fig/Radial_results/Knee},x expr=\coordindex, y=QNP_cost, col sep=comma] {Knee_Wav_QNP.csv};
     
\nextgroupplot[ylabel={Cost},xlabel={CPU Time (Seconds)},xmax = 6,ymax = 145,xtick={0,2,4,6},
xticklabels={0,2,4,6},xtick align=inside,
tick pos=left,ytick={80,110,145},
yticklabels={80,110,145}]

\addplot[dotted,very thick,black,line width=1pt] table [search path={fig/Radial_results/Knee},x = PD_time, y=PD_cost, col sep=comma] {Knee_Wav_PD.csv};

\addplot[dashed,very thick,blue,line width=1pt] table [search path={fig/Radial_results/Knee},x = fista_time, y=fista_cost, col sep=comma] {Knee_Wav_fista.csv};
  
\addplot[solid,very thick,red,line width=1pt] table [search path={fig/Radial_results/Knee},x = QNP_time, y=QNP_cost, col sep=comma] {Knee_Wav_QNP.csv};

\nextgroupplot[xlabel={CPU Time (Seconds)},xmax = 6,xmin = 4.5,xtick={4.5,5,6},
xticklabels={,5,6},xtick align=inside,
tick pos=left,ytick={},yticklabels={},ymin=59.5,ymax = 62]

\addplot[dotted,very thick,black,line width=1pt] table [search path={fig/Radial_results/Knee},x = PD_time, y=PD_cost, col sep=comma] {Knee_Wav_PD.csv};

\addplot[dashed,very thick,blue,line width=1pt] table [search path={fig/Radial_results/Knee},x = fista_time, y=fista_cost, col sep=comma] {Knee_Wav_fista.csv};

\addplot[solid,very thick,red,line width=1pt] table [search path={fig/Radial_results/Knee},x = QNP_time, y=QNP_cost, col sep=comma] {Knee_Wav_QNP.csv};

\end{groupplot}
\end{tikzpicture} 
\caption{\cb Cost values versus iteration (top) and CPU time (bottom) of the knee image with regularizer
$h(\vx) = \|\mT \vx\|_1$ and $\lambda=5\times10^{-4}$
for a left invertible wavelet transform $\mT$
with $5$ levels. Acquisition: radial trajectory with
$96$ projections, $512$ readout points, and $12$ coils. The parameter $\lambda$ was $10^{-3}$.}
 \label{fig:knee:radial:cost}
\end{figure}

\begin{figure*}[!t]
	\centering

\begin{tikzpicture}

   \node (Knee_GT) at (-5,1.3) {\includegraphics[ width=0.22\textwidth]{fig/CropSquare_Knee_GT.pdf}};
   \node at (-6.35,2.9) {\color{white} GT};

    \node (Knee_GT_Crop) at (-1.9,1.3) {\includegraphics[ width=0.11\textwidth]{fig/Crop_Knee_GT.pdf}};
    \node at (-4.4,0.5) {\Large \color{white} (a)};
   \node at (-1.2,0.7) {\Large \color{white} (a)};
  
 	\pgfplotsset{every axis legend/.append style={legend pos=south east,anchor=south east,font=\normalsize, legend cell align={left}}}
    \pgfplotsset{grid style={dotted, gray}}
     \hspace{0.35cm}
 	\begin{groupplot}[enlargelimits=false,scale only axis, group style={group size=1 by 1,x descriptions at=edge bottom,group name=mygroup},
 	width=0.4*\textwidth,
 	height=0.2*\textwidth,
 	every axis/.append style={font=\small,title style={anchor=base,yshift=-1mm}, x label style={yshift = 0.5em}, y label style={yshift = -.5em}, grid = both}, legend style={at={(1,0.5)},anchor=north east}
 	]
     \nextgroupplot[ylabel={PSNR},xlabel={CPU Time (Seconds)},xmax = 6,ymax = 30.7,xtick={0,1,3,6},
xticklabels={0,1,3,6},tick align=inside,
tick pos=left]

    \addplot[solid,very thick,red,line width=1pt] table [search path={fig/Radial_results/Knee},x=QNP_time, y=QNP_PSNR, col sep=comma] {Knee_Wav_QNP.csv};
     \addplot[dashed,very thick,blue,line width=1pt] table [search path={fig/Radial_results/Knee},x=fista_time, y=fista_PSNR, col sep=comma] {Knee_Wav_fista.csv};   
    \addplot[dotted,very thick,black,line width=1pt] table [search path={fig/Radial_results/Knee},x=PD_time, y=PD_PSNR, col sep=comma] {Knee_Wav_PD.csv};  
     
    \addplot[mark=none, darkblue,dashed,line width=1pt]
    coordinates {(3.65,28.9) (6,28.9)};
   \addplot[mark=none, red,dashed,line width=1pt] coordinates {(3.65,19.1988980571864) (3.65,28.9)};
   
     \legend{CQNPM,APM,PD}; 
 \end{groupplot}
\end{tikzpicture}

\vspace{-4.3cm}

\begin{tikzpicture}
    \hspace{1cm}

\begin{axis}[at={(0,0)},anchor = north west, ylabel = APM,
    xmin = 0,xmax = 216,ymin = 0,ymax = 75, width=1\textwidth,
        scale only axis,
        enlargelimits=false,
        yshift = 4.2cm,
      y label style = {yshift = -0.3cm,xshift=-2cm},
        axis line style={draw=none},
        tick style={draw=none},
        axis equal image,
        xticklabels={,,},yticklabels={,,},
        ]
    \node[inner sep=0pt, anchor = south west] (FISTA_1) at (0,0) {\includegraphics[ width=0.2\textwidth]{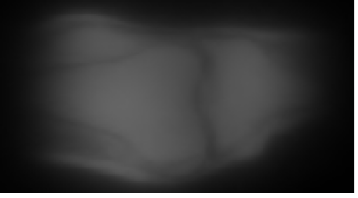}};
    
    \node at (8,22) {\color{white} $\text{iter.}=3$};
    \node at (8,3) {\color{red} $22.15$dB};
    

    \node[inner sep=0pt, anchor = west] (FISTA_2) at (FISTA_1.east) {\includegraphics[ width=0.2\textwidth]{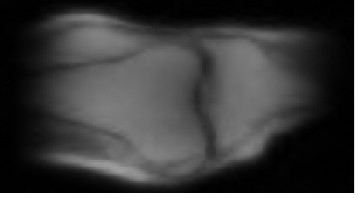}};
    
    \node at (46,22) {\color{white} $10$};
    \node at (52,3) {\color{red} $27.16$dB};
    
    \node[inner sep=0pt, anchor = west] (FISTA_3) at (FISTA_2.east) {\includegraphics[ width=0.2\textwidth]{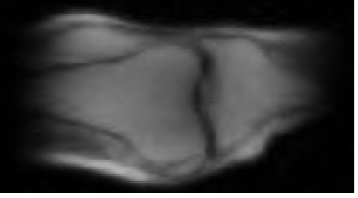}};
    
    \node at (89,22) {\color{white} $13$};
    \node at (95,3) {\color{red} $28.29$dB};

    \node[inner sep=0pt, anchor = west] (FISTA_4) at (FISTA_3.east) {\includegraphics[ width=0.2\textwidth]{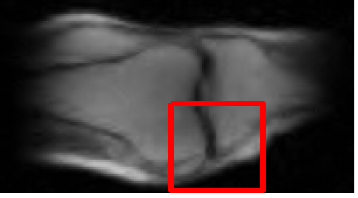}};
    \node at (133,22) {\color{white} $16$};
    \node at (138,3) {\color{red} $29.18$dB};
	\node at (160,4) {\Large \color{white} I};
    \end{axis}

    
    \begin{axis}[at={(FISTA_1.south west)},anchor = north west,ylabel = CQNPM,
    xmin = 0,xmax = 216,ymin = 0,ymax = 75, width=1\textwidth,
        scale only axis,
        enlargelimits=false,
        yshift = 4.3cm,
        y label style = {yshift = -0.3cm,xshift=-2cm},
       axis line style={draw=none},
       tick style={draw=none},
        axis equal image,
        xticklabels={,,},yticklabels={,,}, 
       ]
   \node[inner sep=0pt, anchor = south west] (QN_1) at (0,0) {\includegraphics[ width=0.2\textwidth]{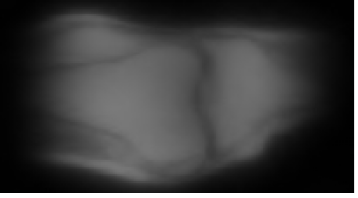}};
    \node at (8,3) {\color{red} $24.62$dB};
    
    \node[inner sep=0pt, anchor = west] (QN_2) at (QN_1.east) {\includegraphics[ width=0.2\textwidth]{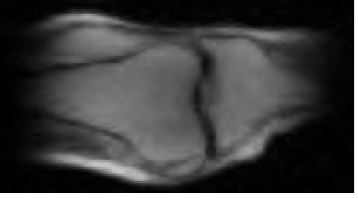}};
    \node at (52,3) {\color{red} $29.35$dB};
    
    \node[inner sep=0pt, anchor = west] (QN_3) at (QN_2.east) {\includegraphics[ width=0.2\textwidth]{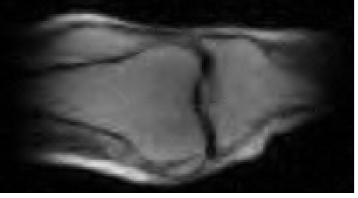}};
    \node at (94,3) {\color{red} $30.3$dB};

    \node[inner sep=0pt, anchor = west] (QN_4) at (QN_3.east) {\includegraphics[ width=0.2\textwidth]{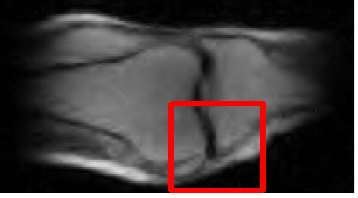}};
     \node at (138,3) {\color{red} $30.44$dB};    
     \node at (160,4) {\Large \color{white} II};
 \end{axis}
 
  \begin{axis}[at={(QN_1.south west)},anchor = north west,
    xmin = 0,xmax = 216,ymin = 0,ymax = 75, width=1\textwidth,
        scale only axis,
        enlargelimits=false,
        yshift = 4.3cm,
       axis line style={draw=none},
       tick style={draw=none},
        axis equal image,
        xticklabels={,,},yticklabels={,,}
       ]
       
   \node[inner sep=0pt,anchor = south west] (Resi_FISTA) at (0,0) {\includegraphics[ width=0.11\textwidth]{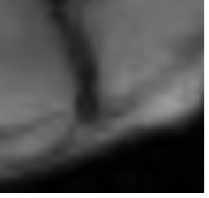}};    
   
       \node[inner sep=18pt, anchor = west] (Resi_FISTA_1) at (Resi_FISTA.east) {\includegraphics[ width=0.11\textwidth]{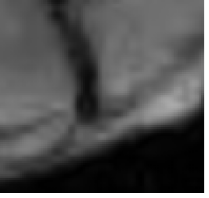}};

 \node[inner sep=18pt, anchor = west] (Resi_FISTA_2) at (Resi_FISTA_1.east) {\includegraphics[ width=0.11\textwidth]{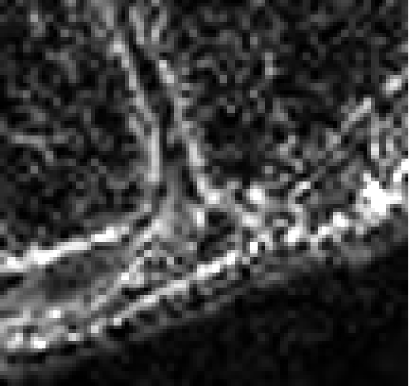}};

    \node[inner sep=0pt, anchor = west] (Resi_FISTA_3) at (Resi_FISTA_2.east) {\includegraphics[ width=0.11\textwidth]{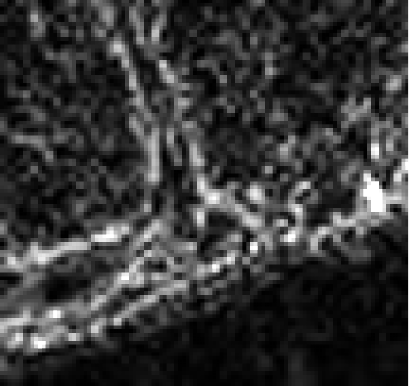}};
    \node at (22,3.5) {\Large \color{white} I};
    \node at (53,3.5) {\Large \color{white} II};
    \node at (93,3.5) {\Large \color{white} I};
      \node at (123,3.5) {\Large \color{white} II};
    \node at (132,10) {\Large \color{red}$\times 5$};
 \end{axis}

\end{tikzpicture} 

\caption{
 { First row: the groundth image and PSNR values versus CPU time; second to third row: the reconstructed knee images at $3$, $10$, $13$, and $16$th iteration with \Cref{fig:knee:radial:cost} setting; fouth row: the zoomed-in regions and the corresponding error maps ($\times 5$) of the $16$th iteration reconstructed images.}
}
\label{fig:RadialKnee:wavelet}
\end{figure*}

%% file: Figs/KneeWaveTVFigRadial.tex

\begin{figure}[!t]
   \centering
	\begin{tikzpicture}

 \draw [draw=blue] (2.7,-0.1) rectangle (3.62,0.2);

 \draw [-,color=blue](3,0.2) -- (3.93,2.15);
\draw [-,color=blue](3,-0.1) -- (3.93,0);

 \draw [draw=blue] (2.9,-3.3) rectangle (3.62,-3);

 \draw [-,color=blue](3.2,-3) -- (3.93,-1);
  \draw [-,color=blue](3.2,-3.3) -- (3.93,-3.2);
 

 	\pgfplotsset{every axis legend/.append style={legend pos=south east,anchor=south east,font=\normalsize, legend cell align={left}}}
    \pgfplotsset{grid style={dotted, gray}}
 	\begin{groupplot}[enlargelimits=false,scale only axis, group style={group size=2 by 2,x descriptions at=edge bottom,group name=mygroup},
 	width=0.2*\textwidth,
 	height=0.12*\textwidth,
 	every axis/.append style={font=\small,title style={anchor=base,yshift=-1mm}, x label style={yshift = 0.5em}, y label style={yshift = -.5em}, grid = both}, legend style={at={(0.75,1)},anchor=north east,font=\tiny,text opacity = 1,fill opacity=0.6}
 	]
  
\nextgroupplot[ylabel={Cost},xlabel={Iteration},xmax = 16,ymin=59,ymax = 145,xtick={0,4,8,12,16},
xticklabels={0,4,8,12,16},tick align=inside,ytick={60,100,145},
yticklabels={60,100,145},
tick pos=left,mark repeat = 5,mark size = 3pt]
    
 \addplot[dotted,very thick,black,line width=1pt] table [search path={fig/Radial_results/Knee},x expr=\coordindex, y=PD_cost, col sep=comma] {Knee_WavTV_PD.csv};
 \addplot[dash pattern=on 2pt off 3pt on 4pt off 4pt,very thick,goldenbrown,line width=1pt] table [search path={fig/Radial_results/Knee},x expr=\coordindex, y=ADMM_cost, col sep=comma] {Knee_WavTV_ADMM.csv};
    
\addplot[dashed,very thick,blue,line width=1pt] table [search path={fig/Radial_results/Knee},x expr=\coordindex, y=fista_cost, col sep=comma] {Knee_WavTV_fista.csv};  
\addplot[dashed,mark=star,very thick,darkblue,line width=1pt] table [search path={fig/Radial_results/Knee},x expr=\coordindex, y=fista_Smooth_cost, col sep=comma] {Knee_WavTV_fistaSmooth.csv}; 
    
\addplot[solid,very thick,red,line width=1pt] table [search path={fig/Radial_results/Knee},x expr=\coordindex, y=QNP_cost, col sep=comma] {Knee_WavTV_QNP.csv};
\addplot[solid,mark=star,very thick,red,line width=1pt] table [search path={fig/Radial_results/Knee},x expr=\coordindex, y=QNP_Smooth_cost, col sep=comma] {Knee_WavTV_QNPSmooth.csv};

\legend{PD,ADMM,APM,S-APM,CQNPM,S-CQNPM};

\nextgroupplot[ylabel={},xlabel={Iteration},xmin=12,xmax = 16,ymin=59,ymax = 62.1,xtick={12,14,16},
xticklabels={,14,16},tick align=inside,ytick={},yticklabels={},
tick pos=left,xshift={-7mm}]

 \addplot[dotted,very thick,black,line width=1pt] table [search path={fig/Radial_results/Knee},x expr=\coordindex, y=PD_cost, col sep=comma] {Knee_WavTV_PD.csv};
 \addplot[dash pattern=on 2pt off 3pt on 4pt off 4pt,very thick,goldenbrown,line width=1pt] table [search path={fig/Radial_results/Knee},x expr=\coordindex, y=ADMM_cost, col sep=comma] {Knee_WavTV_ADMM.csv};

\addplot[dashed,very thick,blue,line width=1pt] table [search path={fig/Radial_results/Knee},x expr=\coordindex, y=fista_cost, col sep=comma] {Knee_WavTV_fista.csv};   
\addplot[dashed,mark=star,very thick,darkblue,line width=1pt] table [search path={fig/Radial_results/Knee},x expr=\coordindex, y=fista_Smooth_cost, col sep=comma] {Knee_WavTV_fistaSmooth.csv}; 

\addplot[solid,very thick,red,line width=1pt] table [search path={fig/Radial_results/Knee},x expr=\coordindex, y=QNP_cost, col sep=comma] {Knee_WavTV_QNP.csv};
\addplot[solid,mark=star,very thick,red,line width=1pt] table [search path={fig/Radial_results/Knee},x expr=\coordindex, y=QNP_Smooth_cost, col sep=comma] {Knee_WavTV_QNPSmooth.csv};

\nextgroupplot[xlabel={CPU Time (Seconds)},xmax = 12,ymax = 145,xtick={0,4,8,12},
xticklabels={0,4,8,12},xtick align=inside,
ylabel={Cost},ytick={60,100,145},
yticklabels={60,100,145},mark repeat = 5,mark size = 3pt,ylabel={Cost}]

 \addplot[dotted,very thick,black,line width=1pt] table [search path={fig/Radial_results/Knee},x = PD_time, y=PD_cost, col sep=comma] {Knee_WavTV_PD.csv};
 \addplot[dash pattern=on 2pt off 3pt on 4pt off 4pt,very thick,goldenbrown,line width=1pt] table [search path={fig/Radial_results/Knee},x = ADMM_time, y=ADMM_cost, col sep=comma] {Knee_WavTV_ADMM.csv};

\addplot[dashed,very thick,blue,line width=1pt] table [search path={fig/Radial_results/Knee},x = fista_time, y=fista_cost, col sep=comma] {Knee_WavTV_fista.csv};
\addplot[dashed,mark=star,very thick,darkblue,line width=1pt] table [search path={fig/Radial_results/Knee},x =fista_Smooth_time, y=fista_Smooth_cost, col sep=comma] {Knee_WavTV_fistaSmooth.csv}; 

\addplot[solid,very thick,red,line width=1pt] table [search path={fig/Radial_results/Knee},x = QNP_time, y=QNP_cost, col sep=comma] {Knee_WavTV_QNP.csv};

\addplot[solid,mark=star,very thick,red,line width=1pt] table [search path={fig/Radial_results/Knee},x = QNP_Smooth_time, y=QNP_Smooth_cost, col sep=comma] {Knee_WavTV_QNPSmooth.csv};

\nextgroupplot[xlabel={CPU Time (Seconds)},xmax = 12,xmin = 10,xtick={0,11,12},
xticklabels={,11,12},xtick align=inside,
tick pos=left,ytick={},yticklabels={},ymin=59.35,ymax = 60.9]

 \addplot[dotted,very thick,black,line width=1pt] table [search path={fig/Radial_results/Knee},x = PD_time, y=PD_cost, col sep=comma] {Knee_WavTV_PD.csv};
 \addplot[dash pattern=on 2pt off 3pt on 4pt off 4pt,very thick,goldenbrown,line width=1pt] table [search path={fig/Radial_results/Knee},x = ADMM_time, y=ADMM_cost, col sep=comma] {Knee_WavTV_ADMM.csv};

\addplot[dashed,very thick,blue,line width=1pt] table [search path={fig/Radial_results/Knee},x = fista_time, y=fista_cost, col sep=comma] {Knee_WavTV_fista.csv};
\addplot[dashed,mark=star,very thick,darkblue,line width=1pt] table [search path={fig/Radial_results/Knee},x =fista_Smooth_time, y=fista_Smooth_cost, col sep=comma] {Knee_WavTV_fistaSmooth.csv}; 

\addplot[solid,very thick,red,line width=1pt] table [search path={fig/Radial_results/Knee},x = QNP_time, y=QNP_cost, col sep=comma] {Knee_WavTV_QNP.csv};

\addplot[solid,mark=star,very thick,red,line width=1pt] table [search path={fig/Radial_results/Knee},x = QNP_Smooth_time, y=QNP_Smooth_cost, col sep=comma] {Knee_WavTV_QNPSmooth.csv};

\end{groupplot}
\end{tikzpicture} 
\caption{\cb Cost values versus iteration (top) and CPU time (bottom) of the knee image with regularizer
$h(\vx)=\alpha \|\mT\vx\|_1+(1-\alpha)\mrm{TV}(\vx)$ and same acquisition as \Cref{fig:knee:radial:cost}.
The parameters $\lambda$ and $\alpha$ were $10^{-3}$ and $\frac{1}{2}$.}
 \label{fig:knee:radial:WavTV:cost}
\end{figure}

\begin{figure*}[!t]
	\centering
\begin{tikzpicture}
  
  \node (Knee_GT) at (-5,1.3) {\includegraphics[ width=0.22\textwidth]{fig/CropSquare_Knee_GT.pdf}};
   \node at (-6.35,2.9) {\color{white} GT};

    \node (Knee_GT_Crop) at (-1.9,1.3) {\includegraphics[ width=0.11\textwidth]{fig/Crop_Knee_GT.pdf}};

  \node at (-4.4,0.5) {\Large \color{white} (a)};
  \node at (-1.2,0.7) {\Large \color{white} (a)};
 	\pgfplotsset{every axis legend/.append style={legend pos=south east,anchor=south east,font=\normalsize, legend cell align={left}}}
    \pgfplotsset{grid style={dotted, gray}}

     \hspace{0.35cm}

 	\begin{groupplot}[enlargelimits=false,scale only axis, group style={group size=1 by 1,x descriptions at=edge bottom,group name=mygroup},
 	width=0.5*\textwidth,
 	height=0.2*\textwidth,
 	every axis/.append style={font=\small,title style={anchor=base,yshift=-1mm}, x label style={yshift = 0.5em}, y label style={yshift = -.5em}, grid = both}, legend style={at={(1,0.6)},anchor=north east,font=\tiny}
 	]
\nextgroupplot[ylabel={PSNR},xlabel={CPU Time (Seconds)},xmax = 12,ymax = 31.8,xtick={0,4,8,12},
xticklabels={0,4,8,12},tick align=inside,
tick pos=left,mark repeat = 5,mark size = 3pt]


         \addplot[solid,mark=star,very thick,red,line width=1pt] table [search path={fig/Radial_results/Knee},x =QNP_Smooth_time, y=QNP_Smooth_PSNR, col sep=comma] {Knee_WavTV_QNPSmooth.csv};
    \addplot[solid,very thick,red,line width=1pt] table [search path={fig/Radial_results/Knee},x =QNP_time, y=QNP_PSNR, col sep=comma] {Knee_WavTV_QNP.csv};
    \addplot[dashed,very thick,blue,line width=1pt] table [search path={fig/Radial_results/Knee},x =fista_time, y=fista_PSNR, col sep=comma] {Knee_WavTV_fista.csv}; 
    \addplot[dashed,mark=star,very thick,darkblue,line width=1pt] table [search path={fig/Radial_results/Knee},x =fista_Smooth_time, y=fista_Smooth_PSNR, col sep=comma] {Knee_WavTV_fistaSmooth.csv}; 
        
 \addplot[dotted,very thick,black,line width=1pt] table [search path={fig/Radial_results/Knee},x =PD_time, y=PD_PSNR, col sep=comma] {Knee_WavTV_PD.csv};
   \addplot[mark=none, darkblue,dashed,line width=1pt]
    coordinates {(5.25,30.1) (12,30.1)};
   \addplot[mark=none, red,dashed,line width=1pt] coordinates {(5.25,19.1988972869329) (5.25,30.1)};
   \addplot[mark=none, red,dashed,line width=1pt] coordinates {(6.95,19.1988972869329) (6.95,30.1)};
   
    \legend{S-CQNPM,CQNPM,APM,S-APM,PD};  
 \end{groupplot}
\end{tikzpicture}

\vspace{-4cm}

\begin{tikzpicture}

    \hspace{0.8cm} 

\begin{axis}[at={(0,0)},anchor = north west, ylabel = PD,
    xmin = 0,xmax = 216,ymin = 0,ymax = 75, width=1\textwidth,
        scale only axis,
        enlargelimits=false,
        yshift=4.1cm,
      y label style = {yshift = -0.2cm,xshift=-2cm},
        axis line style={draw=none},
        tick style={draw=none},
        axis equal image,
        xticklabels={,,},yticklabels={,,},
        ]
    \node[inner sep=0pt, anchor = south west] (PD_1) at (0,0) {\includegraphics[ width=0.22\textwidth]{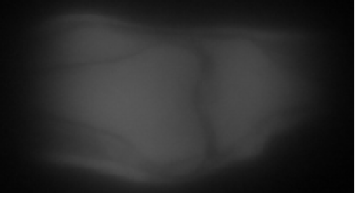}};
    
    \node at (9,24) {\color{white} $\text{iter.}=3$};
    \node at (9,3) {\color{red} $21.09$dB};
    

    \node[inner sep=0pt, anchor = west] (PD_2) at (PD_1.east) {\includegraphics[ width=0.22\textwidth]{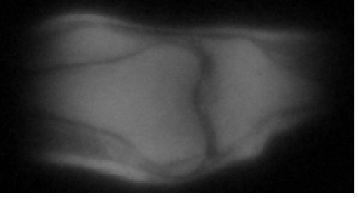}};
    
    \node at (51,24) {\color{white} $10$};
    \node at (56,3) {\color{red} $25.52$dB};
    
    \node[inner sep=0pt, anchor = west] (PD_3) at (PD_2.east) {\includegraphics[ width=0.22\textwidth]{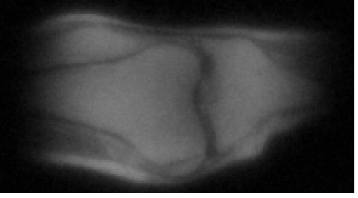}};
    
    \node at (99,24) {\color{white} $13$};
    \node at (103,3) {\color{red} $26.26$dB};

    \node[inner sep=0pt, anchor = west] (PD_4) at (PD_3.east) {\includegraphics[ width=0.22\textwidth]{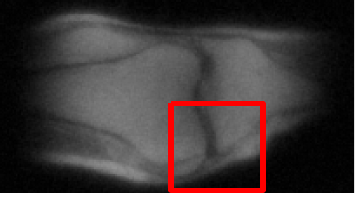}};
    \node at (148,24) {\color{white} $16$};
    \node at (152,3) {\color{red} $26.77$dB};
   \node at (176,4) {\Large \color{white} I};
    \end{axis}
    
     \begin{axis}[at={(PD_1.south west)},anchor = north west, ylabel = APM,
    xmin = 0,xmax = 216,ymin = 0,ymax = 75, width=1\textwidth,
        scale only axis,
        enlargelimits=false,
      yshift= 4.05cm,
      y label style = {yshift = -0.2cm,xshift=-2cm},
        axis line style={draw=none},
        tick style={draw=none},
        axis equal image,
        xticklabels={,,},yticklabels={,,},
        ]
    \node[inner sep=0pt, anchor = south west] (FISTA_1) at (0,0) {\includegraphics[ width=0.22\textwidth]{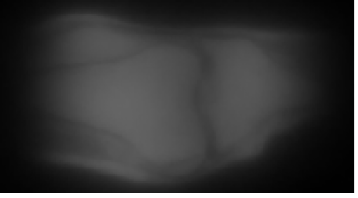}};
    
    \node at (9,3) {\color{red} $22.16$dB};
    

    \node[inner sep=0pt, anchor = west] (FISTA_2) at (FISTA_1.east) {\includegraphics[ width=0.22\textwidth]{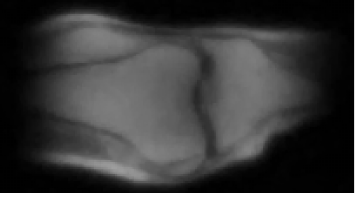}};
    
    \node at (56,3) {\color{red} $27.42$dB};
    
    \node[inner sep=0pt, anchor = west] (FISTA_3) at (FISTA_2.east) {\includegraphics[ width=0.22\textwidth]{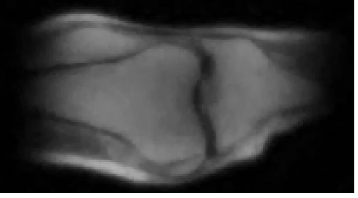}};
    
    \node at (103,3) {\color{red} $28.69$dB};

    \node[inner sep=0pt, anchor = west] (FISTA_4) at (FISTA_3.east) {\includegraphics[ width=0.22\textwidth]{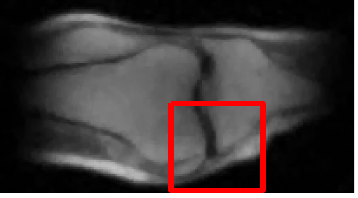}};
    \node at (152,3) {\color{red} $29.74$dB};
    \node at (175,4) {\Large \color{white} II};
    \end{axis}

\begin{axis}[at={(FISTA_1.south west)},anchor =  north west, ylabel = S-APM,
    xmin = 0,xmax = 216,ymin = 0,ymax = 75, width=1\textwidth,
        scale only axis,
        enlargelimits=false,
      yshift= 4.05cm,
        y label style = {yshift = -0.2cm,xshift=-2cm},
        axis line style={draw=none},
        tick style={draw=none},
        axis equal image,
        xticklabels={,,},yticklabels={,,},
        ]
    \node[inner sep=0pt, anchor = south west] (S_FISTA_1) at (0,0) {\includegraphics[ width=0.22\textwidth]{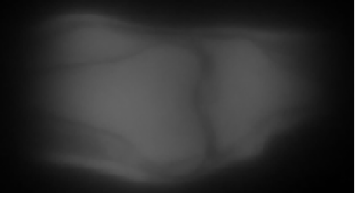}}; 
    \node at (9,3) {\color{red} $22.17$dB};
    
    \node[inner sep=0pt, anchor = west] (S_FISTA_2) at (S_FISTA_1.east) {\includegraphics[ width=0.22\textwidth]{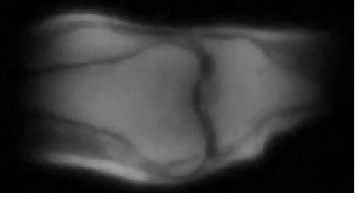}};

    \node at (56,3) {\color{red} $27.42$dB};

     \node[inner sep=0pt, anchor = west] (S_FISTA_3) at (S_FISTA_2.east) {\includegraphics[ width=0.22\textwidth]{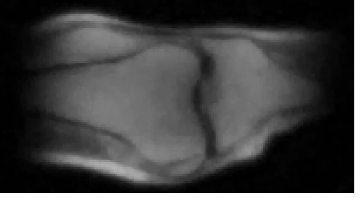}};
     \node at (103,3) {\color{red} $28.68$dB};
     
    \node[inner sep=0pt, anchor = west] (S_FISTA_4) at (S_FISTA_3.east) {\includegraphics[width=0.22\textwidth]{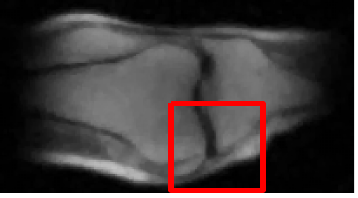}};
    \node at (152,3) {\color{red} $29.71$dB};
    
     \node at (175,4) {\Large \color{white} III};
    \end{axis}

\begin{axis}[at={(S_FISTA_1.south west)},anchor = north west,ylabel = CQNPM,
    xmin = 0,xmax = 216,ymin = 0,ymax = 75, width=1\textwidth,
        scale only axis,
        enlargelimits=false,
        yshift= 4.05cm,
        y label style = {yshift = -0.2cm,xshift=-2cm},
       axis line style={draw=none},
       tick style={draw=none},
        axis equal image,
        xticklabels={,,},yticklabels={,,}
       ]
   \node[inner sep=0pt, anchor = south west] (QN_1) at (0,0) {\includegraphics[ width=0.22\textwidth]{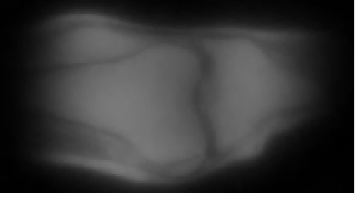}};
    \node at (9,3) {\color{red} $24.69$dB};
    
    \node[inner sep=0pt, anchor = west] (QN_2) at (QN_1.east) {\includegraphics[ width=0.22\textwidth]{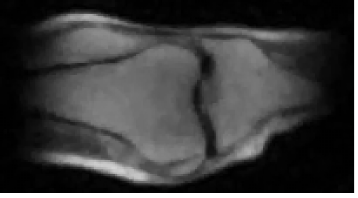}};
    \node at (56,3) {\color{red} $30.57$dB};
    
    \node[inner sep=0pt, anchor = west] (QN_3) at (QN_2.east) {\includegraphics[ width=0.22\textwidth]{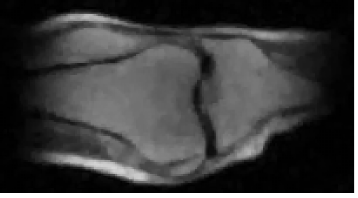}};
    \node at (103,3) {\color{red} $31.16$dB};

    \node[inner sep=0pt, anchor = west] (QN_4) at (QN_3.east) {\includegraphics[ width=0.22\textwidth]{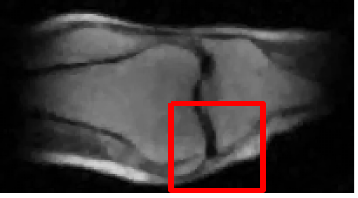}};
     \node at (152,3) {\color{red} $31.61$dB};    
     
           \node at (175,4) {\Large \color{white} IV};
 \end{axis}

    \begin{axis}[at={(QN_1.south west)},anchor = north west,ylabel = S-CQNPM,
    xmin = 0,xmax = 216,ymin = 0,ymax = 75, width=1\textwidth,
        scale only axis,
        enlargelimits=false,
         yshift= 4.05cm,
        y label style = {yshift = -0.2cm,xshift=-2cm},
       axis line style={draw=none},
       tick style={draw=none},
        axis equal image,
        xticklabels={,,},yticklabels={,,}
       ]
       
   \node[inner sep=0pt, anchor = south west] (S_QN_1) at (0,0) {\includegraphics[ width=0.22\textwidth]{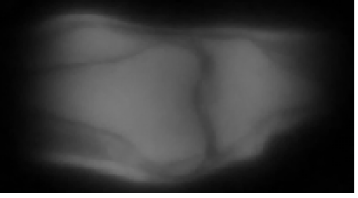}};
   
    \node at (9,3) {\color{red} $24.7$dB};

    \node[inner sep=0pt, anchor = west] (S_QN_2) at (S_QN_1.east) {\includegraphics[ width=0.22\textwidth]{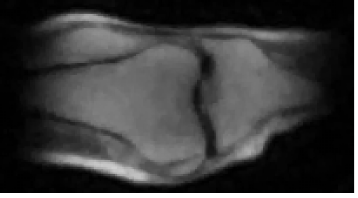}};
    \node at (56,3) {\color{red} $30.47$dB};
  
    \node[inner sep=0pt, anchor = west] (S_QN_3) at (S_QN_2.east) {\includegraphics[ width=0.22\textwidth]{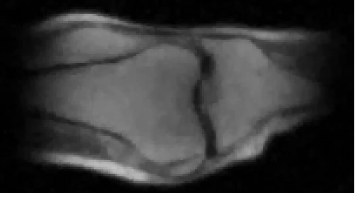}};
    
     \node at (103,3) {\color{red} $30.04$dB};
    \node[inner sep=0pt, anchor = west] (S_QN_4) at (S_QN_3.east) {\includegraphics[ width=0.22\textwidth]{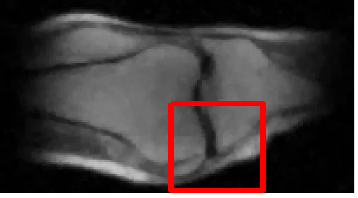}};
    \node at (152,3) {\color{red} $30.47$dB};  
   
      \node at (175,4) {\Large \color{white} V};
 \end{axis}
 
     \begin{axis}[at={(S_QN_1.south west)},anchor = north west, 
    xmin = 0,xmax = 216,ymin = 0,ymax = 75, width=1\textwidth,
        scale only axis,
        enlargelimits=false,
        yshift=4.38cm,
      y label style = {yshift = -0.2cm,xshift=-2cm},
        axis line style={draw=none},
        tick style={draw=none},
        axis equal image,
        xticklabels={,,},yticklabels={,,},
        ]
        
    \node[inner sep=0pt, anchor = south west] (PDZoom_1) at (0,0) {\includegraphics[ width=0.11\textwidth]{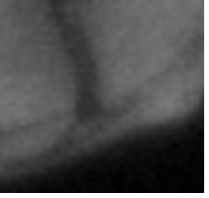}};
    
    \node[inner sep=0pt,yshift=-27pt,anchor = south west] (FISTAZoom_1) at (PDZoom_1.east) {\includegraphics[width=0.11\textwidth]{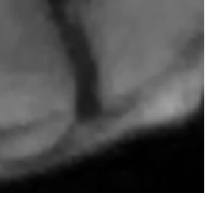}};
    
    \node[inner sep=0pt,yshift=-27pt,anchor = south west] (SFISTAZoom_1) at (FISTAZoom_1.east) {\includegraphics[width=0.11\textwidth]{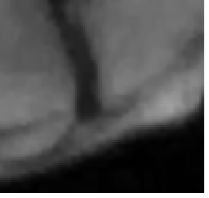}};
    
    \node[inner sep=0pt,yshift=-27pt,anchor = south west] (QNPZoom_1) at (SFISTAZoom_1.east) {\includegraphics[width=0.11\textwidth]{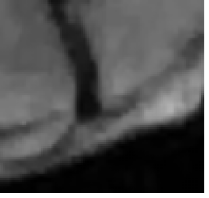}};

    \node[inner sep=0pt,yshift=-27pt,anchor = south west] (SQNPZoom_2) at (QNPZoom_1.east) {\includegraphics[width=0.11\textwidth]{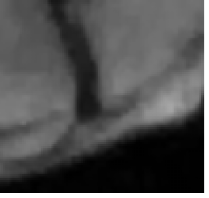}};
    
    \node at (22,3.5) {\Large \color{white} I};
    \node at (45,3.5) {\Large \color{white} II};
    \node at (68,3.5) {\Large \color{white} III};
    \node at (92,3.5) {\Large \color{white} IV};
    \node at (116,3.5) {\Large \color{white} V};
    \end{axis}

     \begin{axis}[at={(PDZoom_1.south west)},anchor = north west, 
    xmin = 0,xmax = 216,ymin = 0,ymax = 75, width=1\textwidth,
        scale only axis,
        yshift=4.38cm,
        enlargelimits=false,
      y label style = {yshift = -0.2cm,xshift=-2cm},
        axis line style={draw=none},
        tick style={draw=none},
        axis equal image,
        xticklabels={,,},yticklabels={,,},
        ]
        
    \node[inner sep=0pt, anchor = south west] (ResiPDZoom_1) at (0,0) {\includegraphics[ width=0.11\textwidth]{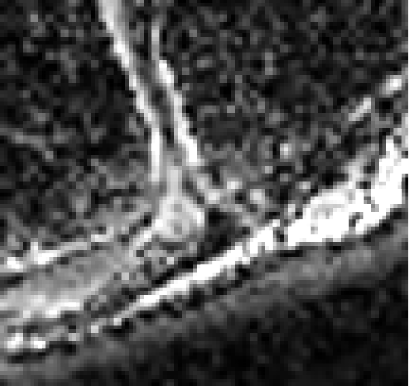}};
    
    \node[inner sep=0pt,yshift=-27pt,anchor = south west] (ResiFISTAZoom_1) at (ResiPDZoom_1.east) {\includegraphics[width=0.11\textwidth]{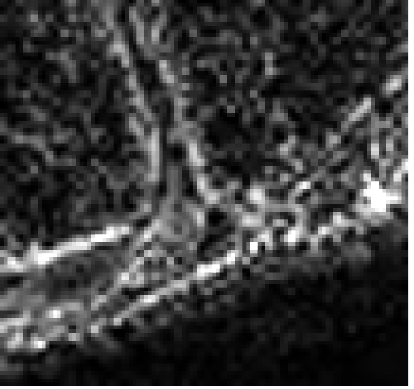}};

    \node[inner sep=0pt,yshift=-27pt,anchor = south west] (ResiSFISTAZoom_2) at (ResiFISTAZoom_1.east) {\includegraphics[width=0.11\textwidth]{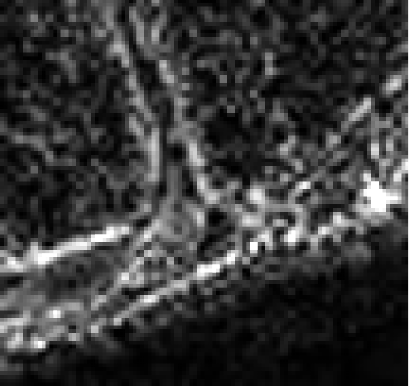}};
    
    \node[inner sep=0pt,yshift=-27pt,anchor = south west] (ResiQNPZoom_1) at (ResiSFISTAZoom_2.east) {\includegraphics[width=0.11\textwidth]{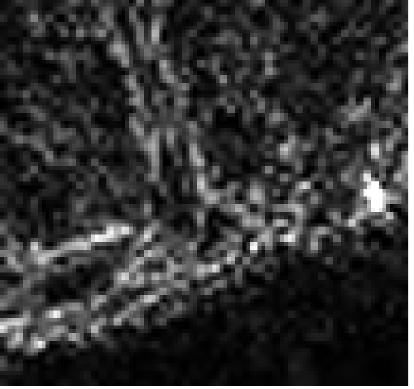}};

    \node[inner sep=0pt,yshift=-27pt,anchor = south west] (ResiSQNPZoom_2) at (ResiQNPZoom_1.east) {\includegraphics[width=0.11\textwidth]{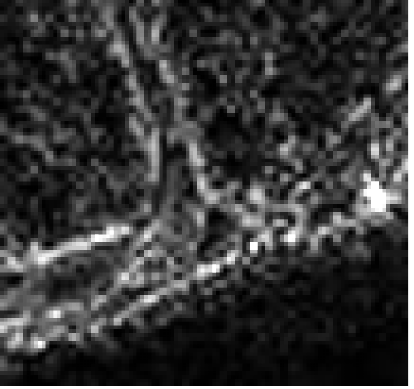}};
     \node at (125,13) {\LARGE \color{red}$\times 5$};
     
    \node at (22,3.5) {\Large \color{white} I};
    \node at (45,3.5) {\Large \color{white} II};
    \node at (68,3.5) {\Large \color{white} III};
    \node at (92,3.5) {\Large \color{white} IV};
    \node at (116,3.5) {\Large \color{white} V};
\end{axis}

\end{tikzpicture} 

\caption{ {First row: the ground truth image and
PSNR values versus CPU time; second to sixth row:  the reconstructed knee images at $3$, $10$, $13$, and $16$th iteration with \Cref{fig:knee:radial:WavTV:cost} setting. We did not show the reconstructed image of ADMM since it yielded a much lower PSNR than other methods. The seventh and eighth rows represent the zoomed-in regions and the corresponding error mapps ($\times 5$) of the $16$th iteration reconstructed images with PD $\rightarrow$ APM $\rightarrow$  S-APM $\rightarrow$  CQNPM $\rightarrow$  S-CQNPM.} }
\label{fig:RadialKnee:WavTV}
\end{figure*}

%% file: Figs/BrainWaveFigSpiral.tex
\begin{figure}[!t]
   \centering
	\begin{tikzpicture}

 \draw [draw=blue] (2.6,-0.1) rectangle (3.62,0.2);

 \draw [-,color=blue](3,0.2) -- (3.93,2.15);
\draw [-,color=blue](3,-0.1) -- (3.93,0);

 \draw [draw=blue] (2.7,-3.3) rectangle (3.62,-3);
 
 \draw [-,color=blue](3.2,-3) -- (3.93,-1);
  \draw [-,color=blue](3.2,-3.3) -- (3.93,-3.2);
 
 

 	\pgfplotsset{every axis legend/.append style={legend pos=south east,anchor=south east,font=\normalsize, legend cell align={left}}}
    \pgfplotsset{grid style={dotted, gray}}
 	\begin{groupplot}[enlargelimits=false,scale only axis, group style={group size=2 by 2,x descriptions at=edge bottom,group name=mygroup},
 	width=0.2*\textwidth,
 	height=0.12*\textwidth,
 	every axis/.append style={font=\small,title style={anchor=base,yshift=-1mm}, x label style={yshift = 0.5em}, y label style={yshift = -.5em}, grid = both}, legend style={at={(0.8,1)},anchor=north east,font=\footnotesize}
 	]
\nextgroupplot[ylabel={Cost},xlabel={Iteration},xmax = 16,ymax = 105,xtick={0,4,8,12,16},
xticklabels={0,4,8,12,16},tick align=inside,ytick={70,85,105},yticklabels={70,85,105},
tick pos=left]

   \addplot[dotted,very thick,black,line width=1pt] table [search path={fig/Spiral_results/Brain},x expr=\coordindex, y=PD_cost, col sep=comma] {Brain_Wav_PD.csv}; 
   
   \addplot[dashed,very thick,blue,line width=1pt] table [search path={fig/Spiral_results/Brain},x expr=\coordindex, y=fista_cost, col sep=comma] {Brain_Wav_fista.csv}; 
     \addplot[solid,very thick,red,line width=1pt] table [search path={fig/Spiral_results/Brain},x expr=\coordindex, y=QNP_cost, col sep=comma] {Brain_Wav_QNP.csv};       
 
    \legend{PD,APM,CQNPM};  

\nextgroupplot[ylabel={},xlabel={Iteration},xmin=12,xmax = 16,ymin= 64.9,ymax = 66.5,xtick={12,14,16},
xticklabels={,14,16},tick align=inside,ytick={},yticklabels={},
tick pos=left,xshift={-7mm}]

   \addplot[dotted,very thick,black,line width=1pt] table [search path={fig/Spiral_results/Brain},x expr=\coordindex, y=PD_cost, col sep=comma] {Brain_Wav_PD.csv}; 
   
     \addplot[dashed,very thick,blue,line width=1pt] table [search path={fig/Spiral_results/Brain},x expr=\coordindex, y=fista_cost, col sep=comma] {Brain_Wav_fista.csv};   
     \addplot[solid,very thick,red,line width=1pt] table [search path={fig/Spiral_results/Brain},x expr=\coordindex, y=QNP_cost, col sep=comma] {Brain_Wav_QNP.csv};
     
\nextgroupplot[xlabel={CPU Time (Seconds)},xmax = 5.5,ymax = 105,xtick={0,1,3,5},
xticklabels={0,1,3,5},xtick align=inside,
tick pos=left,ytick={70,85,105},yticklabels={70,85,105},ylabel={Cost}]

   \addplot[dotted,very thick,black,line width=1pt] table [search path={fig/Spiral_results/Brain},x = PD_time, y=PD_cost, col sep=comma] {Brain_Wav_PD.csv}; 
\addplot[dashed,very thick,blue,line width=1pt] table [search path={fig/Spiral_results/Brain},x = fista_time, y=fista_cost, col sep=comma] {Brain_Wav_fista.csv};
  
\addplot[solid,very thick,red,line width=1pt] table [search path={fig/Spiral_results/Brain},x = QNP_time, y=QNP_cost, col sep=comma] {Brain_Wav_QNP.csv};

\nextgroupplot[xlabel={CPU Time (Seconds)},xmax = 5.5,xmin = 4,xtick={4,5,5.5},
xticklabels={,5,5.5},xtick align=inside,
tick pos=left,ytick={},yticklabels={},ymin=65,ymax = 66.5]

\addplot[dashed,very thick,blue,line width=1pt] table [search path={fig/Spiral_results/Brain},x = fista_time, y=fista_cost, col sep=comma] {Brain_Wav_fista.csv};

\addplot[solid,very thick,red,line width=1pt] table [search path={fig/Spiral_results/Brain},x = QNP_time, y=QNP_cost, col sep=comma] {Brain_Wav_QNP.csv};

\end{groupplot}
\end{tikzpicture} 
\caption{\cb Cost values versus iteration (top) and CPU time (bottom) of the brain image with regularizer
$h(\vx) = \|\mT \vx\|_1$ and $\lambda=10^{-3}$
for a left invertible wavelet transform $\mT$
with $5$ levels. Acquisition: spiral trajectory with
$32$ intervals, $1688$ readout points, and $12$ coils.}
 \label{fig:brain:spiral:cost}
\end{figure}
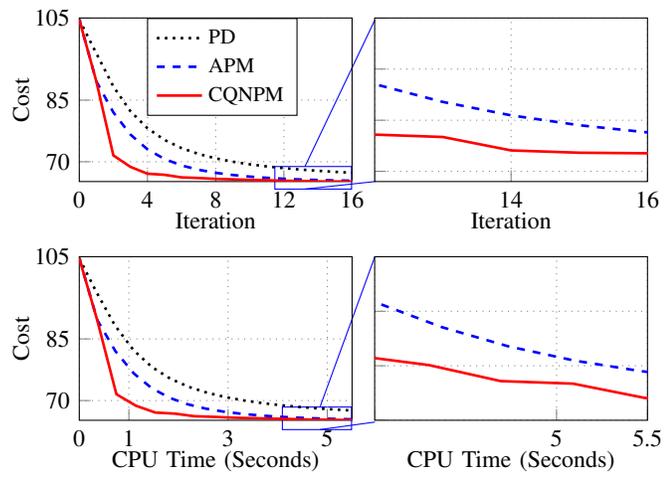

\begin{figure*}[!t]
	\centering
 
\begin{tikzpicture}

 \node (Brain_GT) at (-5,1.3) {\includegraphics[ width=0.18\textwidth]{fig/CropSquare_Brain_GT.pdf}};
   \node at (-6.35,2.9) {\color{white} GT};

    \node (Brain_GT_Crop) at (-1.9,1.3) {\includegraphics[ width=0.15\textwidth]{fig/Crop_Brain_GT.pdf}};

       \node at (-3.9,2.63) {\color{white} (a)};
   \node at (-0.8,2.35) {\color{white} (a)};

 	\pgfplotsset{every axis legend/.append style={legend pos=south east,anchor=south east,font=\normalsize, legend cell align={left}}}
    \pgfplotsset{grid style={dotted, gray}}

      \hspace{0.45cm}
 	\begin{groupplot}[enlargelimits=false,scale only axis, group style={group size=1 by 1,x descriptions at=edge bottom,group name=mygroup},
 	width=0.4*\textwidth,
 	height=0.2*\textwidth,
 	every axis/.append style={font=\small,title style={anchor=base,yshift=-1mm}, x label style={yshift = 0.5em}, y label style={yshift = -.5em}, grid = both}, legend style={at={(1,0.45)},anchor=north east}
 	]

\nextgroupplot[ylabel={PSNR},xlabel={CPU Time (Seconds)},xmax = 5.5,ymax = 30.5,xtick={0,1,3,5.5},
xticklabels={0,1,3,5.5},tick align=inside,
tick pos=left,ytick={20,25,30},
yticklabels={20,25,30}]

     \addplot[solid,very thick,red,line width=1pt] table [search path={fig/Spiral_results/Brain},x=QNP_time, y=QNP_PSNR, col sep=comma] {Brain_Wav_QNP.csv};
     \addplot[dashed,very thick,blue,line width=1pt] table [search path={fig/Spiral_results/Brain},x=fista_time, y=fista_PSNR, col sep=comma] {Brain_Wav_fista.csv};  
    \addplot[dotted,very thick,black,line width=1pt] table [search path={fig/Spiral_results/Brain},x = PD_time, y=PD_PSNR, col sep=comma] {Brain_Wav_PD.csv}; 
   
\legend{CQNPM,APM,PD}; 
 \end{groupplot}
\end{tikzpicture}

\vspace{-2.5cm}

\begin{tikzpicture}
 \hspace{1.8cm} 
    
     \begin{axis}[at={(0,0)},anchor = north west, ylabel = APM,
    xmin = 0,xmax = 216,ymin = 0,ymax = 75, width=1\textwidth,
        scale only axis,
        enlargelimits=false,
        yshift=2.6cm,
      y label style = {yshift = -0.2cm,xshift=-1.5cm},
        axis line style={draw=none},
        tick style={draw=none},
        axis equal image,
        xticklabels={,,},yticklabels={,,},
        ]
    \node[inner sep=0pt, anchor = south west] (FISTA_1) at (0,0) {\includegraphics[ width=0.18\textwidth]{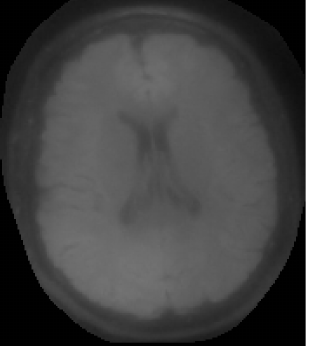}};
    
    \node at (8,41.5) {\color{white} $\text{iter.}=3$};
    \node at (9,3) {\color{red} $21.71$dB};
    

    \node[inner sep=0pt, anchor = west] (FISTA_2) at (FISTA_1.east) {\includegraphics[ width=0.18\textwidth]{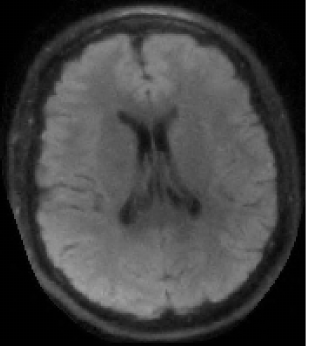}};
    
    \node at (43,41.5) {\color{white} $10$};
    \node at (48,3) {\color{red} $27.01$dB};
    
    \node[inner sep=0pt, anchor = west] (FISTA_3) at (FISTA_2.east) {\includegraphics[ width=0.18\textwidth]{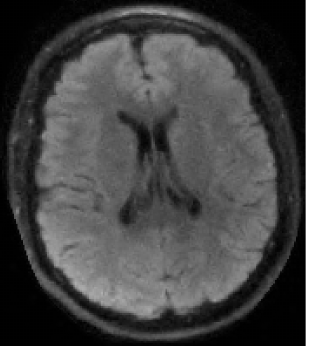}};
    
    \node at (81,41.5) {\color{white} $13$};
    \node at (87,3) {\color{red} $28.84$dB};

    \node[inner sep=0pt, anchor = west] (FISTA_4) at (FISTA_3.east) {\includegraphics[ width=0.18\textwidth]{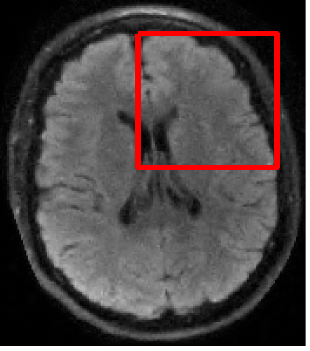}};
    \node at (120,41.5) {\color{white} $16$};
    \node at (125,3) {\color{red} $30.03$dB};
        \node at (150,37) {\Large\color{white} I};
    \end{axis}

    
    \begin{axis}[at={(FISTA_1.south west)},anchor = north west,ylabel = CQNPM,
    xmin = 0,xmax = 216,ymin = 0,ymax = 75, width=1\textwidth,
        scale only axis,
        enlargelimits=false,
        yshift = 2.6cm,
        y label style = {yshift = -0.2cm,xshift=-1.3cm},
       axis line style={draw=none},
       tick style={draw=none},
        axis equal image,
        xticklabels={,,},yticklabels={,,}
       ]
   \node[inner sep=0pt, anchor = south west] (QN_1) at (0,0) {\includegraphics[ width=0.18\textwidth]{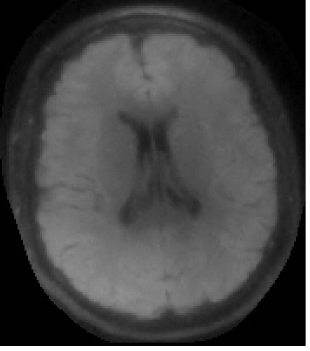}};
    \node at (9,3) {\color{red} $23.82$dB};
    
    \node[inner sep=0pt, anchor = west] (QN_2) at (QN_1.east) {\includegraphics[ width=0.18\textwidth]{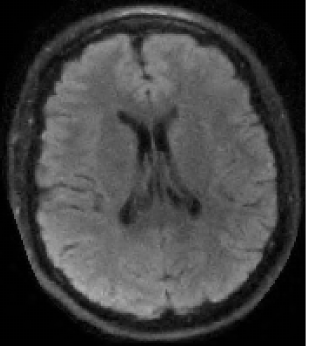}};
    \node at (48,3) {\color{red} $29.34$dB};
    
    \node[inner sep=0pt, anchor = west] (QN_3) at (QN_2.east) {\includegraphics[ width=0.18\textwidth]{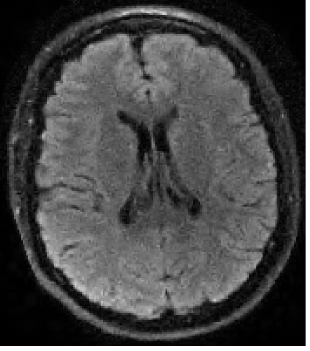}};
    \node at (87,3) {\color{red} $30.28$dB};

    \node[inner sep=0pt, anchor = west] (QN_4) at (QN_3.east) {\includegraphics[ width=0.18\textwidth]{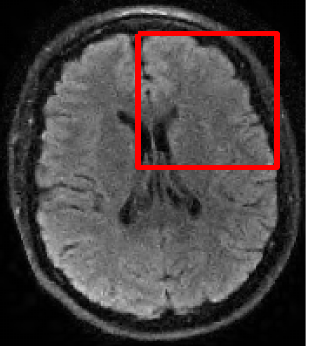}};
     \node at (125,3) {\color{red} $30.19$dB};   
         \node at (150,37) {\Large\color{white} II}; 
 \end{axis}
 
   \begin{axis}[at={(QN_1.south west)},anchor = north west,
    xmin = 0,xmax = 216,ymin = 0,ymax = 75, width=1\textwidth,
        scale only axis,
        enlargelimits=false,
        yshift = 3.65cm,
       axis line style={draw=none},
       tick style={draw=none},
        axis equal image,
        xticklabels={,,},yticklabels={,,}
       ]
       
   \node[inner sep=0pt,anchor = south west] (Resi_FISTA) at (0,0) {\includegraphics[ width=0.15\textwidth]{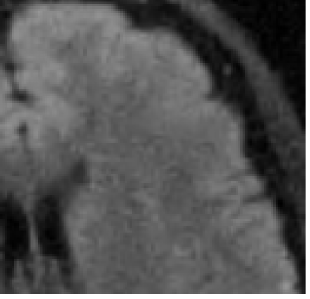}};    
   
    \node[inner sep=8pt, anchor = west] (Resi_FISTA_1) at (Resi_FISTA.east) {\includegraphics[ width=0.15\textwidth]{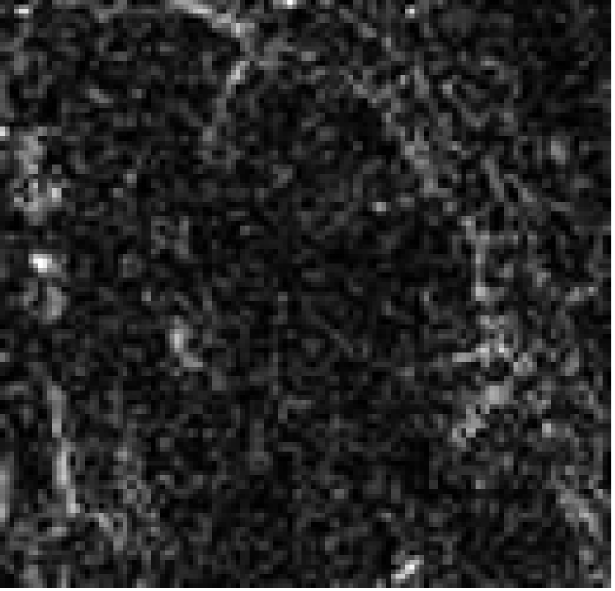}};

    \node[inner sep=15pt, anchor = west] (Resi_FISTA_2) at (Resi_FISTA_1.east) {\includegraphics[ width=0.15\textwidth]{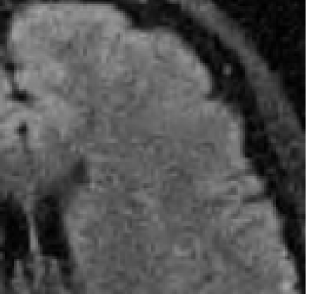}};

    \node[inner sep=5pt, anchor = west] (Resi_FISTA_3) at (Resi_FISTA_2.east) {\includegraphics[ width=0.15\textwidth]{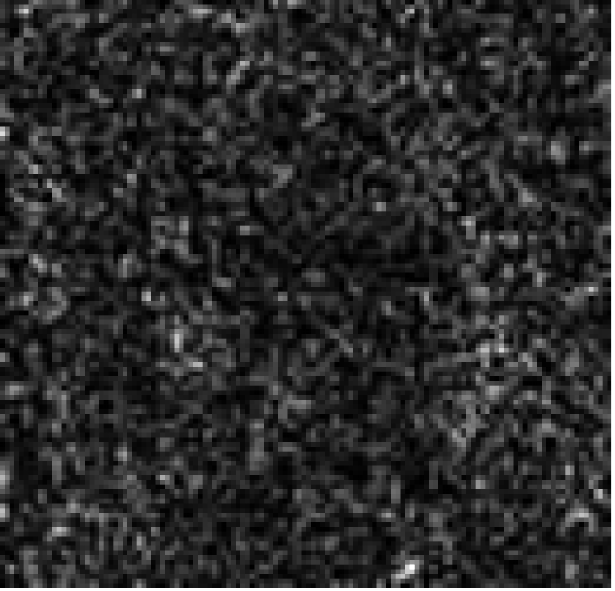}};  
    
    \node at (30,28) {\Large\color{white} I};
    \node at (65,28) {\Large\color{white} II};
    \node at (108,28) {\Large\color{white} I};
    \node at (146,28) {\Large\color{white} II};
\node at (114.5,17) {\Large\color{red} $\times 5$};
\node at (156,17) {\Large\color{red} $\times 5$};

 \end{axis}

\end{tikzpicture} 

\caption{
{ First row: the ground truth image and PSNR values versus CPU time; second to third row: the reconstructed brain images at $3$, $10$, $13$, and $16$th iteration by APM and CQNPM methods with
\Cref{fig:brain:spiral:cost} setting; 
 fourth row: the zoomed-in regions
and the corresponding error maps ($\times 5$)
of the $16$th iteration reconstructed images.}
}
\label{fig:SpiralBrain:wavelet}
\end{figure*}

%% file: Figs/KneeWaveFigSpiral.tex

\begin{figure}[!t]
   \centering
	\begin{tikzpicture}

 \draw [draw=blue] (2.6,-0.1) rectangle (3.62,0.2);

 \draw [-,color=blue](3,0.2) -- (3.93,2.15);
\draw [-,color=blue](3,-0.1) -- (3.93,0);

 \draw [draw=blue] (2.7,-3.3) rectangle (3.62,-3);

 \draw [-,color=blue](3.2,-3) -- (3.93,-1);
  \draw [-,color=blue](3.2,-3.3) -- (3.93,-3.2);
 
 

 	\pgfplotsset{every axis legend/.append style={legend pos=south east,anchor=south east,font=\normalsize, legend cell align={left}}}
    \pgfplotsset{grid style={dotted, gray}}
 	\begin{groupplot}[enlargelimits=false,scale only axis, group style={group size=2 by 2,x descriptions at=edge bottom,group name=mygroup},
 	width=0.2*\textwidth,
 	height=0.12*\textwidth,
 	every axis/.append style={font=\small,title style={anchor=base,yshift=-1mm}, x label style={yshift = 0.5em}, y label style={yshift = -.5em}, grid = both}, legend style={at={(0.8,1)},anchor=north east,font=\footnotesize}
 	]
  
\nextgroupplot[ylabel={Cost},xlabel={Iteration},xmax = 16,ymax = 125,xtick={0,4,8,12,16},
xticklabels={0,4,8,12,16},tick align=inside,ytick={85,105,125},yticklabels={85,105,125},
tick pos=left]

 \addplot[dotted,very thick,black,line width=1pt] table [search path={fig/Spiral_results/Knee},x expr=\coordindex, y=PD_cost, col sep=comma] {Knee_Wav_PD.csv}; 

 \addplot[dashed,very thick,blue,line width=1pt] table [search path={fig/Spiral_results/Knee},x expr=\coordindex, y=fista_cost, col sep=comma] {Knee_Wav_fista.csv};        
 
 \addplot[solid,very thick,red,line width=1pt] table [search path={fig/Spiral_results/Knee},x expr=\coordindex, y=QNP_cost, col sep=comma] {Knee_Wav_QNP.csv};
 
    \legend{PD,APM,CQNPM};  

\nextgroupplot[ylabel={},xlabel={Iteration},xmin=12,xmax = 16,ymin= 65.5,ymax = 66.8,xtick={12,14,16},
xticklabels={,14,16},tick align=inside,ytick={},yticklabels={},
tick pos=left,xshift={-7mm}]

   \addplot[dotted,very thick,black,line width=1pt] table [search path={fig/Spiral_results/Knee},x expr=\coordindex, y=PD_cost, col sep=comma] {Knee_Wav_PD.csv}; 
   
     \addplot[dashed,very thick,blue,line width=1pt] table [search path={fig/Spiral_results/Knee},x expr=\coordindex, y=fista_cost, col sep=comma] {Knee_Wav_fista.csv};   
     \addplot[solid,very thick,red,line width=1pt] table [search path={fig/Spiral_results/Knee},x expr=\coordindex, y=QNP_cost, col sep=comma] {Knee_Wav_QNP.csv};
     
\nextgroupplot[xlabel={CPU Time (Seconds)},xmax = 5,ymax = 125,xtick={0,1,3,5},
xticklabels={0,1,3,5},xtick align=inside,
tick pos=left,ytick={85,105,125},yticklabels={85,105,125},ylabel={Cost}]

   \addplot[dotted,very thick,black,line width=1pt] table [search path={fig/Spiral_results/Knee},x = PD_time, y=PD_cost, col sep=comma] {Knee_Wav_PD.csv}; 
   
\addplot[dashed,very thick,blue,line width=1pt] table [search path={fig/Spiral_results/Knee},x = fista_time, y=fista_cost, col sep=comma] {Knee_Wav_fista.csv};
  
\addplot[solid,very thick,red,line width=1pt] table [search path={fig/Spiral_results/Knee},x = QNP_time, y=QNP_cost, col sep=comma] {Knee_Wav_QNP.csv};

\nextgroupplot[xlabel={CPU Time (Seconds)},xmax = 5,xmin = 3.5,xtick={3.5,4,5},
xticklabels={,4,5},xtick align=inside,
tick pos=left,ytick={},yticklabels={},ymin=65.5,ymax = 67]

   \addplot[dotted,very thick,black,line width=1pt] table [search path={fig/Spiral_results/Knee},x = PD_time, y=PD_cost, col sep=comma] {Knee_Wav_PD.csv}; 
\addplot[dashed,very thick,blue,line width=1pt] table [search path={fig/Spiral_results/Knee},x = fista_time, y=fista_cost, col sep=comma] {Knee_Wav_fista.csv};

\addplot[solid,very thick,red,line width=1pt] table [search path={fig/Spiral_results/Knee},x = QNP_time, y=QNP_cost, col sep=comma] {Knee_Wav_QNP.csv};
\end{groupplot}
\end{tikzpicture} 
\caption{\cb Cost values versus iteration (top) and CPU time (bottom) of the knee image with same regularizer
and acquisition as \Cref{fig:brain:spiral:cost}. The parameter $\lambda=2\times 10^{-3}$.}
 \label{fig:knee:spiral:cost}
\end{figure}
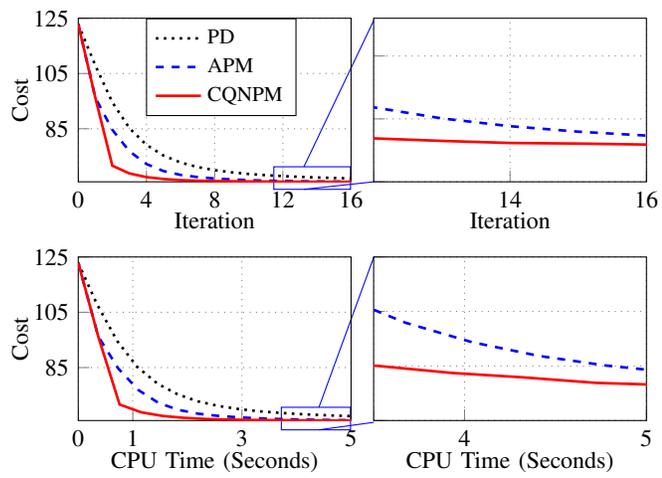

\begin{figure*}[!t]
	\centering
\begin{tikzpicture}

   \node (Knee_GT) at (-5,1.3) {\includegraphics[ width=0.22\textwidth]{fig/CropSquare_Knee_GT.pdf}};
   \node at (-6.35,2.9) {\color{white} GT};

    \node (Knee_GT_Crop) at (-1.9,1.3) {\includegraphics[ width=0.11\textwidth]{fig/Crop_Knee_GT.pdf}};

  \node at (-4.4,0.5) {\Large \color{white} (a)};
  \node at (-1.2,0.7) {\Large \color{white} (a)};
    
 	\pgfplotsset{every axis legend/.append style={legend pos=south east,anchor=south east,font=\normalsize, legend cell align={left}}}
    \pgfplotsset{grid style={dotted, gray}}

    \hspace{0.35cm}
 	
  \begin{groupplot}[enlargelimits=false,scale only axis, group style={group size=1 by 1,x descriptions at=edge bottom,group name=mygroup},
 	width=0.4*\textwidth,
 	height=0.2*\textwidth,
 	every axis/.append style={font=\small,title style={anchor=base,yshift=-1mm}, x label style={yshift = 0.5em}, y label style={yshift = -.5em}, grid = both}, legend style={at={(1,0.4)},anchor=north east}
 	]
  
\nextgroupplot[ylabel={PSNR},xlabel={CPU Time (Seconds)},xmax = 5,ymax = 32.6,xtick={0,1,3,5},
xticklabels={0,1,3,5},tick align=inside,
tick pos=left,ytick={20,25,30},
yticklabels={20,25,30}]
 
     \addplot[solid,very thick,red,line width=1pt] table [search path={fig/Spiral_results/Knee},x=QNP_time, y=QNP_PSNR, col sep=comma] {Knee_Wav_QNP.csv};
     \addplot[dashed,very thick,blue,line width=1pt] table [search path={fig/Spiral_results/Knee},x=fista_time, y=fista_PSNR, col sep=comma] {Knee_Wav_fista.csv};  
 \addplot[dotted,very thick,black,line width=1pt] table [search path={fig/Spiral_results/Knee},x = PD_time, y = PD_PSNR, col sep=comma] {Knee_Wav_PD.csv};      
\legend{CQNPM,APM,PD}; 
 \end{groupplot}
\end{tikzpicture}

\vspace{-4cm}

\begin{tikzpicture}
    \hspace{1cm}
 
\begin{axis}[at={(0,0)},anchor = north west, ylabel = APM,
    xmin = 0,xmax = 216,ymin = 0,ymax = 75, width=1\textwidth,
        scale only axis,
        enlargelimits=false,
        yshift = 4.2cm,
      y label style = {yshift = -0.3cm,xshift=-2cm},
        axis line style={draw=none},
        tick style={draw=none},
        axis equal image,
        xticklabels={,,},yticklabels={,,},
        ]
    \node[inner sep=0pt, anchor = south west] (FISTA_1) at (0,0) {\includegraphics[ width=0.2\textwidth]{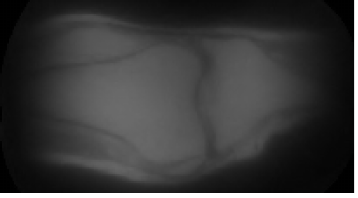}};
    
    \node at (8,22) {\color{white} $\text{iter.}=3$};
    \node at (8,3) {\color{red} $22.14$dB};
    

    \node[inner sep=0pt, anchor = west] (FISTA_2) at (FISTA_1.east) {\includegraphics[ width=0.2\textwidth]{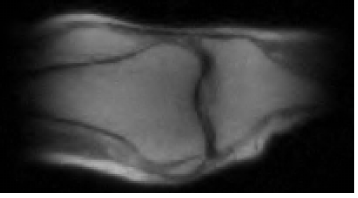}};
    
    \node at (46,22) {\color{white} $10$};
    \node at (51,3) {\color{red} $27.16$dB};
    
    \node[inner sep=0pt, anchor = west] (FISTA_3) at (FISTA_2.east) {\includegraphics[ width=0.2\textwidth]{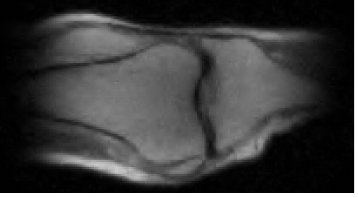}};
    
    \node at (90,22) {\color{white} $13$};
    \node at (95,3) {\color{red} $28.29$dB};

    \node[inner sep=0pt, anchor = west] (FISTA_4) at (FISTA_3.east) {\includegraphics[ width=0.2\textwidth]{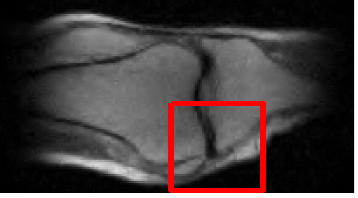}};
    \node at (133,22) {\color{white} $16$};
    \node at (138,3) {\color{red} $29.18$dB};
    \node at (160,4) {\Large \color{white} I};

    \end{axis}

    
    \begin{axis}[at={(FISTA_1.south west)},anchor = north west,ylabel = CQNPM,
    xmin = 0,xmax = 216,ymin = 0,ymax = 75, width=1\textwidth,
        scale only axis,
        enlargelimits=false,
        yshift = 4.3cm,
        y label style = {yshift = -0.3cm,xshift=-2cm},
       axis line style={draw=none},
       tick style={draw=none},
        axis equal image,
        xticklabels={,,},yticklabels={,,}, 
       ]
   \node[inner sep=0pt, anchor = south west] (QN_1) at (0,0) {\includegraphics[ width=0.2\textwidth]{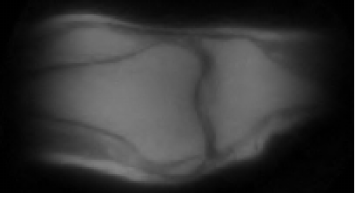}};
    \node at (8,3) {\color{red} $24.62$dB};
    
    \node[inner sep=0pt, anchor = west] (QN_2) at (QN_1.east) {\includegraphics[ width=0.2\textwidth]{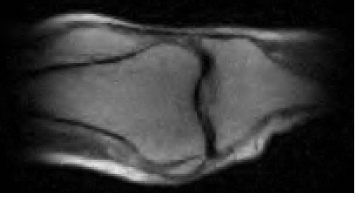}};
    \node at (51,3) {\color{red} $29.35$dB};
    
    \node[inner sep=0pt, anchor = west] (QN_3) at (QN_2.east) {\includegraphics[ width=0.2\textwidth]{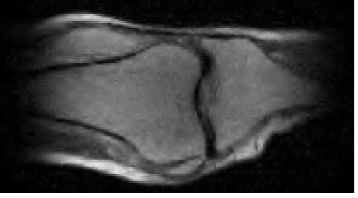}};
    \node at (95,3) {\color{red} $30.3$dB};

    \node[inner sep=0pt, anchor = west] (QN_4) at (QN_3.east) {\includegraphics[ width=0.2\textwidth]{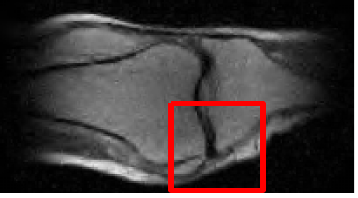}};
     \node at (138,3) {\color{red} $30.44$dB};   
     \node at (160,4) {\Large \color{white} II}; 
 \end{axis}
 
  \begin{axis}[at={(QN_1.south west)},anchor = north west,
    xmin = 0,xmax = 216,ymin = 0,ymax = 75, width=1\textwidth,
        scale only axis,
        enlargelimits=false,
          yshift = 4.3cm,
       axis line style={draw=none},
       tick style={draw=none},
        axis equal image,
        xticklabels={,,},yticklabels={,,}
       ]
       
   \node[inner sep=0pt,anchor = south west] (Resi_FISTA) at (0,0) {\includegraphics[ width=0.11\textwidth]{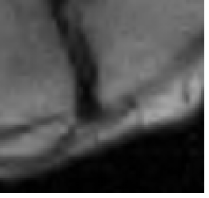}};    
   
    \node[inner sep=18pt, anchor = west] (Resi_FISTA_1) at (Resi_FISTA.east) {\includegraphics[ width=0.11\textwidth]{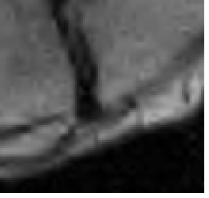}};

\node[inner sep=18pt, anchor = west] (Resi_FISTA_2) at (Resi_FISTA_1.east) {\includegraphics[ width=0.11\textwidth]{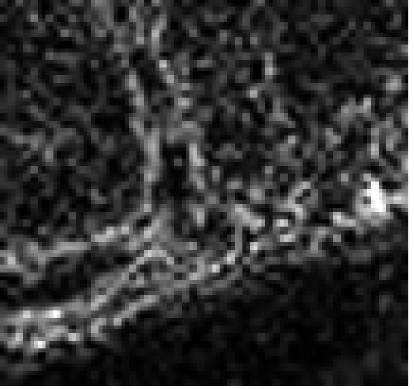}};

    \node[inner sep=0pt, anchor = west] (Resi_FISTA_3) at (Resi_FISTA_2.east) {\includegraphics[ width=0.11\textwidth]{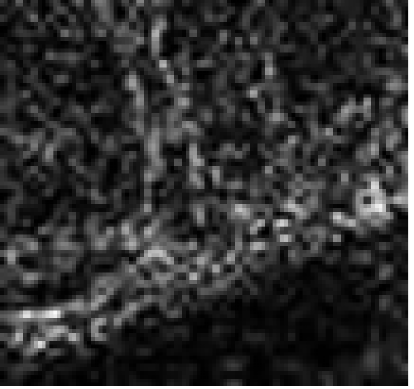}};

       \node at (22,3.5) {\Large \color{white} I};
    \node at (53,3.5) {\Large \color{white} II};
    \node at (93,3.5) {\Large \color{white} I};
      \node at (123,3.5) {\Large \color{white} II};
    \node at (132,10) {\Large \color{red}$\times 5$};
   
 \end{axis}

\end{tikzpicture} 

\caption{
 First row: the groundth image and PSNR values versus CPU time; second to third row: the reconstructed knee images at $3$, $10$, $13$, and $16$th iteration with \Cref{fig:knee:spiral:cost} setting; fouth row: the zoomed-in regions and the corresponding error maps ($\times 5$) of the $16$th iteration reconstructed images.
}
	\label{fig:SpiralKnee:wavelet}
\end{figure*}

%% file: Figs/BrainWaveTVFigSpiral.tex

\begin{figure}[!t]
   \centering
	\begin{tikzpicture}

 \draw [draw=blue] (2.7,-0.1) rectangle (3.62,0.08);

 \draw [-,color=blue](3,0.2) -- (3.93,2.15);
\draw [-,color=blue](3,-0.1) -- (3.93,0);

\draw [draw=blue] (2.9,-3.3) rectangle (3.62,-3);

\draw [-,color=blue](3.2,-3) -- (3.93,-1);
\draw [-,color=blue](3.2,-3.3) -- (3.93,-3.2);
 
 

 	\pgfplotsset{every axis legend/.append style={legend pos=south east,anchor=south east,font=\normalsize, legend cell align={left}}}
    \pgfplotsset{grid style={dotted, gray}}
 	\begin{groupplot}[enlargelimits=false,scale only axis, group style={group size=2 by 2,x descriptions at=edge bottom,group name=mygroup},
 	width=0.2*\textwidth,
 	height=0.12*\textwidth,
 	every axis/.append style={font=\small,title style={anchor=base,yshift=-1mm}, x label style={yshift = 0.5em}, y label style={yshift = -.5em}, grid = both}, legend style={at={(0.75,1)},anchor=north east,font=\tiny,text opacity = 1,fill opacity=0.6}
 	]
  
\nextgroupplot[ylabel={Cost},xlabel={Iteration},xmax = 16,ymax = 105,xtick={0,4,8,12,16},
xticklabels={0,4,8,12,16},tick align=inside,ytick={60,100,145},
yticklabels={60,100,145},
tick pos=left,mark repeat = 5,mark size = 3pt]

\addplot[dotted,very thick,black,line width=1pt] table [search path={fig/Spiral_results/Brain},x expr=\coordindex, y=PD_cost, col sep=comma] {Brain_WavTV_PD.csv};    

\addplot[dash pattern=on 2pt off 3pt on 4pt off 4pt,very thick,goldenbrown,line width=1pt] table [search path={fig/Spiral_results/Brain},x expr=\coordindex, y=ADMM_cost, col sep=comma] {Brain_WavTV_ADMM.csv};

\addplot[dashed,very thick,blue,line width=1pt] table [search path={fig/Spiral_results/Brain},x expr=\coordindex, y=fista_cost, col sep=comma] {Brain_WavTV_fista.csv};  

\addplot[dashed,mark=star,very thick,darkblue,line width=1pt] table [search path={fig/Spiral_results/Brain},x expr=\coordindex, y=fista_Smooth_cost, col sep=comma] {Brain_WavTV_fistaSmooth.csv}; 

\addplot[solid,very thick,red,line width=1pt] table [search path={fig/Spiral_results/Brain},x expr=\coordindex, y=QNP_cost, col sep=comma] {Brain_WavTV_QNP.csv};

\addplot[solid,mark=star,very thick,red,line width=1pt] table [search path={fig/Spiral_results/Brain},x expr=\coordindex, y=QNP_Smooth_cost, col sep=comma] {Brain_WavTV_QNPSmooth.csv};

\legend{PD,ADMM,APM,S-APM,CQNPM,S-CQNPM};

\nextgroupplot[ylabel={},xlabel={Iteration},xmin=12,xmax = 16,ymin= 65.5,ymax = 66.6,xtick={12,14,16},
xticklabels={,14,16},tick align=inside,ytick={},yticklabels={},
tick pos=left,xshift={-7mm}]

  \addplot[dotted,very thick,black,line width=1pt] table [search path={fig/Spiral_results/Brain},x expr=\coordindex, y=PD_cost, col sep=comma] {Brain_WavTV_PD.csv};
  
 \addplot[dash pattern=on 2pt off 3pt on 4pt off 4pt,very thick,goldenbrown,line width=1pt] table [search path={fig/Spiral_results/Brain},x expr=\coordindex, y=ADMM_cost, col sep=comma] {Brain_WavTV_ADMM.csv};

\addplot[dashed,very thick,blue,line width=1pt] table [search path={fig/Spiral_results/Brain},x expr=\coordindex, y=fista_cost, col sep=comma] {Brain_WavTV_fista.csv};   
\addplot[dashed,mark=star,very thick,darkblue,line width=1pt] table [search path={fig/Spiral_results/Brain},x expr=\coordindex, y=fista_Smooth_cost, col sep=comma] {Brain_WavTV_fistaSmooth.csv}; 

\addplot[solid,very thick,red,line width=1pt] table [search path={fig/Spiral_results/Brain},x expr=\coordindex, y=QNP_cost, col sep=comma] {Brain_WavTV_QNP.csv};

\addplot[solid,mark=star,very thick,red,line width=1pt] table [search path={fig/Spiral_results/Brain},x expr=\coordindex, y=QNP_Smooth_cost, col sep=comma] {Brain_WavTV_QNPSmooth.csv};

\nextgroupplot[xlabel={CPU Time (Seconds)},xmax = 12,ymax = 105,xtick={0,4,8,12},
xticklabels={0,4,8,12},xtick align=inside,
ylabel={Cost},ytick={60,100,145},
yticklabels={60,100,145},mark repeat = 5,mark size = 3pt,ylabel={Cost}]

  \addplot[dotted,very thick,black,line width=1pt] table [search path={fig/Spiral_results/Brain},x =PD_time, y=PD_cost, col sep=comma] {Brain_WavTV_PD.csv};
  
 \addplot[dash pattern=on 2pt off 3pt on 4pt off 4pt,very thick,goldenbrown,line width=1pt] table [search path={fig/Spiral_results/Brain},x = ADMM_time, y=ADMM_cost, col sep=comma] {Brain_WavTV_ADMM.csv};

\addplot[dashed,very thick,blue,line width=1pt] table [search path={fig/Spiral_results/Brain},x = fista_time, y=fista_cost, col sep=comma] {Brain_WavTV_fista.csv};

\addplot[dashed,mark=star,very thick,darkblue,line width=1pt] table [search path={fig/Spiral_results/Brain},x =fista_Smooth_time, y=fista_Smooth_cost, col sep=comma] {Brain_WavTV_fistaSmooth.csv}; 

\addplot[solid,very thick,red,line width=1pt] table [search path={fig/Spiral_results/Brain},x = QNP_time, y=QNP_cost, col sep=comma] {Brain_WavTV_QNP.csv};

\addplot[solid,mark=star,very thick,red,line width=1pt] table [search path={fig/Spiral_results/Brain},x = QNP_Smooth_time, y=QNP_Smooth_cost, col sep=comma] {Brain_WavTV_QNPSmooth.csv};

\nextgroupplot[xlabel={CPU Time (Seconds)},xmin=10,xmax = 12,xtick={10,11,12},
xticklabels={,11,12},xtick align=inside,
tick pos=left,ytick={},yticklabels={},ymin=65,ymax = 66.35]

  \addplot[dotted,very thick,black,line width=1pt] table [search path={fig/Spiral_results/Brain},x =PD_time, y=PD_cost, col sep=comma] {Brain_WavTV_PD.csv};
  
 \addplot[dash pattern=on 2pt off 3pt on 4pt off 4pt,very thick,goldenbrown,line width=1pt] table [search path={fig/Spiral_results/Brain},x = ADMM_time, y=ADMM_cost, col sep=comma] {Brain_WavTV_ADMM.csv};
 
\addplot[dashed,very thick,blue,line width=1pt] table [search path={fig/Spiral_results/Knee},x = fista_time, y=fista_cost, col sep=comma] {Knee_WavTV_fista.csv};

\addplot[dashed,mark=star,very thick,darkblue,line width=1pt] table [search path={fig/Spiral_results/Knee},x =fista_Smooth_time, y=fista_Smooth_cost, col sep=comma] {Knee_WavTV_fistaSmooth.csv}; 

\addplot[solid,very thick,red,line width=1pt] table [search path={fig/Spiral_results/Knee},x = QNP_time, y=QNP_cost, col sep=comma] {Knee_WavTV_QNP.csv};

\addplot[solid,mark=star,very thick,red,line width=1pt] table [search path={fig/Spiral_results/Knee},x = QNP_Smooth_time, y=QNP_Smooth_cost, col sep=comma] {Knee_WavTV_QNPSmooth.csv};

\end{groupplot}
\end{tikzpicture} 
\caption{\cb Cost values versus iteration (top) and CPU time (bottom) of the knee image with regularizer
$h(\vx)=\alpha \|\mT\vx\|_1+(1-\alpha)\mrm{TV}(\vx)$ and same acquisition as \Cref{fig:brain:spiral:cost}.
The parameters $\lambda$ and $\alpha$ were $10^{-3}$ and $\frac{1}{2}$.}
 \label{fig:brain:spiral:WavTV:cost}
\end{figure}
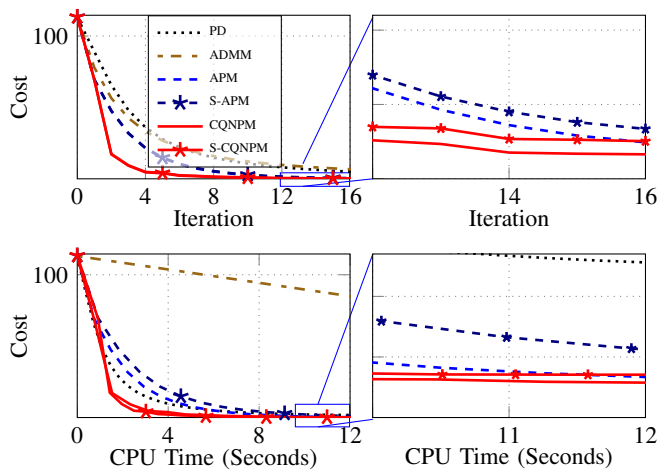

\begin{figure*}[!htb]
	\centering

\begin{tikzpicture}
    \node (Brain_GT) at (-5,1.3) {\includegraphics[ width=0.18\textwidth]{fig/CropSquare_Brain_GT.pdf}};
   \node at (-6.35,2.9) {\color{white} GT};

    \node (Brain_GT_Crop) at (-1.9,1.3) {\includegraphics[ width=0.15\textwidth]{fig/Crop_Brain_GT.pdf}};
    \node at (-3.9,2.63) {\color{white} (a)};

   \node at (-0.8,2.35) {\color{white} (a)};
   
 	\pgfplotsset{every axis legend/.append style={legend pos=south east,anchor=south east,font=\normalsize, legend cell align={left}}}
    \pgfplotsset{grid style={dotted, gray}}
    \hspace{0.45cm}
 	\begin{groupplot}[enlargelimits=false,scale only axis, group style={group size=1 by 1,x descriptions at=edge bottom,group name=mygroup},
 	width=0.5*\textwidth,
 	height=0.2*\textwidth,
 	every axis/.append style={font=\small,title style={anchor=base,yshift=-1mm}, x label style={yshift = 0.5em}, y label style={yshift = -.5em}, grid = both}, legend style={at={(1,0.63)},anchor=north east,font =\tiny}
 	]
  
\nextgroupplot[ylabel={PSNR},xlabel={CPU Time (Seconds)},xmax = 12,ymax = 32,xtick={0,10,25,40},
xticklabels={0,10,25,40},tick align=inside,
tick pos=left,mark repeat = 5,mark size = 3pt]
    
     \addplot[solid,mark=star,very thick,red,line width=1pt] table [search path={fig/Spiral_results/Brain},x =QNP_Smooth_time, y=QNP_Smooth_PSNR, col sep=comma] {Brain_WavTV_QNPSmooth.csv};
    
    \addplot[solid,very thick,red,line width=1pt] table [search path={fig/Spiral_results/Brain},x =QNP_time, y=QNP_PSNR, col sep=comma] {Brain_WavTV_QNP.csv};

\addplot[dashed,very thick,blue,line width=1pt] table [search path={fig/Spiral_results/Brain},x =fista_time, y=fista_PSNR, col sep=comma] {Brain_WavTV_fista.csv}; 
        
    \addplot[dashed,mark=star,very thick,darkblue,line width=1pt] table [search path={fig/Spiral_results/Brain},x =fista_Smooth_time, y=fista_Smooth_PSNR, col sep=comma] {Brain_WavTV_fistaSmooth.csv}; 
   
  \addplot[dotted,very thick,black,line width=1pt] table [search path={fig/Spiral_results/Brain},x =PD_time, y=PD_PSNR, col sep=comma] {Brain_WavTV_PD.csv};
  
 \addplot[dash pattern=on 2pt off 3pt on 4pt off 4pt,very thick,goldenbrown,line width=1pt] table [search path={fig/Spiral_results/Brain},x = ADMM_time, y=ADMM_PSNR, col sep=comma] {Brain_WavTV_ADMM.csv};
 
  \legend{S-CQNPM,CQNPM,APM,S-APM,PD,ADMM};

 \end{groupplot}
\end{tikzpicture}

\vspace{-2cm}

\begin{tikzpicture}

 \hspace{0.8cm} 
 
     \begin{axis}[at={(0,0)},anchor = north west, ylabel = PD,
    xmin = 0,xmax = 216,ymin = 0,ymax = 75, width=1\textwidth,
        scale only axis,
        enlargelimits=false,
      y label style = {yshift = -0.2cm,xshift=-1.5cm},
        axis line style={draw=none},
        tick style={draw=none},
        axis equal image,
        xticklabels={,,},yticklabels={,,},
        ]
    \node[inner sep=0pt, anchor = south west] (PD_1) at (0,0) {\includegraphics[ width=0.18\textwidth]{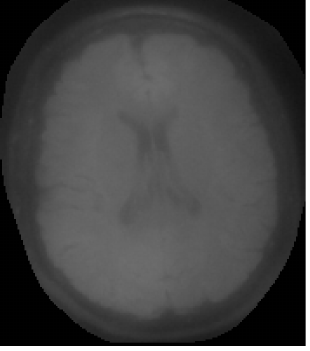}};
    
    \node at (8,41.5) {\color{white} $\text{iter.}=3$};
    \node at (9,3) {\color{red} $21.11$dB};

    \node[inner sep=0pt, anchor = west] (PD_2) at (PD_1.east) {\includegraphics[ width=0.18\textwidth]{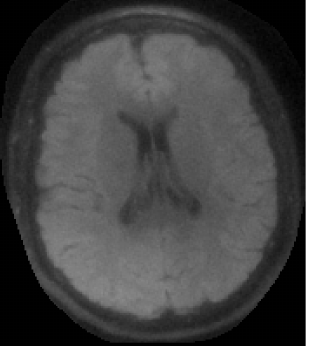}};
    
    \node at (42,41.5) {\color{white} $10$};
    \node at (47,3) {\color{red} $24.56$dB};
    
    \node[inner sep=0pt, anchor = west] (PD_3) at (PD_2.east) {\includegraphics[ width=0.18\textwidth]{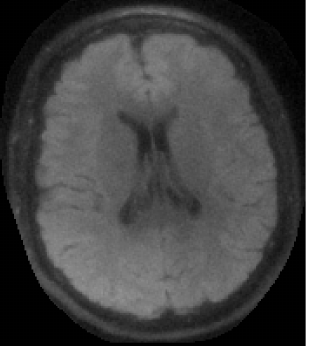}};
    
    \node at (82,41.5) {\color{white} $13$};
    \node at (88,3) {\color{red} $25.54$dB};

    \node[inner sep=0pt, anchor = west] (PD_4) at (PD_3.east) {\includegraphics[ width=0.18\textwidth]{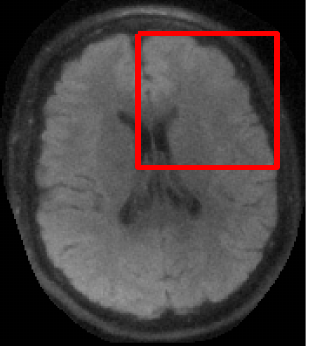}};
    \node at (120,41.5) {\color{white} $16$};
    \node at (125,3) {\color{red} $26.35$dB};
    \node[inner sep=-0pt, anchor = west] (PD_5_crop) at (PD_4.east) {\includegraphics[ width=0.15\textwidth]{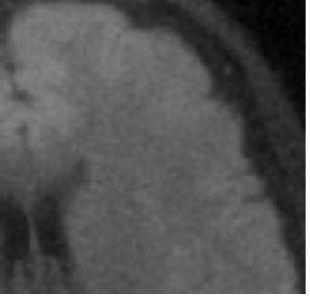}};
  \node at (150,37) {\Large\color{white} I};
   \node at (186,34) {\Large\color{white} I};
    \end{axis}

     \begin{axis}[at={(PD_1.south west)},anchor = north west, ylabel = APM,
    xmin = 0,xmax = 216,ymin = 0,ymax = 75, width=1\textwidth,
        scale only axis,
        enlargelimits=false,
        yshift=2.6cm,
      y label style = {yshift = -0.2cm,xshift=-1.3cm},
        axis line style={draw=none},
        tick style={draw=none},
        axis equal image,
        xticklabels={,,},yticklabels={,,},
        ]
    \node[inner sep=0pt, anchor = south west] (FISTA_1) at (0,0) {\includegraphics[ width=0.18\textwidth]{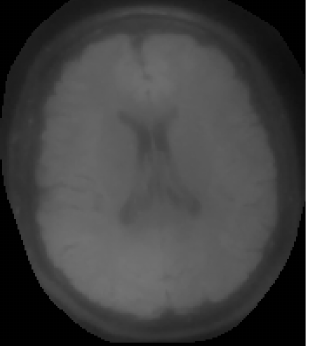}};
    
    \node at (9,3) {\color{red} $21.7$dB};

    \node[inner sep=0pt, anchor = west] (FISTA_2) at (FISTA_1.east) {\includegraphics[ width=0.18\textwidth]{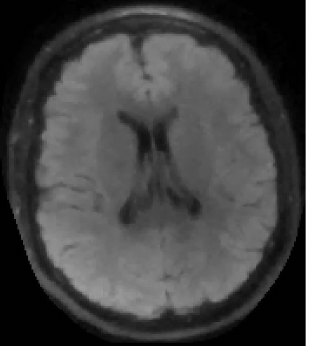}};

    \node at (48,3) {\color{red} $26.93$dB};
    
    \node[inner sep=0pt, anchor = west] (FISTA_3) at (FISTA_2.east) {\includegraphics[ width=0.18\textwidth]{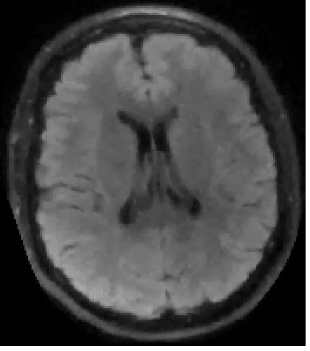}};
    
    \node at (88,3) {\color{red} $28.75$dB};

    \node[inner sep=0pt, anchor = west] (FISTA_4) at (FISTA_3.east) {\includegraphics[ width=0.18\textwidth]{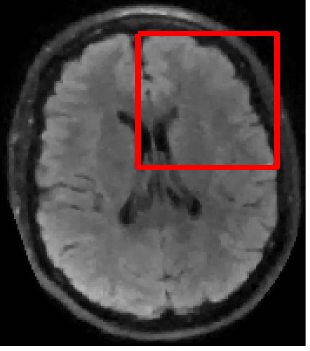}};
    \node at (125,3) {\color{red} $30.04$dB};

    \node[inner sep=0pt, anchor = west] (CropFISTA_5) at (FISTA_4.east) {\includegraphics[ width=0.15\textwidth]{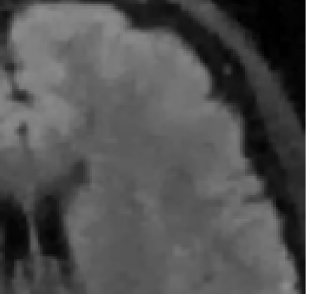}};
      \node at (150,37) {\Large\color{white} II};
   \node at (186,34) {\Large\color{white} II};
    \end{axis}

\begin{axis}[at={(FISTA_1.south west)},anchor =  north west, ylabel = S-APM,
    xmin = 0,xmax = 216,ymin = 0,ymax = 75, width=1\textwidth,
        scale only axis,
        enlargelimits=false,
        yshift = 2.6cm,
        y label style = {yshift = -0.2cm,xshift=-1.3cm},
        axis line style={draw=none},
        tick style={draw=none},
        axis equal image,
        xticklabels={,,},yticklabels={,,},
        ]
    \node[inner sep=0pt, anchor = south west] (S_FISTA_1) at (0,0) {\includegraphics[ width=0.18\textwidth]{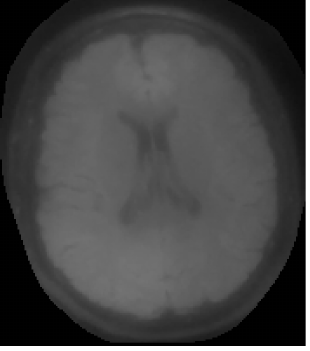}}; 
    \node at (9,3) {\color{red} $21.7$dB};
    \node[inner sep=0pt, anchor = west] (S_FISTA_2) at (S_FISTA_1.east) {\includegraphics[ width=0.18\textwidth]{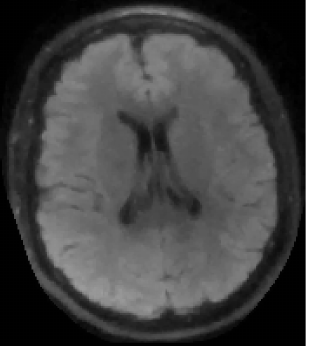}};

    \node at (49,3) {\color{red} $26.94$dB};

    \node[inner sep=0pt, anchor = west] (S_FISTA_3) at (S_FISTA_2.east) {\includegraphics[ width=0.18\textwidth]{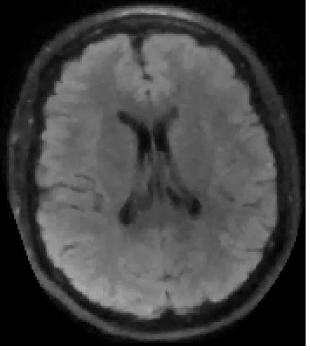}};
     \node at (87,3) {\color{red} $28.76$dB};
    \node[inner sep=0pt, anchor = west] (S_FISTA_4) at (S_FISTA_3.east) {\includegraphics[width=0.18\textwidth]{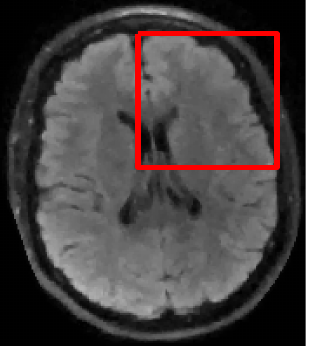}};
    \node at (125,3) {\color{red} $30.05$dB};

 \node[inner sep=0pt, anchor = west] (CropS_FISTA_4) at (S_FISTA_4.east) {\includegraphics[width=0.15\textwidth]{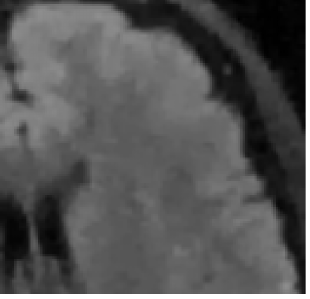}};
      \node at (149,37) {\Large\color{white} III};
   \node at (185,34) {\Large\color{white} III};
    \end{axis}

\begin{axis}[at={(S_FISTA_1.south west)},anchor = north west,ylabel = CQNPM,
    xmin = 0,xmax = 216,ymin = 0,ymax = 75, width=1\textwidth,
        scale only axis,
        enlargelimits=false,
        yshift = 2.6cm,
        y label style = {yshift = -0.2cm,xshift=-1.3cm},
       axis line style={draw=none},
       tick style={draw=none},
        axis equal image,
        xticklabels={,,},yticklabels={,,}
       ]
   \node[inner sep=0pt, anchor = south west] (QN_1) at (0,0) {\includegraphics[ width=0.18\textwidth]{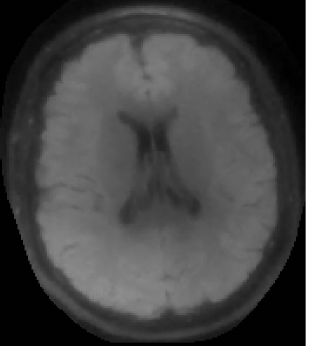}};
    \node at (9,3) {\color{red} $23.8$dB};
    
    \node[inner sep=0pt, anchor = west] (QN_2) at (QN_1.east) {\includegraphics[ width=0.18\textwidth]{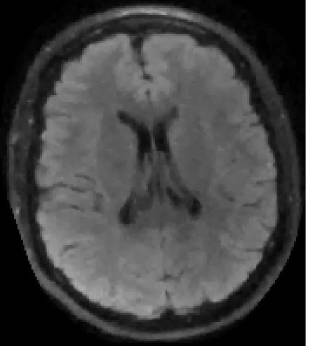}};
    \node at (49,3) {\color{red} $29.26$dB};
    
    \node[inner sep=0pt, anchor = west] (QN_3) at (QN_2.east) {\includegraphics[width=0.18\textwidth]{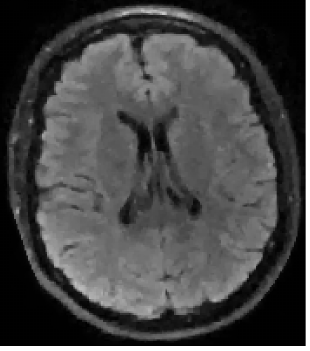}};
    \node at (88,3) {\color{red} $30.61$dB};

    \node[inner sep=0pt, anchor = west] (QN_4) at (QN_3.east) {\includegraphics[ width=0.18\textwidth]{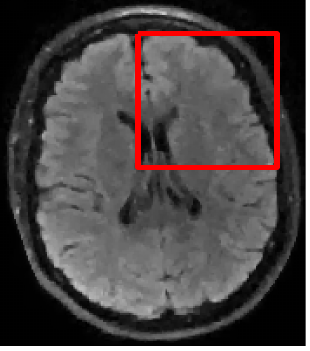}};
     \node at (125,3) {\color{red} $31.25$dB};  

    \node[inner sep=0pt, anchor = west] (CropQN_5) at (QN_4.east) {\includegraphics[width=0.15\textwidth]{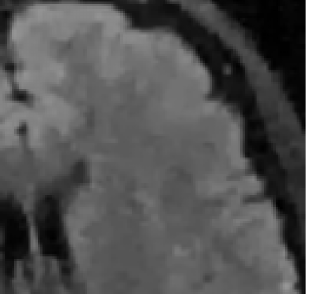}};
    
      \node at (149,37) {\Large\color{white} IV};
   \node at (185,34) {\Large\color{white} IV};
 \end{axis}

    \begin{axis}[at={(QN_1.south west)},anchor = north west,ylabel = S-CQNPM,
    xmin = 0,xmax = 216,ymin = 0,ymax = 75, width=1\textwidth,
        scale only axis,
        enlargelimits=false,
          yshift = 2.6cm,
        y label style = {yshift = -0.2cm,xshift=-1.3cm},
       axis line style={draw=none},
       tick style={draw=none},
        axis equal image,
        xticklabels={,,},yticklabels={,,}
       ]
       
   \node[inner sep=0pt, anchor = south west] (S_QN_1) at (0,0) {\includegraphics[ width=0.18\textwidth]{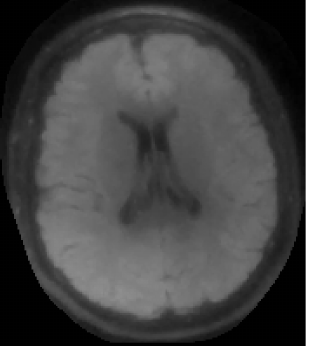}};
   
    \node at (9,3) {\color{red} $23.8$dB};

    \node[inner sep=0pt, anchor = west] (S_QN_2) at (S_QN_1.east) {\includegraphics[ width=0.18\textwidth]{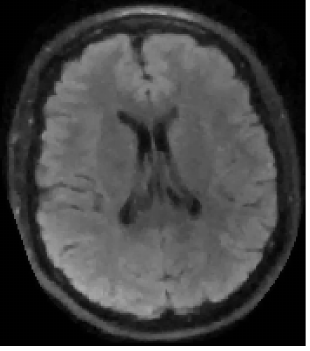}};
    \node at (49,3) {\color{red} $29.26$dB};
  
    \node[inner sep=0pt, anchor = west] (S_QN_3) at (S_QN_2.east) {\includegraphics[ width=0.18\textwidth]{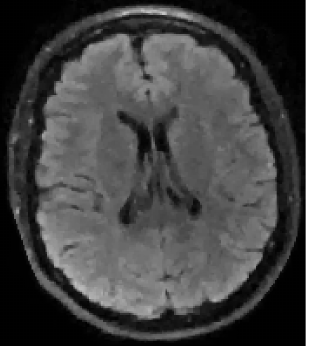}};
    
     \node at (88,3) {\color{red} $30.64$dB};
    \node[inner sep=0pt, anchor = west] (S_QN_4) at (S_QN_3.east) {\includegraphics[ width=0.18\textwidth]{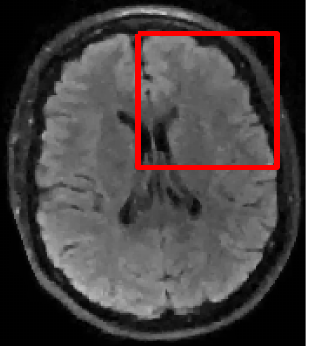}};
    \node at (125,3) {\color{red} $31.18$dB};  

       \node[inner sep=0pt, anchor = west] (CropS_QN_5) at (S_QN_4.east) {\includegraphics[ width=0.15\textwidth]{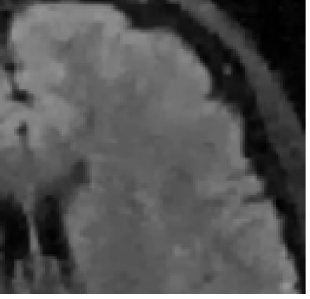}};
       
    \node at (149,37) {\Large\color{white} V};
   \node at (185,34) {\Large\color{white} V};
       
 \end{axis}
 
\end{tikzpicture} 
\caption{
 First row: the ground truth image and PSNR values versus CPU time; second to sixth row: the reconstructed brain images  at $3$, $10$, $13$, and $16$th iteration with \Cref{fig:brain:spiral:WavTV:cost} setting and the zoomed-in regions of the $16$th iteration reconstruction. We did not show the reconstructed image of ADMM since it yielded a much lower PSNR than other methods.
}
\label{fig:SpiralBrain:WavTV}
\end{figure*}

\begin{figure}[!t]
	\centering
\begin{tikzpicture}
     \begin{axis}[at={(0,0)},anchor = north west,
    xmin = 0,xmax = 216,ymin = 0,ymax = 75, width=1\textwidth,
        scale only axis,
        enlargelimits=false,
      y label style = {yshift = -0.2cm,xshift=-1.5cm},
        axis line style={draw=none},
        tick style={draw=none},
        axis equal image,
        xticklabels={,,},yticklabels={,,},
        ]
    \node[inner sep=0pt, anchor = south west] (ResiPD_1) at (0,0) {\includegraphics[ width=0.15\textwidth]{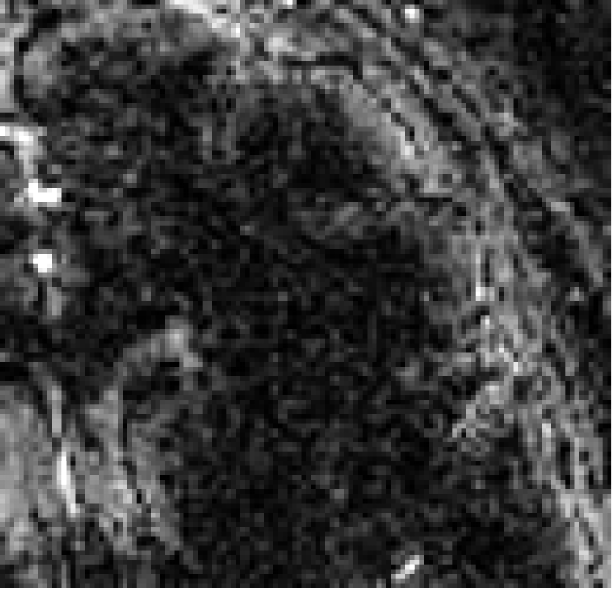}};

    \node[inner sep=0pt, anchor = west] (ResiFISTA_1) at (ResiPD_1.east) {\includegraphics[ width=0.15\textwidth]{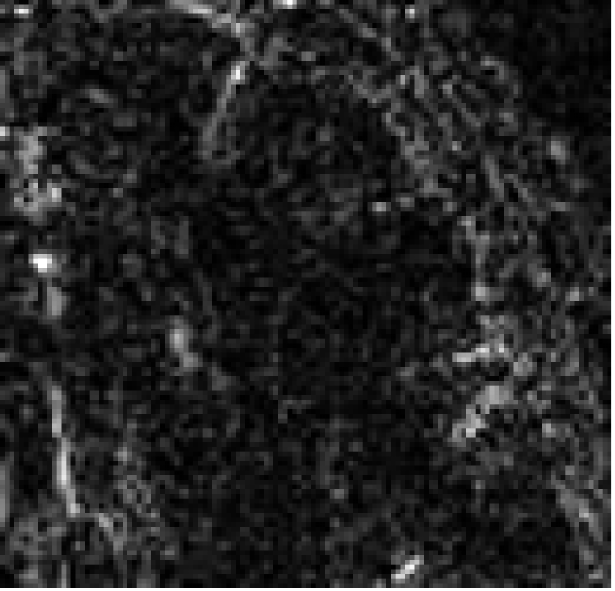}};
    
    \node[inner sep=0pt, anchor = west] (ResiSFISTA_1) at (ResiFISTA_1.east) {\includegraphics[ width=0.15\textwidth]{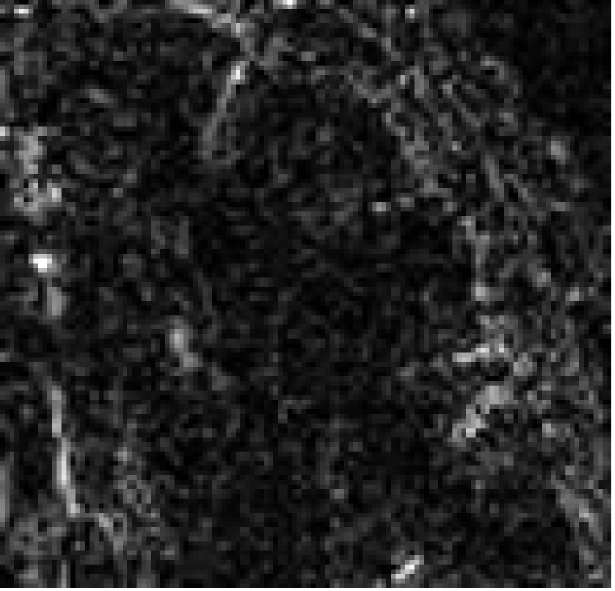}};
    
    \end{axis}

        \begin{axis}[at={(ResiPD_1.south west)},anchor = north west,
    xmin = 0,xmax = 216,ymin = 0,ymax = 75, width=1\textwidth,
        scale only axis,
        enlargelimits=false,
        yshift=3.6cm,
      y label style = {yshift = -0.2cm,xshift=-1.5cm},
        axis line style={draw=none},
        tick style={draw=none},
        axis equal image,
        xticklabels={,,},yticklabels={,,},
        ]
    
    \node[inner sep=0pt, anchor = south west] (ResiQNP_1) at (0,0) {\includegraphics[ width=0.15\textwidth]{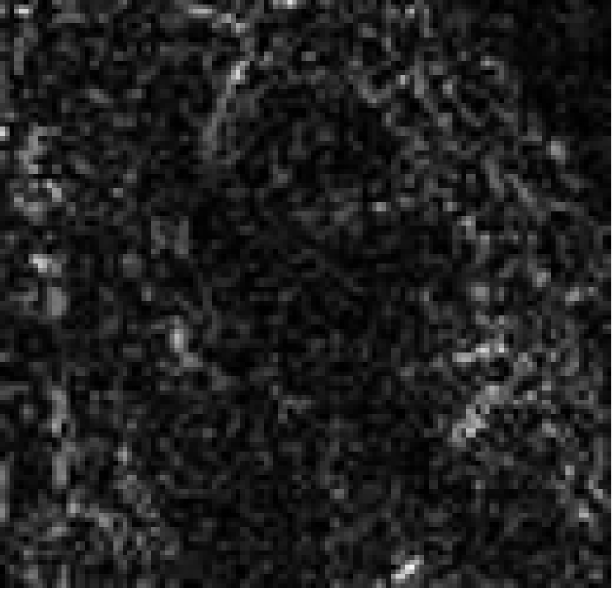}};

    \node[inner sep=0pt, anchor = west] (ResiSQNP_1) at (ResiQNP_1.east) {\includegraphics[ width=0.15\textwidth]{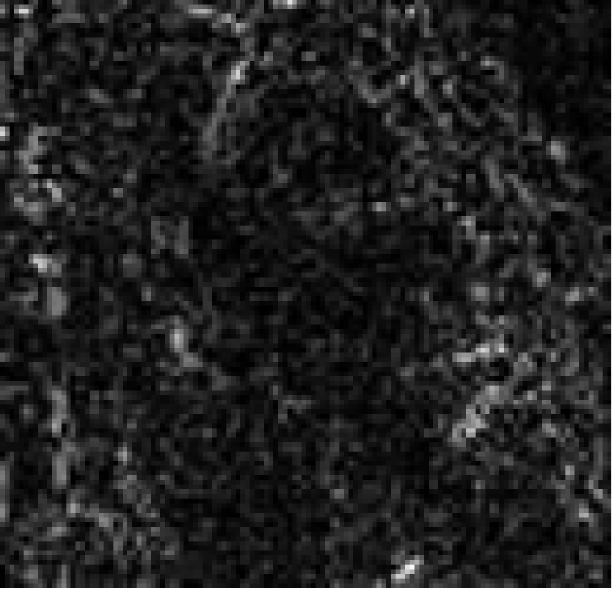}};
    \node at (72,18) {\LARGE \color{red}$\times 5$};
    \end{axis}
\end{tikzpicture} 

\caption{ {
 Error maps ($\times 5$) of the zoomed-in regions of the $16$ reconstructed images in \Cref{fig:SpiralBrain:WavTV}. From left to right, first row: PD $\rightarrow$ APM $\rightarrow$ S-APM;
 second row: CQNPM $\rightarrow$ S-CQNPM.}
}
\label{fig:SpiralBrain:TVwavelet}
\end{figure}


\begin{figure*}[!t]
	\centering
\begin{tikzpicture}
  
  \node (Knee_GT) at (-5,1.3) {\includegraphics[ width=0.22\textwidth]{fig/CropSquare_Knee_GT.pdf}};
   \node at (-6.35,2.9) {\color{white} GT};

    \node (Knee_GT_Crop) at (-1.9,1.3) {\includegraphics[ width=0.11\textwidth]{fig/Crop_Knee_GT.pdf}};

  \node at (-4.4,0.5) {\Large \color{white} (a)};
  \node at (-1.2,0.7) {\Large \color{white} (a)};
 	\pgfplotsset{every axis legend/.append style={legend pos=south east,anchor=south east,font=\normalsize, legend cell align={left}}}
    \pgfplotsset{grid style={dotted, gray}}

     \hspace{0.35cm}

 	\begin{groupplot}[enlargelimits=false,scale only axis, group style={group size=1 by 1,x descriptions at=edge bottom,group name=mygroup},
 	width=0.5*\textwidth,
 	height=0.2*\textwidth,
 	every axis/.append style={font=\small,title style={anchor=base,yshift=-1mm}, x label style={yshift = 0.5em}, y label style={yshift = -.5em}, grid = both}, legend style={at={(1,0.5)},anchor=north east,font=\tiny}
 	]
\nextgroupplot[ylabel={PSNR},xlabel={CPU Time (Seconds)},xmax = 12,ymax = 39,xtick={0,4,8,12},
xticklabels={0,4,8,12},tick align=inside,
tick pos=left,mark repeat = 5,mark size = 3pt]


  \addplot[solid,mark=star,very thick,red,line width=1pt] table [search path={fig/SpiralHighSNR_results/Knee},x =QNP_Smooth_time, y=QNP_Smooth_PSNR, col sep=comma] {Knee_WavTV_QNPSmooth.csv};
    \addplot[solid,very thick,red,line width=1pt] table [search path={fig/SpiralHighSNR_results/Knee},x =QNP_time, y=QNP_PSNR, col sep=comma] {Knee_WavTV_QNP.csv};
    \addplot[dashed,very thick,blue,line width=1pt] table [search path={fig/SpiralHighSNR_results/Knee},x =fista_time, y=fista_PSNR, col sep=comma] {Knee_WavTV_fista.csv}; 
    \addplot[dashed,mark=star,very thick,darkblue,line width=1pt] table [search path={fig/SpiralHighSNR_results/Knee},x =fista_Smooth_time, y=fista_Smooth_PSNR, col sep=comma] {Knee_WavTV_fistaSmooth.csv}; 

   
   
    \legend{S-CQNPM,CQNPM,APM,S-APM};  
 \end{groupplot}
\end{tikzpicture}

\vspace{-4cm}

\begin{tikzpicture}

    \hspace{0.8cm} 

\begin{axis}[at={(0,0)},anchor = north west, ylabel = APM,
    xmin = 0,xmax = 216,ymin = 0,ymax = 75, width=1\textwidth,
        scale only axis,
        enlargelimits=false,
        yshift=4.1cm,
      y label style = {yshift = -0.2cm,xshift=-2cm},
        axis line style={draw=none},
        tick style={draw=none},
        axis equal image,
        xticklabels={,,},yticklabels={,,},
        ]
       \node[inner sep=0pt, anchor = south west] (FISTA_1) at (0,0) {\includegraphics[ width=0.22\textwidth]{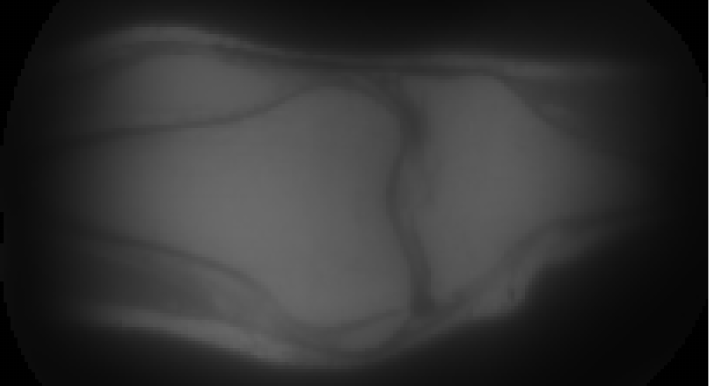}};
    
    \node at (9,3) {\color{red} $22.91$dB};
    

    \node[inner sep=0pt, anchor = west] (FISTA_2) at (FISTA_1.east) {\includegraphics[ width=0.22\textwidth]{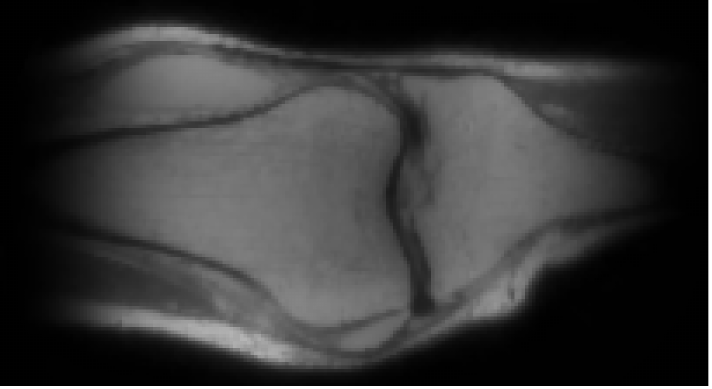}};
    
    \node at (56,3) {\color{red} $29.75$dB};
    
    \node[inner sep=0pt, anchor = west] (FISTA_3) at (FISTA_2.east) {\includegraphics[ width=0.22\textwidth]{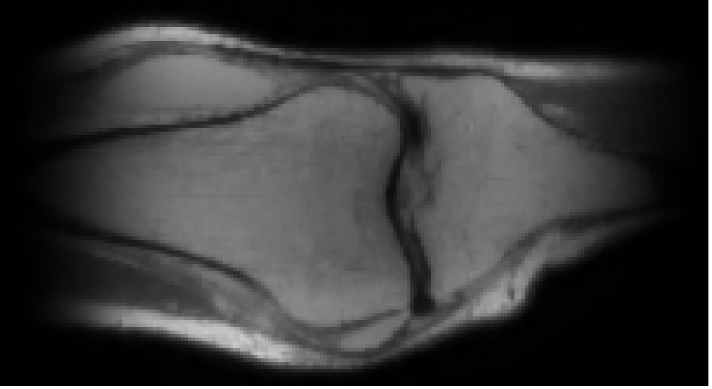}};
    
    \node at (103,3) {\color{red} $32.21$dB};

    \node[inner sep=0pt, anchor = west] (FISTA_4) at (FISTA_3.east) {\includegraphics[ width=0.22\textwidth]{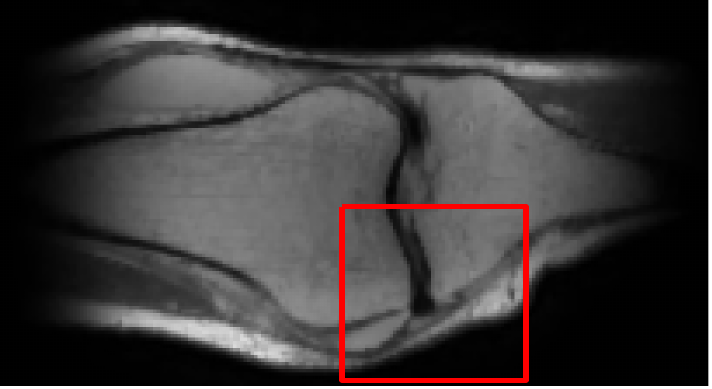}};
    \node at (152,3) {\color{red} $34.37$dB};
    \node at (175,4) {\Large \color{white} I};
    \end{axis}

\begin{axis}[at={(FISTA_1.south west)},anchor =  north west, ylabel = S-APM,
    xmin = 0,xmax = 216,ymin = 0,ymax = 75, width=1\textwidth,
        scale only axis,
        enlargelimits=false,
      yshift= 4.05cm,
        y label style = {yshift = -0.2cm,xshift=-2cm},
        axis line style={draw=none},
        tick style={draw=none},
        axis equal image,
        xticklabels={,,},yticklabels={,,},
        ]
    \node[inner sep=0pt, anchor = south west] (S_FISTA_1) at (0,0) {\includegraphics[ width=0.22\textwidth]{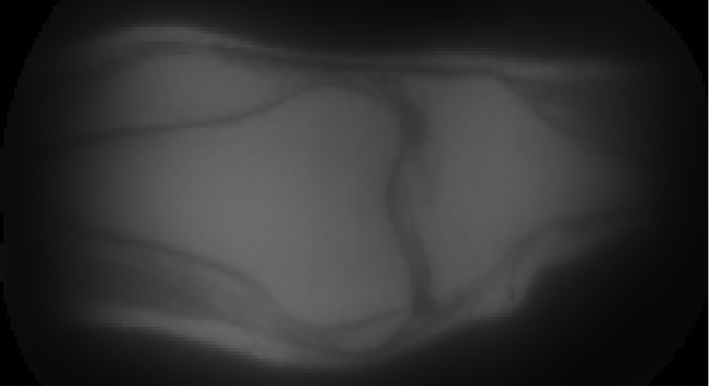}}; 
    \node at (9,3) {\color{red} $22.92$dB};
    
    \node[inner sep=0pt, anchor = west] (S_FISTA_2) at (S_FISTA_1.east) {\includegraphics[ width=0.22\textwidth]{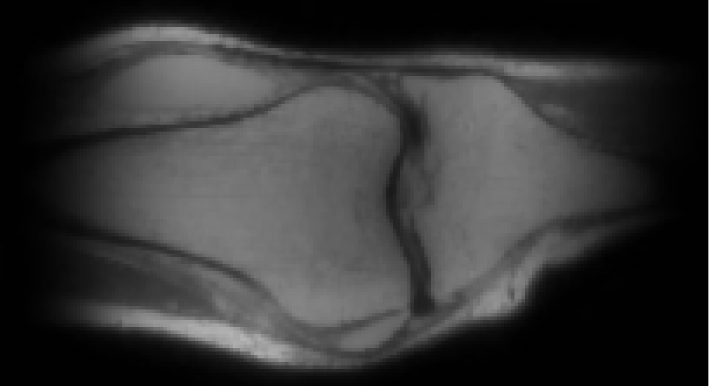}};

    \node at (56,3) {\color{red} $29.76$dB};

     \node[inner sep=0pt, anchor = west] (S_FISTA_3) at (S_FISTA_2.east) {\includegraphics[ width=0.22\textwidth]{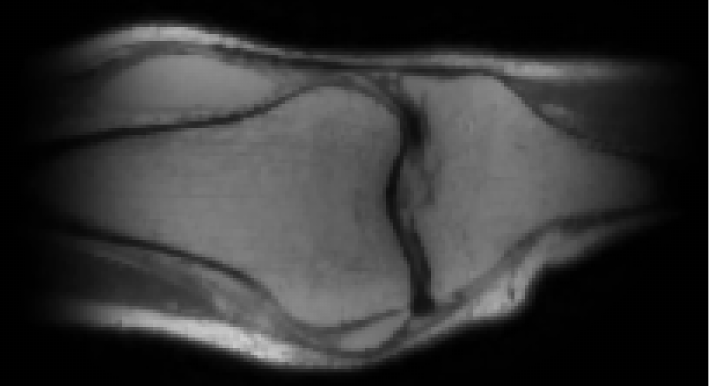}};
     \node at (103,3) {\color{red} $32.24$dB};
     
    \node[inner sep=0pt, anchor = west] (S_FISTA_4) at (S_FISTA_3.east) {\includegraphics[width=0.22\textwidth]{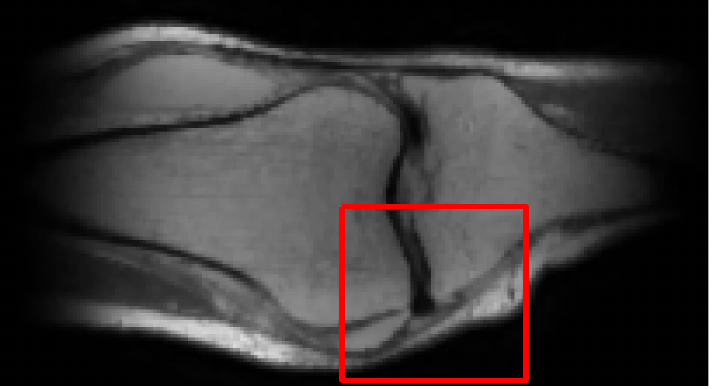}};
    \node at (152,3) {\color{red} $34.41$dB};
    
     \node at (175,4) {\Large \color{white} II};
    \end{axis}

\begin{axis}[at={(S_FISTA_1.south west)},anchor = north west,ylabel = CQNPM,
    xmin = 0,xmax = 216,ymin = 0,ymax = 75, width=1\textwidth,
        scale only axis,
        enlargelimits=false,
        yshift= 4.05cm,
        y label style = {yshift = -0.2cm,xshift=-2cm},
       axis line style={draw=none},
       tick style={draw=none},
        axis equal image,
        xticklabels={,,},yticklabels={,,}
       ]
   \node[inner sep=0pt, anchor = south west] (QN_1) at (0,0) {\includegraphics[ width=0.22\textwidth]{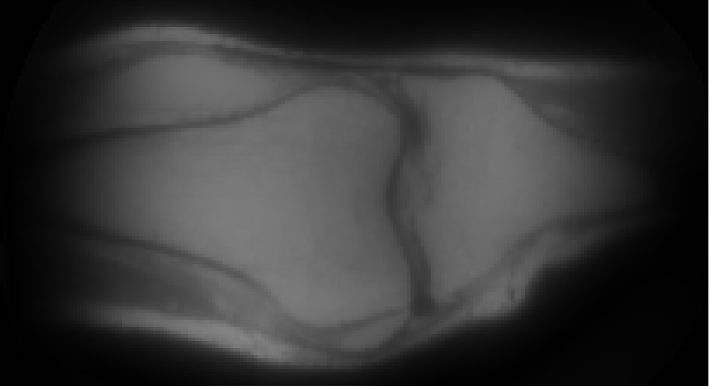}};
    \node at (9,3) {\color{red} $25.27$dB};
    
    \node[inner sep=0pt, anchor = west] (QN_2) at (QN_1.east) {\includegraphics[ width=0.22\textwidth]{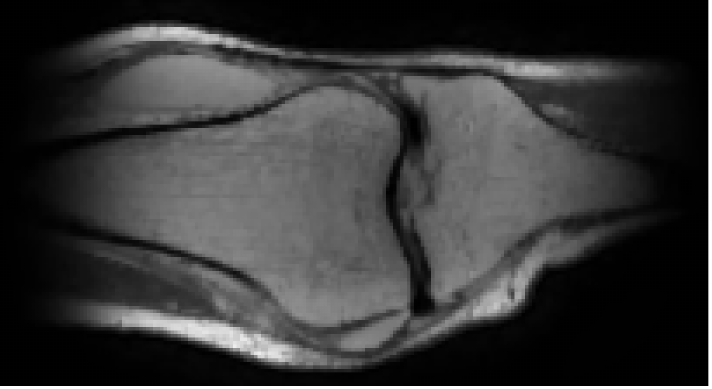}};
    \node at (56,3) {\color{red} $35.17$dB};
    
    \node[inner sep=0pt, anchor = west] (QN_3) at (QN_2.east) {\includegraphics[ width=0.22\textwidth]{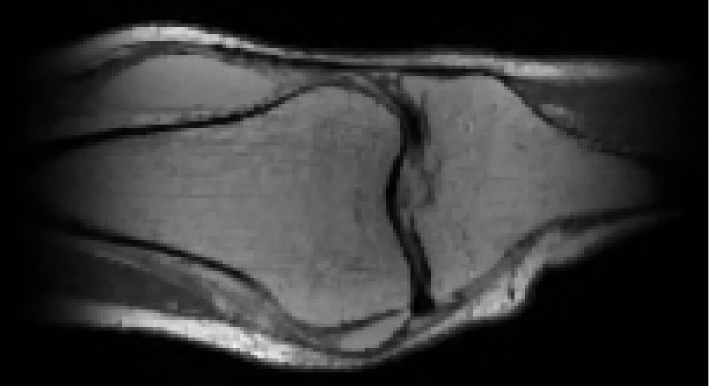}};
    \node at (103,3) {\color{red} $36.75$dB};

    \node[inner sep=0pt, anchor = west] (QN_4) at (QN_3.east) {\includegraphics[ width=0.22\textwidth]{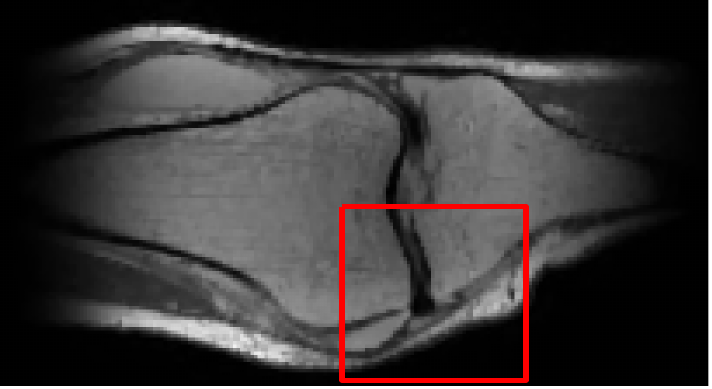}};
     \node at (152,3) {\color{red} $37.66$dB};    
     
     \node at (175,4) {\Large \color{white} III};
 \end{axis}

    \begin{axis}[at={(QN_1.south west)},anchor = north west,ylabel = S-CQNPM,
    xmin = 0,xmax = 216,ymin = 0,ymax = 75, width=1\textwidth,
        scale only axis,
        enlargelimits=false,
         yshift= 4.05cm,
        y label style = {yshift = -0.2cm,xshift=-2cm},
       axis line style={draw=none},
       tick style={draw=none},
        axis equal image,
        xticklabels={,,},yticklabels={,,}
       ]
       
   \node[inner sep=0pt, anchor = south west] (S_QN_1) at (0,0) {\includegraphics[ width=0.22\textwidth]{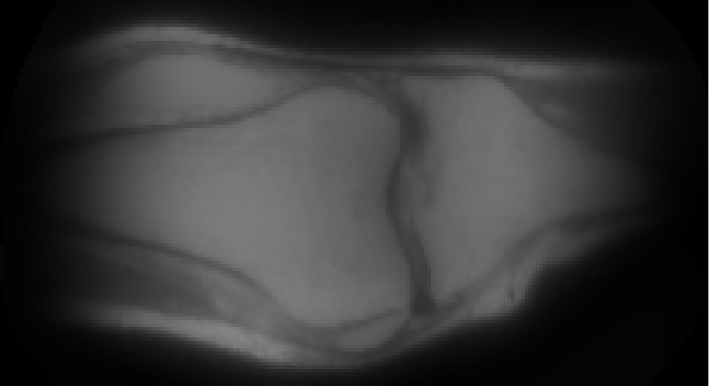}};
   
    \node at (9,3) {\color{red} $25.28$dB};

    \node[inner sep=0pt, anchor = west] (S_QN_2) at (S_QN_1.east) {\includegraphics[ width=0.22\textwidth]{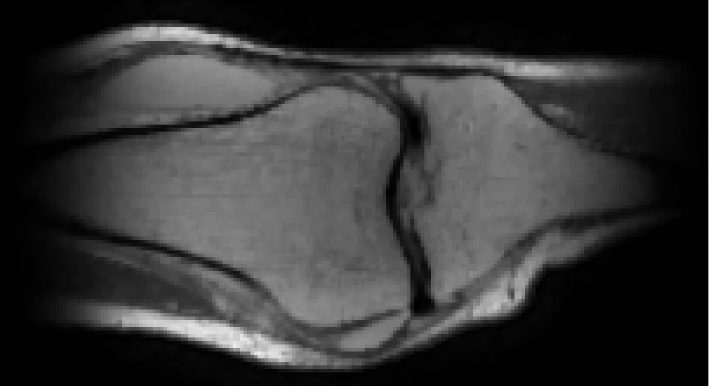}};
    \node at (56,3) {\color{red} $35.22$dB};
  
    \node[inner sep=0pt, anchor = west] (S_QN_3) at (S_QN_2.east) {\includegraphics[ width=0.22\textwidth]{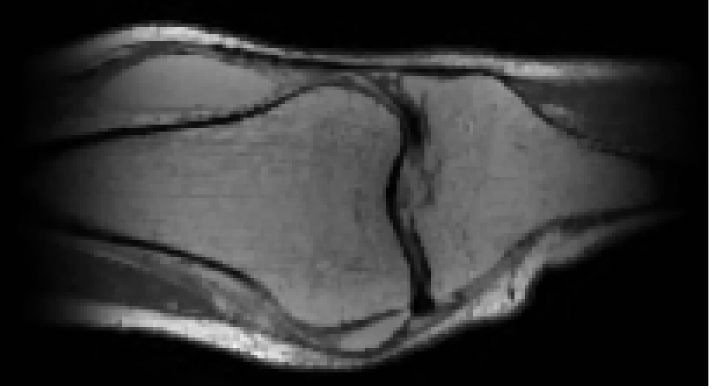}};
    
     \node at (103,3) {\color{red} $36.81$dB};
    \node[inner sep=0pt, anchor = west] (S_QN_4) at (S_QN_3.east) {\includegraphics[ width=0.22\textwidth]{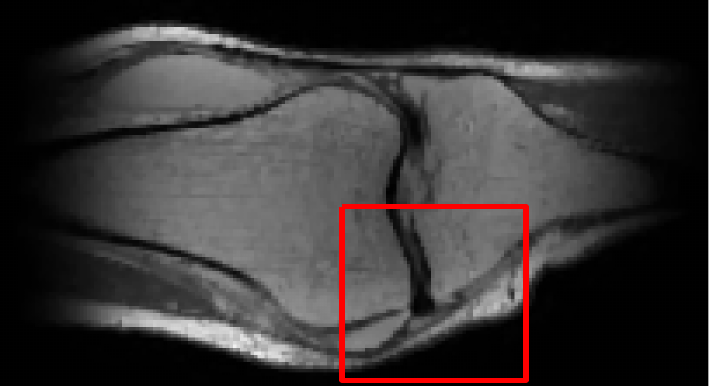}};
    \node at (152,3) {\color{red} $37.7$dB};  
   
      \node at (175,4) {\Large \color{white} IV};
 \end{axis}
 
     \begin{axis}[at={(S_QN_1.south west)},anchor = north west, 
    xmin = 0,xmax = 216,ymin = 0,ymax = 75, width=1\textwidth,
        scale only axis,
        enlargelimits=false,
        yshift=4.38cm,
      y label style = {yshift = -0.2cm,xshift=-2cm},
        axis line style={draw=none},
        tick style={draw=none},
        axis equal image,
        xticklabels={,,},yticklabels={,,},
        ]
        
    \node[inner sep=0pt, anchor = south west] (PDZoom_1) at (0,0) {\includegraphics[width=0.11\textwidth]{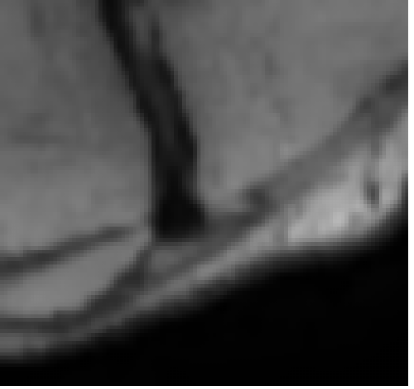}};
    
    \node[inner sep=0pt,yshift=-27pt,anchor = south west] (FISTAZoom_1) at (PDZoom_1.east) {\includegraphics[width=0.11\textwidth]{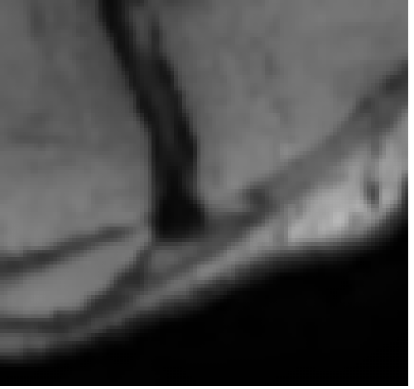}};
        
    \node[inner sep=0pt,yshift=-27pt,anchor = south west] (SFISTAZoom_1) at (FISTAZoom_1.east) {\includegraphics[width=0.11\textwidth]{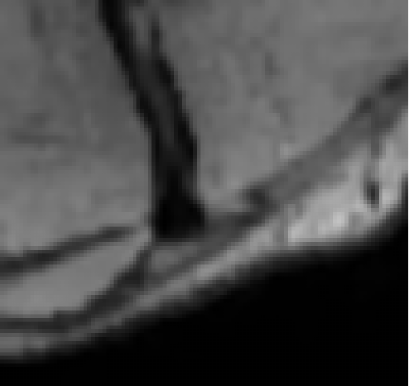}};
    
    \node[inner sep=0pt,yshift=-27pt,anchor = south west] (QNPZoom_1) at (SFISTAZoom_1.east) {\includegraphics[width=0.11\textwidth]{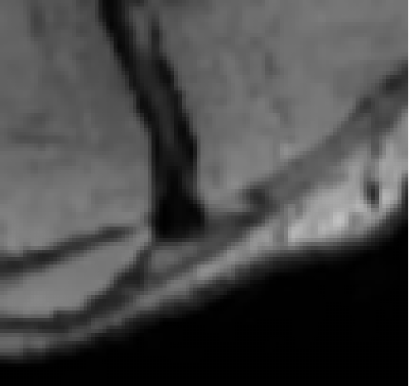}};
    
    \node at (22,3.5) {\Large \color{white} I};
    \node at (45,3.5) {\Large \color{white} II};
    \node at (68,3.5) {\Large \color{white} III};
    \node at (92,3.5) {\Large \color{white} IV};
    \end{axis}

     \begin{axis}[at={(PDZoom_1.south west)},anchor = north west, 
    xmin = 0,xmax = 216,ymin = 0,ymax = 75, width=1\textwidth,
        scale only axis,
        yshift=4.38cm,
        enlargelimits=false,
      y label style = {yshift = -0.2cm,xshift=-2cm},
        axis line style={draw=none},
        tick style={draw=none},
        axis equal image,
        xticklabels={,,},yticklabels={,,},
        ]
        
    \node[inner sep=0pt, anchor = south west] (ResiPDZoom_1) at (0,0) {\includegraphics[width=0.11\textwidth]{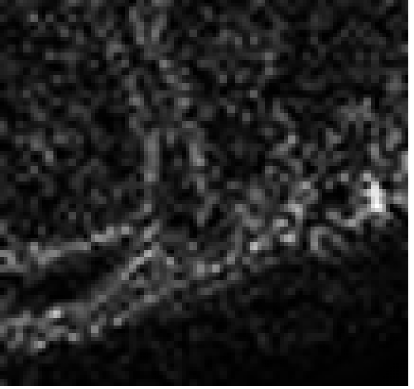}};
    
    \node[inner sep=0pt,yshift=-27pt,anchor = south west] (ResiFISTAZoom_1) at (ResiPDZoom_1.east) {\includegraphics[width=0.11\textwidth]{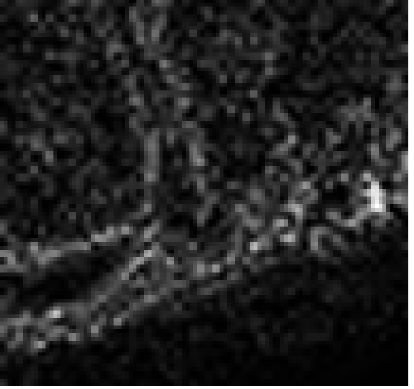}};

    \node[inner sep=0pt,yshift=-27pt,anchor = south west] (ResiSFISTAZoom_2) at (ResiFISTAZoom_1.east) {\includegraphics[width=0.11\textwidth]{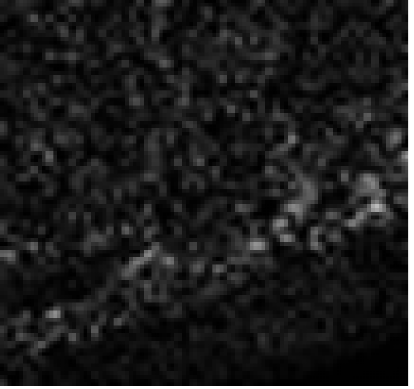}};
    
    \node[inner sep=0pt,yshift=-27pt,anchor = south west] (ResiQNPZoom_1) at (ResiSFISTAZoom_2.east)  {\includegraphics[width=0.11\textwidth]{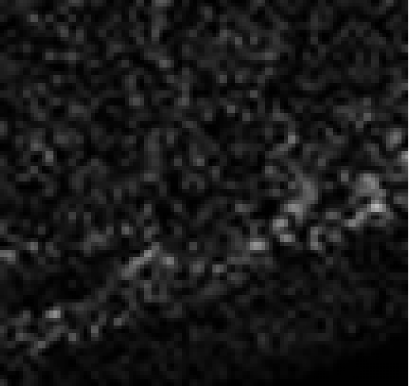}};

     \node at (105,13) {\LARGE \color{red}$\times 5$};
     
    \node at (22,3.5) {\Large \color{white} I};
    \node at (45,3.5) {\Large \color{white} II};
    \node at (68,3.5) {\Large \color{white} III};
    \node at (92,3.5) {\Large \color{white} IV};
\end{axis}

\end{tikzpicture} 

\caption{ {\Mcb Performance of different algorithms with $h(\vx)=\alpha \|\mT\vx\|_1+(1-\alpha)\mathrm{TV}(\vx)$ for the spiral acquisition and $\sim 30$dB data input SNR. The parameters $\lambda=2\times 10^{-4}$ and $\alpha=10^{-4}$. First row: the ground truth image and
PSNR values versus CPU time; second to fifth row:  the reconstructed knee images at $3$, $10$, $13$, and $16$th iteration with \Cref{fig:knee:radial:WavTV:cost} setting. The sixth and seventh rows represent the zoomed-in regions and the corresponding error mapps ($\times 5$) of the $16$th itertion reconstructed images with APM $\rightarrow$  S-APM $\rightarrow$  CQNPM $\rightarrow$  S-CQNPM.} }
\label{fig:SpiralKneeHigh:WavTV}
\end{figure*}


\begin{figure*}[!t]
	\centering
\begin{tikzpicture}

\hspace{0.8cm} 

\begin{axis}[at={(0,0)},anchor = north west, xmin = 0,xmax = 216,ymin = 0,ymax = 75, width=1\textwidth,
        scale only axis,
        enlargelimits=false,
        yshift=4.1cm,
      y label style = {yshift = -0.2cm,xshift=-2cm},
        axis line style={draw=none},
        tick style={draw=none},
        axis equal image,
        xticklabels={,,},yticklabels={,,},
        ]
       \node[inner sep=0pt, anchor = south west] (FISTA_1_temp) at (0,0) {\includegraphics[ width=0.22\textwidth]{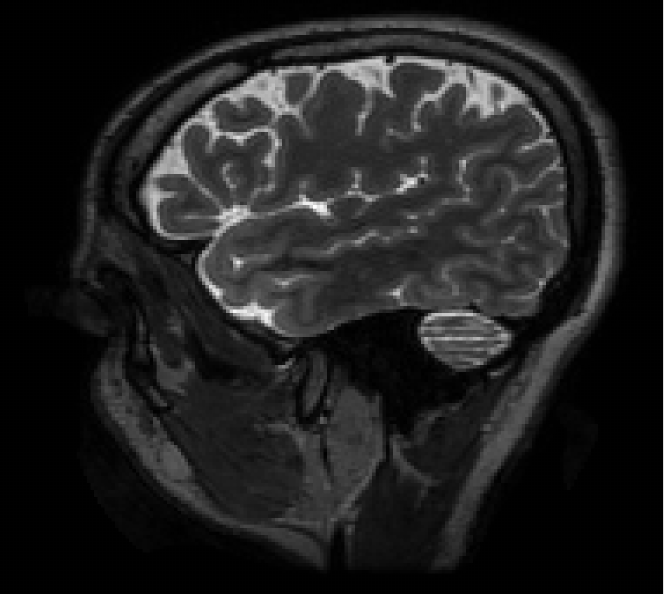}};
        \node at (3,40) {\color{white} GT};
    \end{axis}

\begin{axis}[at={(FISTA_1_temp.south west)},anchor = north west, xmin = 0,xmax = 216,ymin = 0,ymax = 75, width=1\textwidth,
        scale only axis,
        enlargelimits=false,
        yshift= 2.6cm,
        y label style = {yshift = -0.2cm,xshift=-2cm},
        axis line style={draw=none},
        tick style={draw=none},
        axis equal image,
        xticklabels={,,},yticklabels={,,},
        ]
       \node[inner sep=0pt, anchor = south west] (FISTA_1) at (0,0) {\includegraphics[ width=0.22\textwidth]{fig/SpiralHighSNRDeep_Results/DeepBrain_PPP.pdf}};
    
    \node at (12,40) {\color{white} PnP-DnCNN};
    \node at (9.5,3) {\color{red} $45.21$dB};

      \node[inner sep=0pt, anchor = west] (FISTA_2_temp) at (FISTA_1.east) {\includegraphics[ width=0.22\textwidth]{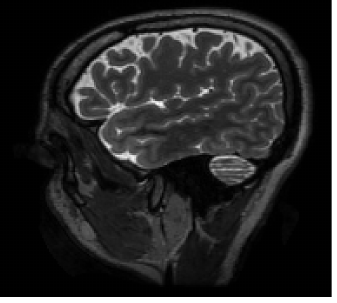}};
    \node at (58.5,40) {\color{white} PnP-BM3D};
    \node at (56,3) {\color{red} $46.8$dB};
    

    \node[inner sep=0pt, anchor = west] (FISTA_2) at (FISTA_2_temp.east) {\includegraphics[ width=0.22\textwidth]{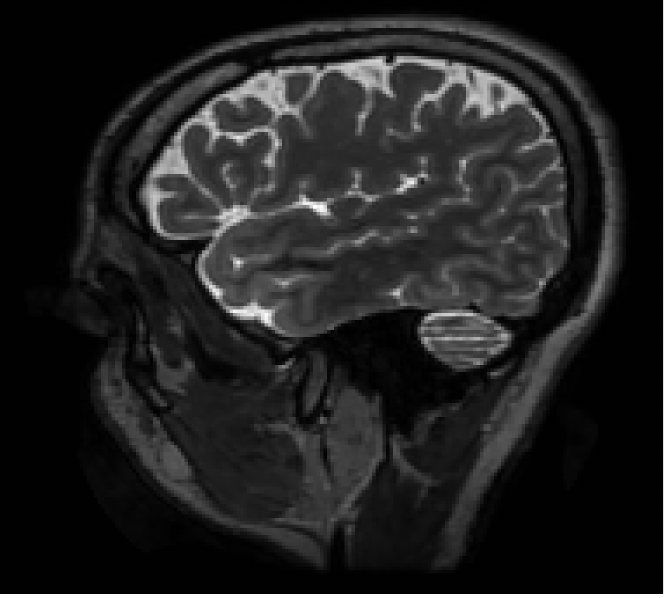}};
    \node at (102,40) {\color{white} APM};
    \node at (103,3) {\color{red} $45.34$dB};
    
    \node[inner sep=0pt, anchor = west] (FISTA_3) at (FISTA_2.east) {\includegraphics[ width=0.22\textwidth]{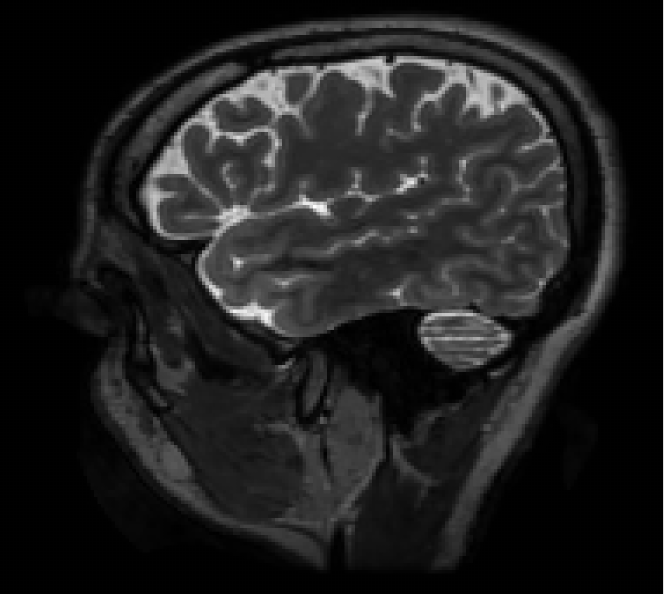}};
   \node at (152,40) {\color{white} CQNPM};
   \node at (152,3) {\color{red} $46.29$dB};


    \end{axis}

\begin{axis}[at={(FISTA_1.south west)},anchor =  north west,
    xmin = 0,xmax = 216,ymin = 0,ymax = 75, width=1\textwidth,
        scale only axis,
        enlargelimits=false,
        yshift= 2.6cm,
        y label style = {yshift = -0.2cm,xshift=-2cm},
        axis line style={draw=none},
        tick style={draw=none},
        axis equal image,
        xticklabels={,,},yticklabels={,,},
        ]
        
    \node[inner sep=0pt, anchor = south west] (S_FISTA_1) at (0,0) {\includegraphics[ width=0.22\textwidth]{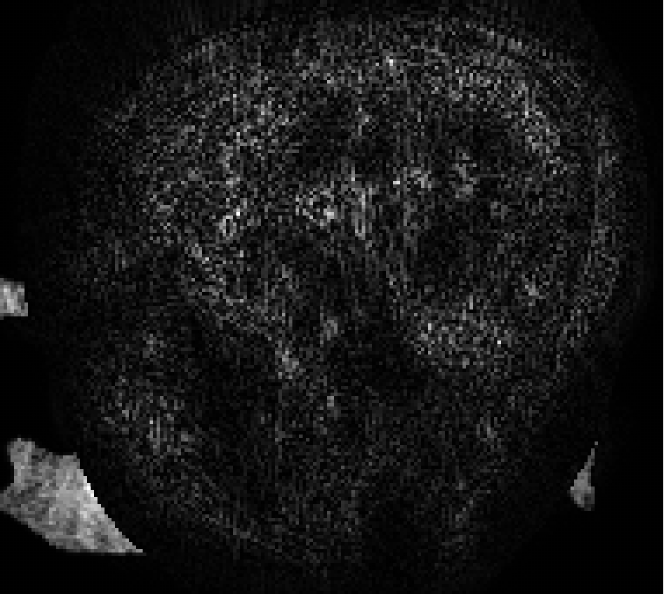}}; 

    \node[inner sep=0pt, anchor = west] (S_FISTA_2_temp) at (S_FISTA_1.east) {\includegraphics[ width=0.22\textwidth]{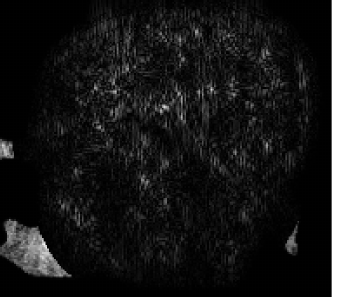}};

    \node[inner sep=0pt, anchor = west] (S_FISTA_2) at (S_FISTA_2_temp.east) {\includegraphics[ width=0.22\textwidth]{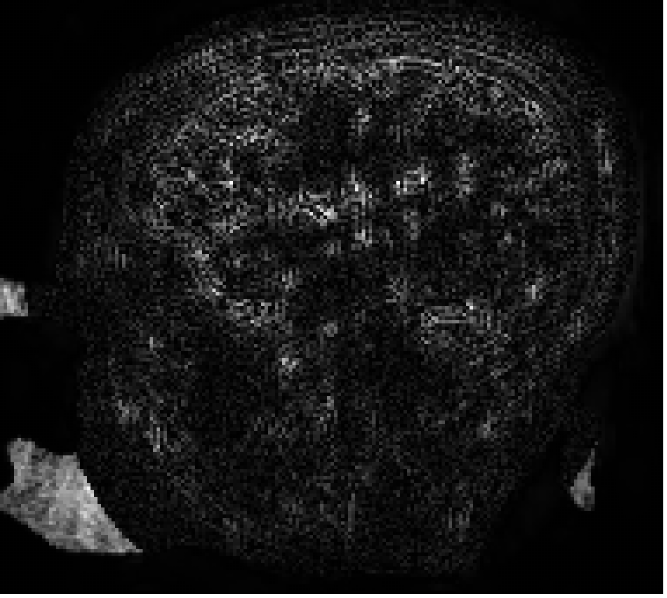}};

     \node[inner sep=0pt, anchor = west] (S_FISTA_3) at (S_FISTA_2.east) {\includegraphics[ width=0.22\textwidth]{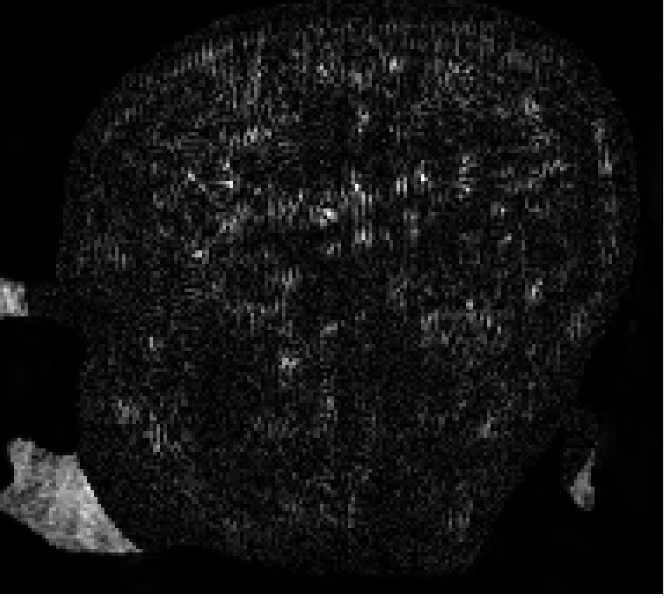}};

    \node at (198,15) {\color{red} {\LARGE $\times 20$}};
    \end{axis}

\end{tikzpicture} 

\caption{{\Mcb The reconstruction of PnP with DnCNN and BM3D denoisers, APM, and CQNPM
with the spiral acquisition and $\sim 30$dB data input SNR.
APM and CQNPM used the wavelet and total variation regularizers.}}
\label{fig:SpiralKneeHighDeep:WavTV}
\end{figure*}